\documentclass[11pt,oneside,reqno]{amsart}
\usepackage[latin9]{inputenc}
\usepackage{color}
\usepackage{mathrsfs}
\usepackage{bm}
\usepackage{amstext}
\usepackage{amsthm}
\usepackage{amssymb}
\usepackage{geometry}
\usepackage{setspace}
\usepackage[bookmarks=false,
 breaklinks=false,pdfborder={0 0 1},backref=false,colorlinks=false]
 {hyperref}

\makeatletter
\numberwithin{equation}{section}
\numberwithin{figure}{section}
\theoremstyle{plain}
\newtheorem{thm}{\protect\theoremname}[section]
\theoremstyle{remark}
\newtheorem{rem}[thm]{\protect\remarkname}

\theoremstyle{plain}
\newtheorem{prop}[thm]{\protect\propositionname}
\newtheorem*{claim*}{\protect\claimname}

\theoremstyle{definition}
\newtheorem{defn}[thm]{\protect\definitionname}
\theoremstyle{plain}
\newtheorem{lem}[thm]{\protect\lemmaname}
\theoremstyle{plain}

\theoremstyle{remark}
\newtheorem*{acknowledgement*}{\protect\acknowledgementname}

\usepackage{amsthm}
\usepackage{mathrsfs}
\usepackage[noadjust]{cite}
\usepackage{enumitem}
\setlist[itemize]{leftmargin=*}
\setlist[enumerate]{leftmargin=*}
\usepackage{tikz}
\usetikzlibrary[patterns]
\usepackage{bbm}
\usepackage{stmaryrd}

\makeatother

\providecommand{\corollaryname}{Corollary}
\providecommand{\definitionname}{Definition}
\providecommand{\lemmaname}{Lemma}
\providecommand{\remarkname}{Remark}
\providecommand{\theoremname}{Theorem}
\providecommand{\acknowledgementname}{Acknowledgement}
\providecommand{\claimname}{Claim}

\providecommand{\propositionname}{Proposition}

\newcommand{\mb}[1]{{\mathbf #1}}
\newcommand{\mf}[1]{{\mathfrak #1}}
\newcommand{\bs}[1]{{\boldsymbol #1}}
\newcommand{\bb}[1]{{\mathbb #1}}

\usepackage[notcite,notref]{showkeys}

\newcommand{\<}{\langle}
\renewcommand{\>}{\rangle}

\newcommand{\cb}[1]{{\color{blue} #1}}

\begin{document}
\title[Covergence of the Condensing SIP to coalescing BMs]{Convergence of the Condensing Symmetric Inclusion Process on the Torus in the Thermodynamical Limit to Coalescing Brownian Motions}
\author{Seonwoo Kim and Claudio Landim}
\address{S. Kim. Yonsei University, 50 Yonsei-ro, Seodaemun-gu, Seoul 03722, South Korea.}
\email{seonwookim@yonsei.ac.kr}
\address{C. Landim. IMPA, Estrada Dona Castorina 110, J. Botanico, 22460 Rio de Janeiro, Brazil and
CNRS UMR 6085, Universit\'e Rouen, Avenue de l'Universit\'e, BP.12, Technop\^ole du Madrillet, F76801 Saint-\'Etienne-du-Rouvray, France.}
\email{landim@impa.br}
\begin{abstract}
We investigate the saturation regime of the condensing symmetric inclusion
process on the discrete one-dimensional torus in the thermodynamical
limit. In this regime, the total mass concentrates on a finite number
of sites, forming condensates. 

Our main result establishes that, under appropriate scaling, the
positions of the condensates converge to a system of coalescing
Brownian motions on the continuum torus. In particular, condensates
perform diffusive motion until they meet, at which point they merge
and their masses coagulate. This provides a rigorous derivation of a
macroscopic coalescing diffusion from an underlying interacting
particle system with condensation.

The main technical difficulty arises from the complicated coalescence mechanism of two condensates of particles, whose trajectories are very difficult to track completely. The key idea is to control the coalescing time instead and prove that it is negligible compared to the time-scale of condensate movement. By combining this with precise estimates of movements without coalescence, we can prove its convergence to coalescing Brownian motions.
\end{abstract}	

\maketitle
\tableofcontents{}

\section{\label{sec1}Introduction and Main Results}

\subsection{\label{sec1.1}Condensing Symmetric Inclusion Process}

Consider the discrete one-dimensional torus
${\color{blue} \mathbb{T}_{L}} := \mathbb Z / ( L \mathbb Z )$ where $L$ is a positive integer. The $N$-particle configuration space on $\mathbb{T}_{L}$
is defined by
\[
{\color{blue} \Omega_{N}}:=\Omega_{N,L}:=\left\{ \eta\in\mathbb{N}_0^{\mathbb{T}_{L}} :
|\eta|:=\sum_{x\in\mathbb{T}_{L}}\eta_{x}=N \right\}.
\]
Here, ${\color{blue} \mathbb N_0} := \{ 0 \} \cup \mathbb N$ where
${\color{blue} \mathbb N} := \{ 1,2,3, \dots \}$. We are interested in
the regime in which both $N$ and $L$ diverge in such a way that the
particle density $N/L$ converges to a positive limit $\rho$,
corresponding to the thermodynamic limit:
\begin{equation}
N\to\infty,\qquad L=L(N)\to\infty
\qquad\text{such that}\qquad
\cb{\rho} \,:=\, \lim_{N\to\infty} \frac{N}{L(N)} >0.
\label{eq:NL-rho}
\end{equation}
The inclusion process $\color{blue} \{ \eta_N(t) \}_{t \ge 0}$ is a continuous-time
Markov process on $\Omega_{N}$ with generator
\begin{equation}
{\color{blue}\mathcal{L}_{N}}f(\eta)
=\sum_{x\in\mathbb{T}_{L}}\sum_{y\in\{x+1,x-1\}}
\theta_N \eta_{x}\bigl(d_{N}+\eta_{y}\bigr)
\bigl(f(\eta-\delta^{x}+\delta^{y})-f(\eta)\bigr).
\label{eq:inc-gen}
\end{equation}
Here, $\delta^{z}$, $z\in \mathbb{T}_{L}$, represents the
configuration with a single particle at $z$, and ${\color{blue} d_{N}}>0$ is a parameter
controlling the diffusivity of the system. 
In addition, the time-scale $\theta_N$ multiplying the jump rates is defined as
\begin{equation}
{\color{blue} \theta_{N}}:=\frac{N^{2}}{d_{N}}.\label{eq:thetaN}
\end{equation}
We assume that the process
lies in the \emph{condensing regime} (see Remark~\ref{rem:cond-reg} for a literal reasoning),
namely that $d_{N}$ decays sufficiently fast as $N\to\infty$:
\begin{equation}
\lim_{N\to\infty} d_{N} N^{3} \log N = 0.
\label{eq:dN-cond}
\end{equation}
We denote by ${\color{blue} r_{N}}(\cdot,\cdot)$ the corresponding transition rates
of the inclusion process.
Denote by $\color{blue}\mathbb{P}_{\eta}^{N}$ the law on the c\`adl\`ag space
$D([0,\infty);\Omega_{N})$ of the process starting from
$\eta\in\Omega_{N}$.

The inclusion process was originally introduced as a dual process of
Brownian energy diffusions \cite{GKR07,GKRV09}, its name is due to an
algebraic similarity to the well-known classical exclusion process. It
was shown in \cite{GRV11} that, when $d_N$ decays as $N \to \infty$,
the system exhibits a \emph{condensation} phenomenon, i.e., above a
critical density a macroscopic fraction of the particles condensate on
a single site.

The long-time dynamical behavior of the condensate was first studied
in \cite{GRV13} for dynamics on a fixed finite graph with symmetric
jumps. The work sparked interest within the metastability community to
study the dynamical behavior under more general settings. The program
was resolved completely in a series of works \cite{BDG17, Kim21,
Kim25} for the reversible case, and partially studied in \cite{KS21}
for the most general non-reversible case. It is worth mentioning that
the non-reversible case is very difficult to study due to the absence
of an explicit formula for the stationary distribution.

It was predicted in \cite{CCG14} that, starting from a uniformly
distributed initial configuration, the symmetric inclusion process
exhibits a four step dynamical condensation:
\begin{enumerate}
\item \textit{nucleation regime}: all occupied sites become isolated in a time of order $\Theta(\log N)$;
\item \textit{coarsening regime}: the clusters merge with each other to sizes of order $\Theta(N)$ in a time of order $\Theta(\frac1{d_N})$;
\item \textit{saturation regime}: the remaining finite number of
clusters of size $\Theta(N)$ merge into a single condensate in a time
of order $\Theta(\frac{N^2}{d_N})$;
\item \textit{stationary regime}: the single condensate roams around the lattice in the same time-scale.
\end{enumerate}

The present paper aims to study the saturation regime of the symmetric
inclusion process in the thermodynamic limit $L\to\infty$,
$N/L\to \rho>0$. More precisely, we assume that the system starts from
an initial configuration with fixed $k \in \mathbb N$ number of
condensates, and prove that in the time-scale of order
$\theta_N = \frac{N^2}{d_N}$ (cf. \eqref{eq:thetaN}), the merging
mechanism of $k$ condensates is well approximated by coalescing
Brownian motions of $k$ points, each point carrying the mass
information of the corresponding initial condensate along its
trajectory. This will be mathematically formulated in the remainder of
this section. We refer to Remark \ref{rem:prop} for further
explanations.

\begin{rem}
The scale $\theta_N$ in \eqref{eq:thetaN} is obtained from the microscopic metastable time-scale $\frac1{d_N}$ (cf. \cite{GRV13,BDG17,Kim21,Kim25}) multiplied by the diffusive scaling factor $N^2$.
\end{rem}

\subsection{\label{sec1.2}Condensed Configurations and Trace Process}

As explained in the last subsection, we fix a positive integer $k \ge 1$ and
focus on initial configurations consisting of $k$ condensates. For each $\ell \in \llbracket 1,k \rrbracket$, define
\[
{\color{blue}{\bf N}_{N}^{\ell}}
:=\left\{ \bm{n}=(n_{1},\dots,n_{\ell})\in\mathbb{N}^{\ell} :
n_{1}+\cdots+n_{\ell}=N \right\}.
\]
Here, $\llbracket a,b \rrbracket := [a,b] \cap \mathbb Z$.

\begin{defn}\label{def:IL-ILhat}
An element $\bm x = (x_1,\dots,x_\ell) \in \mathbb T_L^\ell$ is \emph{ordered} (resp. \emph{weakly ordered}) if
\[
0 < x_2 -x_1 < \cdots < x_\ell -x_1 < L \qquad (\text{resp.} \quad 0 \le x_2 -x_1 \le \cdots \le x_\ell -x_1 < L ),
\]
by recognizing each element $x_i - x_1$, $i \in \llbracket 1,\ell \rrbracket$, in the representative set $\llbracket 0, L-1 \rrbracket$.
In addition, $\bm{x}\in\mathbb{T}_{L}^{\ell}$ is \emph{isolated} if
any two sites in $\bm x$ are either equal or at least distance two away from each other. For instance, if $L=5$ then $(3,3,0)$ is a weakly ordered isolated element, whereas $(1,3,0)$ is ordered but not isolated.
Denote by $\color{blue} \mathbb I_L^\ell$ (resp. $\color{blue} \widehat{\mathbb I}_L^\ell$) the set of all ordered (resp. weakly ordered) isolated elements of $\mathbb T_L^\ell$.
\end{defn}

For $\bm{x}=(x_{1},\dots,x_{\ell})\in\mathbb I_{L}^{\ell}$ and
$\bm{n}=(n_{1},\dots,n_{\ell})\in{\bf N}_{N}^{\ell}$, define a
configuration $\color{blue}\xi_{\bm{n}}^{\bm{x}}\in\Omega_{N}$ by
\begin{equation}\label{eq:xinx}
\xi_{\bm{n}}^{\bm{x}}
:=\sum_{i=1}^{\ell} n_{i}\,\delta^{x_{i}},
\end{equation}
where $n_{i}\,\delta^{x_{i}}$ denotes the configuration with $n_i$
particles at site $x_{i}$ and no particles elsewhere.
Then, define
\begin{equation}
{\color{blue}\mathcal{E}_{\bm n}^{\bm x}}
:= \{\xi_{\bm{n}}^{\bm{x}}\},
\qquad
{\color{blue}\mathcal{E}_{N}^{\ell}}
:=\bigcup_{\bm{x}\in\mathbb{I}_{L}^{\ell}}
\bigcup_{\bm{n}\in{\bf N}_{N}^{\ell}}
\mathcal E_{\bm n}^{\bm x},
\qquad
{\color{blue}\mathcal{E}_{N}}
:=\bigcup_{\ell=1}^{k}\mathcal{E}_{N}^{\ell}\,,
\label{eq:EN}
\end{equation}
where, recall, $k$ represents the initial number of condensates. For each $\xi_{\bm n}^{\bm x} \in \mathcal E_N$, we call each stack of $n_i$ particles at $x_i$ a \emph{condensate} of $\xi_{\bm n}^{\bm x}$. In words, $\mathcal E_N$ collects the condensed configurations with at most $k$ isolated condensates of particles.

Our first result shows that the process stays in $\mathcal E_N$ within any finite time window with high probability.\footnote{An event happens \emph{with high probability}, or \emph{w.h.p.}, if the probability of its occurrence converges to $1$ as $N \to \infty$.} Denote by $\color{blue} {\bf 1}_A$ or $\color{blue}{\bf 1}\{A\}$ the indicator function of set/event $A$.

\begin{thm}\label{thm1}
For any $T>0$,\footnote{In this paper, ${\rm E}^{Q}$ denotes the expectation corresponding to the law $Q$.}
\[
\lim_{N\to\infty} \sup_{\xi \in \mathcal E_N} {\rm E}^{\mathbb{P}_{\xi}^{N}}\left[\int_{0}^{T}{\bf 1}\{ \eta_N(t) \in\Omega_{N}\setminus\mathcal{E}_{N}\}\,{\rm d}t\right]=0.
\]
\end{thm}
We prove Theorem \ref{thm1} in Section \ref{sec3}.

\begin{rem}
\label{rem:cond-reg}
We refer to \eqref{eq:dN-cond} as the condensing regime for the following
reason. One readily checks that the measure $\mu_{N}$ defined in
\eqref{eq:muN-def} is the unique invariant measure of the system (up to
normalization). By \cite[Theorem~3.25]{KS21},
\[
\mu_{N} \left( \Omega_{N}\setminus\mathcal{E}_{N}^{1} \right)
\ll \mu_{N} \left( \mathcal{E}_{N}^{1} \right)
\]
if and only if $d_{N}N\log N\ll1$, where $a_{N}\ll b_{N}$ means
$\lim_{N\to\infty}a_{N}/b_{N}=0$. Thus, requiring $d_{N}$ to decay faster
than $1/(N\log N)$ is equivalent to the stability of fully condensed
configurations in $\mathcal{E}_{N}^{1}$ at stationarity. The
additional $N^{2}$ diffusive scale in \eqref{eq:dN-cond} is needed to
pass from the microscopic dynamics to the continuum limit.
\end{rem}

Consider the trace
of the process $\eta_N(\cdot)$ on $\mathcal{E}_{N}$
(cf.~\cite[Section~6.1]{BL10}), obtained by turning off the clock
outside the set $\mathcal{E}_{N}$. Formally, this is defined as follows. Denote by $\color{blue} T_N(t)$ the local time in $\mathcal E_N$ until time $t$:
\begin{equation}\label{eq:TNt}
T_N (t) := \int_0^t {\bf 1} \{ \eta_N(s) \in \mathcal E_N \} \, {\rm d}s.
\end{equation}
We define its (generalized) inverse as ${\color{blue} S_N(t)} := \sup \, \{ s \ge 0 : T_N (s) \le t \}$. Then, the trace process is defined as
\begin{equation}\label{eq:trace}
\eta_N^{\mathcal E_N} (t) := \eta_N (S_N (t)) \qquad \text{for all time} \quad t \ge 0.
\end{equation}
Denote by ${\color{blue} R_N} : \mathcal E_N \times \mathcal E_N \to [0,\infty)$ its transition rate function, and by $\color{blue} \mathbb Q_\xi^{N}$ the law of the trace process on $D([0,\infty);\mathcal E_N)$ starting from $\xi \in \mathcal E_N$.

\subsection{\label{sec1.3}Typical Trajectory of the Trace Process}

In this subsection, we describe typical trajectories of the trace process as $N \to \infty$. Suppose that we run our original inclusion process starting from a configuration $\xi_{\bm n}^{\bm x} \in \mathcal E_N$ where $\bm x \in \mathbb I_L^\ell$ and $\bm n \in {\bf N}_N^\ell$ for some $\ell \in \llbracket 1,k \rrbracket$. There are two types of possible jumps from $\xi_{\bm n}^{\bm x}$.
\begin{itemize}
\item \textbf{(Type A)} First, suppose that a particle jump $x_i \to x_i+1$ occurs where $x_{i+1} \ne x_i+2$. According to \eqref{eq:inc-gen} and \eqref{eq:dN-cond}, as long as there exists a neighboring pair of occupied sites, it is unlikely to observe another particle jump to an empty site. Thus after the first jump from $x_i$ to $x_i+1$, only jumps of the $n_i$ particles between $x_i$ and $x_i+1$ are expected until either one of them becomes empty. This heuristic indicates that the first return to $\mathcal E_N$ should typically be either to $\xi_{\bm n}^{\bm x}$ or $\xi_{\bm n}^{\bm x + \bm e_i} = \xi_{\bm n}^{\bm x} - n_i \delta^{x_i} + n_i \delta^{x_i+1}$. Here, $\color{blue} \bm e_i$ denotes the unit vector on the $i$-th coordinate.
The same logic applies if the first jump occurs from $x_i$ to $x_i-1$ provided $x_{i-1} \ne x_i -2$, in which case the first return to $\mathcal E_N$ should typically be either to $\xi_{\bm n}^{\bm x}$ or $\xi_{\bm n}^{\bm x - \bm e_i}$. From the perspective of the trace process trajectory, this means that typically the condensate at $x_i$ may move to its right or left, given that the resulting configuration is still in $\mathcal E_N$.
\item \textbf{(Type B)} Otherwise, suppose that a jump $x_i \to x_i+1$ occurs where $x_{i+1} = x_i+2$. In this case, we are in a situation where there are $n_i+n_{i+1}$ particles on three consecutive occupied sites $x_i , x_i+1 ,x_i+2$ (or just two sites $x_i+1,x_i+2$ if originally there was only one particle at $x_i$), and then using the same logic, the return to $\mathcal E_N$ should occur at a configuration $\zeta$ such that: $\zeta_y = (\xi_{\bm n}^{\bm x})_y$ for any $y \notin \{x_i ,x_i+1,x_i+2 \}$, and either
\begin{equation}\label{eq:EN-nbd-1}
\begin{cases}
(\zeta_{x_i},\zeta_{x_i+1},\zeta_{x_i+2}) = (0, n_i + n_{i+1} , 0) & \text{or} \\
(\zeta_{x_i},\zeta_{x_i+1},\zeta_{x_i+2}) = (m,0, n_i + n_{i+1} - m) & \text{for some} \enspace m \in \llbracket 0 , n_i + n_{i+1} \rrbracket.
\end{cases}
\end{equation}
A similar characterization holds as well if a jump $x_i \to x_i-1$ occurs where $x_{i-1} = x_i - 2$. In this case, the process returns to $\mathcal E_N$ at a configuration $\zeta$ such that $\zeta_y = (\xi_{\bm n}^{\bm x})_y$ for any $y \notin \{x_i ,x_i-1,x_i-2 \}$, and either
\begin{equation}\label{eq:EN-nbd-2}
\begin{cases}
(\zeta_{x_i},\zeta_{x_i-1},\zeta_{x_i-2}) = (0, n_i + n_{i-1} , 0) & \text{or} \\
(\zeta_{x_i},\zeta_{x_i-1},\zeta_{x_i-2}) = (m,0, n_i + n_{i-1} - m) & \text{for some} \enspace m \in \llbracket 0 , n_i + n_{i-1} \rrbracket.
\end{cases}
\end{equation}
From the trace process perspective, two condensates at distance $2$ may either exchange masses or merge into a singe condensate at any of the three relevant sites.
\end{itemize}

\begin{figure}
\begin{tikzpicture}[scale=0.6]
\draw[ thick,<-] (2.55,0.45)--(3.45,0.45); \draw[ thick,<-] (2.55,3)--(3.45,2.4); \draw[ thick,<-] (2.55,-2.1)--(3.34,-1.5);
\draw[ thick,->] (8.55,0.45)--(9.45,0.45); \draw[ thick,->] (8.55,2.4)--(9.45,3); \draw[ thick,->] (8.55,-1.5)--(9.45,-2.1);

\draw[] (-2.25,-1.3) rectangle (2.25,2.2);
\draw[] (0,0) ellipse (2 and 1);
\foreach \i in {0,...,15} { \draw[thick] ({2*cos(\i*22.5)},{1*sin(\i*22.5)-0.1})--({2*cos(\i*22.5)},{1*sin(\i*22.5)}); }

\foreach \i in {2} { \foreach \j in {0,...,3} { \draw[very thick,blue] ({2*cos(\i*22.5)},{1*sin(\i*22.5)-0.2})--({2*cos(\i*22.5)},{1*sin(\i*22.5)});
\fill[blue!50!white] ({2*cos(\i*22.5)},{1*sin(\i*22.5)+0.15+0.3*\j}) circle (0.15); \draw[blue] ({2*cos(\i*22.5)},{1*sin(\i*22.5)+0.15+0.3*\j}) circle (0.15); } }
\foreach \i in {6} { \foreach \j in {0,...,2} { \draw[very thick,teal] ({2*cos(\i*22.5)},{1*sin(\i*22.5)-0.2})--({2*cos(\i*22.5)},{1*sin(\i*22.5)});
\fill[teal!50!white] ({2*cos(\i*22.5)},{1*sin(\i*22.5)+0.15+0.3*\j}) circle (0.15); \draw[teal] ({2*cos(\i*22.5)},{1*sin(\i*22.5)+0.15+0.3*\j}) circle (0.15); } }
\foreach \i in {11} { \foreach \j in {0,...,4} { \draw[very thick,red] ({2*cos(\i*22.5)},{1*sin(\i*22.5)-0.2})--({2*cos(\i*22.5)},{1*sin(\i*22.5)});
\fill[red!50!white] ({2*cos(\i*22.5)},{1*sin(\i*22.5)+0.15+0.3*\j}) circle (0.15); \draw[red] ({2*cos(\i*22.5)},{1*sin(\i*22.5)+0.15+0.3*\j}) circle (0.15); } }
\foreach \i in {13} { \foreach \j in {0,...,3} { \draw[very thick,red] ({2*cos(\i*22.5)},{1*sin(\i*22.5)-0.2})--({2*cos(\i*22.5)},{1*sin(\i*22.5)});
\fill[red!50!white] ({2*cos(\i*22.5)},{1*sin(\i*22.5)+0.15+0.3*\j}) circle (0.15); \draw[red] ({2*cos(\i*22.5)},{1*sin(\i*22.5)+0.15+0.3*\j}) circle (0.15); } }

\begin{scope}[shift={(0,-4.5)}]
\draw[] (-2.25,-1.3) rectangle (2.25,2.2);
\draw[] (0,0) ellipse (2 and 1);
\foreach \i in {0,...,15} { \draw[thick] ({2*cos(\i*22.5)},{1*sin(\i*22.5)-0.1})--({2*cos(\i*22.5)},{1*sin(\i*22.5)}); }

\foreach \i in {1} { \foreach \j in {0,...,3} { \draw[very thick,blue] ({2*cos(\i*22.5)},{1*sin(\i*22.5)-0.2})--({2*cos(\i*22.5)},{1*sin(\i*22.5)});
\fill[blue!50!white] ({2*cos(\i*22.5)},{1*sin(\i*22.5)+0.15+0.3*\j}) circle (0.15); \draw[blue] ({2*cos(\i*22.5)},{1*sin(\i*22.5)+0.15+0.3*\j}) circle (0.15); } }
\foreach \i in {7} { \foreach \j in {0,...,2} { \draw[very thick,teal] ({2*cos(\i*22.5)},{1*sin(\i*22.5)-0.2})--({2*cos(\i*22.5)},{1*sin(\i*22.5)});
\fill[teal!50!white] ({2*cos(\i*22.5)},{1*sin(\i*22.5)+0.15+0.3*\j}) circle (0.15); \draw[teal] ({2*cos(\i*22.5)},{1*sin(\i*22.5)+0.15+0.3*\j}) circle (0.15); } }
\foreach \i in {11} { \foreach \j in {0,...,4} { \draw[very thick,red] ({2*cos(\i*22.5)},{1*sin(\i*22.5)-0.2})--({2*cos(\i*22.5)},{1*sin(\i*22.5)});
\fill[red!50!white] ({2*cos(\i*22.5)},{1*sin(\i*22.5)+0.15+0.3*\j}) circle (0.15); \draw[red] ({2*cos(\i*22.5)},{1*sin(\i*22.5)+0.15+0.3*\j}) circle (0.15); } }
\foreach \i in {13} { \foreach \j in {0,...,3} { \draw[very thick,red] ({2*cos(\i*22.5)},{1*sin(\i*22.5)-0.2})--({2*cos(\i*22.5)},{1*sin(\i*22.5)});
\fill[red!50!white] ({2*cos(\i*22.5)},{1*sin(\i*22.5)+0.15+0.3*\j}) circle (0.15); \draw[red] ({2*cos(\i*22.5)},{1*sin(\i*22.5)+0.15+0.3*\j}) circle (0.15); } }
\end{scope}

\begin{scope}[shift={(0,4.5)}]
\draw[] (-2.25,-1.3) rectangle (2.25,2.2);
\draw[] (0,0) ellipse (2 and 1);
\foreach \i in {0,...,15} { \draw[thick] ({2*cos(\i*22.5)},{1*sin(\i*22.5)-0.1})--({2*cos(\i*22.5)},{1*sin(\i*22.5)}); }

\foreach \i in {1} { \foreach \j in {0,...,3} { \draw[very thick,blue] ({2*cos(\i*22.5)},{1*sin(\i*22.5)-0.2})--({2*cos(\i*22.5)},{1*sin(\i*22.5)});
\fill[blue!50!white] ({2*cos(\i*22.5)},{1*sin(\i*22.5)+0.15+0.3*\j}) circle (0.15); \draw[blue] ({2*cos(\i*22.5)},{1*sin(\i*22.5)+0.15+0.3*\j}) circle (0.15); } }
\foreach \i in {6} { \foreach \j in {0,...,2} { \draw[very thick,teal] ({2*cos(\i*22.5)},{1*sin(\i*22.5)-0.2})--({2*cos(\i*22.5)},{1*sin(\i*22.5)});
\fill[teal!50!white] ({2*cos(\i*22.5)},{1*sin(\i*22.5)+0.15+0.3*\j}) circle (0.15); \draw[teal] ({2*cos(\i*22.5)},{1*sin(\i*22.5)+0.15+0.3*\j}) circle (0.15); } }
\foreach \i in {11} { \foreach \j in {0,...,4} { \draw[very thick,red] ({2*cos(\i*22.5)},{1*sin(\i*22.5)-0.2})--({2*cos(\i*22.5)},{1*sin(\i*22.5)});
\fill[red!50!white] ({2*cos(\i*22.5)},{1*sin(\i*22.5)+0.15+0.3*\j}) circle (0.15); \draw[red] ({2*cos(\i*22.5)},{1*sin(\i*22.5)+0.15+0.3*\j}) circle (0.15); } }
\foreach \i in {14} { \foreach \j in {0,...,3} { \draw[very thick,red] ({2*cos(\i*22.5)},{1*sin(\i*22.5)-0.2})--({2*cos(\i*22.5)},{1*sin(\i*22.5)});
\fill[red!50!white] ({2*cos(\i*22.5)},{1*sin(\i*22.5)+0.15+0.3*\j}) circle (0.15); \draw[red] ({2*cos(\i*22.5)},{1*sin(\i*22.5)+0.15+0.3*\j}) circle (0.15); } }
\end{scope}

\begin{scope}[shift={(6,0)}]
\draw[] (-2.25,-1.3) rectangle (2.25,2.2);
\draw[] (0,0) ellipse (2 and 1);
\foreach \i in {0,...,15} { \draw[thick] ({2*cos(\i*22.5)},{1*sin(\i*22.5)-0.1})--({2*cos(\i*22.5)},{1*sin(\i*22.5)}); }

\foreach \i in {1} { \foreach \j in {0,...,3} { \draw[very thick,blue] ({2*cos(\i*22.5)},{1*sin(\i*22.5)-0.2})--({2*cos(\i*22.5)},{1*sin(\i*22.5)});
\fill[blue!50!white] ({2*cos(\i*22.5)},{1*sin(\i*22.5)+0.15+0.3*\j}) circle (0.15); \draw[blue] ({2*cos(\i*22.5)},{1*sin(\i*22.5)+0.15+0.3*\j}) circle (0.15); } }
\foreach \i in {6} { \foreach \j in {0,...,2} { \draw[very thick,teal] ({2*cos(\i*22.5)},{1*sin(\i*22.5)-0.2})--({2*cos(\i*22.5)},{1*sin(\i*22.5)});
\fill[teal!50!white] ({2*cos(\i*22.5)},{1*sin(\i*22.5)+0.15+0.3*\j}) circle (0.15); \draw[teal] ({2*cos(\i*22.5)},{1*sin(\i*22.5)+0.15+0.3*\j}) circle (0.15); } }
\foreach \i in {11} { \foreach \j in {0,...,4} { \draw[very thick,red] ({2*cos(\i*22.5)},{1*sin(\i*22.5)-0.2})--({2*cos(\i*22.5)},{1*sin(\i*22.5)});
\fill[red!50!white] ({2*cos(\i*22.5)},{1*sin(\i*22.5)+0.15+0.3*\j}) circle (0.15); \draw[red] ({2*cos(\i*22.5)},{1*sin(\i*22.5)+0.15+0.3*\j}) circle (0.15); } }
\foreach \i in {13} { \foreach \j in {0,...,3} { \draw[very thick,red] ({2*cos(\i*22.5)},{1*sin(\i*22.5)-0.2})--({2*cos(\i*22.5)},{1*sin(\i*22.5)});
\fill[red!50!white] ({2*cos(\i*22.5)},{1*sin(\i*22.5)+0.15+0.3*\j}) circle (0.15); \draw[red] ({2*cos(\i*22.5)},{1*sin(\i*22.5)+0.15+0.3*\j}) circle (0.15); } }
\draw (0,-1.4) node[below]{$\xi \in \mathcal E_N$};
\end{scope}

\begin{scope}[shift={(12,-4.5)}]
\draw[] (-2.25,-1.3) rectangle (2.25,2.2);
\draw[] (0,0) ellipse (2 and 1);
\foreach \i in {0,...,15} { \draw[thick] ({2*cos(\i*22.5)},{1*sin(\i*22.5)-0.1})--({2*cos(\i*22.5)},{1*sin(\i*22.5)}); }

\foreach \i in {1} { \foreach \j in {0,...,3} { \draw[very thick,blue] ({2*cos(\i*22.5)},{1*sin(\i*22.5)-0.2})--({2*cos(\i*22.5)},{1*sin(\i*22.5)});
\fill[blue!50!white] ({2*cos(\i*22.5)},{1*sin(\i*22.5)+0.15+0.3*\j}) circle (0.15); \draw[blue] ({2*cos(\i*22.5)},{1*sin(\i*22.5)+0.15+0.3*\j}) circle (0.15); } }
\foreach \i in {6} { \foreach \j in {0,...,2} { \draw[very thick,teal] ({2*cos(\i*22.5)},{1*sin(\i*22.5)-0.2})--({2*cos(\i*22.5)},{1*sin(\i*22.5)});
\fill[teal!50!white] ({2*cos(\i*22.5)},{1*sin(\i*22.5)+0.15+0.3*\j}) circle (0.15); \draw[teal] ({2*cos(\i*22.5)},{1*sin(\i*22.5)+0.15+0.3*\j}) circle (0.15); } }
\foreach \i in {11} { \foreach \j in {0,...,8} { \draw[very thick,red] ({2*cos(\i*22.5)},{1*sin(\i*22.5)-0.2})--({2*cos(\i*22.5)},{1*sin(\i*22.5)});
\fill[red!50!white] ({2*cos(\i*22.5)},{1*sin(\i*22.5)+0.15+0.3*\j}) circle (0.15); \draw[red] ({2*cos(\i*22.5)},{1*sin(\i*22.5)+0.15+0.3*\j}) circle (0.15); } }
\end{scope}

\begin{scope}[shift={(12,0)}]
\draw[] (-2.25,-1.3) rectangle (2.25,2.2);
\draw[] (0,0) ellipse (2 and 1);
\foreach \i in {0,...,15} { \draw[thick] ({2*cos(\i*22.5)},{1*sin(\i*22.5)-0.1})--({2*cos(\i*22.5)},{1*sin(\i*22.5)}); }

\foreach \i in {1} { \foreach \j in {0,...,3} { \draw[very thick,blue] ({2*cos(\i*22.5)},{1*sin(\i*22.5)-0.2})--({2*cos(\i*22.5)},{1*sin(\i*22.5)});
\fill[blue!50!white] ({2*cos(\i*22.5)},{1*sin(\i*22.5)+0.15+0.3*\j}) circle (0.15); \draw[blue] ({2*cos(\i*22.5)},{1*sin(\i*22.5)+0.15+0.3*\j}) circle (0.15); } }
\foreach \i in {6} { \foreach \j in {0,...,2} { \draw[very thick,teal] ({2*cos(\i*22.5)},{1*sin(\i*22.5)-0.2})--({2*cos(\i*22.5)},{1*sin(\i*22.5)});
\fill[teal!50!white] ({2*cos(\i*22.5)},{1*sin(\i*22.5)+0.15+0.3*\j}) circle (0.15); \draw[teal] ({2*cos(\i*22.5)},{1*sin(\i*22.5)+0.15+0.3*\j}) circle (0.15); } }
\foreach \i in {12} { \foreach \j in {0,...,8} { \draw[very thick,red] ({2*cos(\i*22.5)},{1*sin(\i*22.5)-0.2})--({2*cos(\i*22.5)},{1*sin(\i*22.5)});
\fill[red!50!white] ({2*cos(\i*22.5)},{1*sin(\i*22.5)+0.15+0.3*\j}) circle (0.15); \draw[red] ({2*cos(\i*22.5)},{1*sin(\i*22.5)+0.15+0.3*\j}) circle (0.15); } }
\end{scope}

\begin{scope}[shift={(12,4.5)}]
\draw[] (-2.25,-1.3) rectangle (2.25,2.2);
\draw[] (0,0) ellipse (2 and 1);
\foreach \i in {0,...,15} { \draw[thick] ({2*cos(\i*22.5)},{1*sin(\i*22.5)-0.1})--({2*cos(\i*22.5)},{1*sin(\i*22.5)}); }

\foreach \i in {1} { \foreach \j in {0,...,3} { \draw[very thick,blue] ({2*cos(\i*22.5)},{1*sin(\i*22.5)-0.2})--({2*cos(\i*22.5)},{1*sin(\i*22.5)});
\fill[blue!50!white] ({2*cos(\i*22.5)},{1*sin(\i*22.5)+0.15+0.3*\j}) circle (0.15); \draw[blue] ({2*cos(\i*22.5)},{1*sin(\i*22.5)+0.15+0.3*\j}) circle (0.15); } }
\foreach \i in {6} { \foreach \j in {0,...,2} { \draw[very thick,teal] ({2*cos(\i*22.5)},{1*sin(\i*22.5)-0.2})--({2*cos(\i*22.5)},{1*sin(\i*22.5)});
\fill[teal!50!white] ({2*cos(\i*22.5)},{1*sin(\i*22.5)+0.15+0.3*\j}) circle (0.15); \draw[teal] ({2*cos(\i*22.5)},{1*sin(\i*22.5)+0.15+0.3*\j}) circle (0.15); } }
\foreach \i in {11} { \foreach \j in {0,...,1} { \draw[very thick,red] ({2*cos(\i*22.5)},{1*sin(\i*22.5)-0.2})--({2*cos(\i*22.5)},{1*sin(\i*22.5)});
\fill[red!50!white] ({2*cos(\i*22.5)},{1*sin(\i*22.5)+0.15+0.3*\j}) circle (0.15); \draw[red] ({2*cos(\i*22.5)},{1*sin(\i*22.5)+0.15+0.3*\j}) circle (0.15); } }
\foreach \i in {13} { \foreach \j in {0,...,6} { \draw[very thick,red] ({2*cos(\i*22.5)},{1*sin(\i*22.5)-0.2})--({2*cos(\i*22.5)},{1*sin(\i*22.5)});
\fill[red!50!white] ({2*cos(\i*22.5)},{1*sin(\i*22.5)+0.15+0.3*\j}) circle (0.15); \draw[red] ({2*cos(\i*22.5)},{1*sin(\i*22.5)+0.15+0.3*\j}) circle (0.15); } }
\end{scope}
\end{tikzpicture}\caption{\label{fig1.1}Example of $\xi \in \mathcal E_N$ and its neighbor configurations in $\mathcal N(\xi)$. Left configurations are obtained by type A jumps and right configurations are obtained by type B jumps.}
\end{figure}

\begin{defn}[Neighbor configurations  in $\mathcal E_N$]\label{def:typ-jump}
For each $\xi_{\bm n}^{\bm x} \in \mathcal E_N^\ell$, define ${\color{blue} \mathcal N (\xi_{\bm n}^{\bm x})} = {\color{blue} \mathcal N_{\bm n}^{\bm x}}$ as the set which contains the following elements of $\mathcal E_N$ for each $i \in \llbracket 1,\ell \rrbracket$:
\begin{itemize}
\item if $x_{i+1} \ne x_i + 2$, then add $\xi_{\bm n}^{\bm x + \bm e_i}$;
\item if $x_{i-1} \ne x_i - 2$, then add $\xi_{\bm n}^{\bm x - \bm e_i}$;
\item if $x_{i+1} = x_i + 2$, then add all $\zeta \in \mathcal E_N \setminus \{\xi_{\bm n}^{\bm x} \}$ that satisfies \eqref{eq:EN-nbd-1}; 
\item if $x_{i-1} = x_i - 2$, then add all $\zeta \in \mathcal E_N \setminus \{\xi_{\bm n}^{\bm x} \}$ that satisfies \eqref{eq:EN-nbd-2}.
\end{itemize}
See Figure \ref{fig1.1}. Recall that $R_N(\cdot,\cdot)$ denotes the rate function of the trace process. Define ${\color{blue} R_N'}: \mathcal E_N \times \mathcal E_N \to [0,\infty)$ as
\[
R_N' \left( \xi,\xi' \right) := \begin{cases} R_N ( \xi, \xi' ) & \text{if} \enspace \xi \in \mathcal E_N, \enspace \xi' \in \mathcal N(\xi), \\
0 & \text{otherwise}. \end{cases}
\]
Denote by $\color{blue} \overline{\mathbb Q}_{\xi}^{N}$ the law of the process on $\mathcal E_N$ defined via $R_N'(\cdot,\cdot)$ starting from $\xi \in \mathcal E_N$.
\end{defn}

Then within any finite time window, the two laws $\mathbb Q_\xi^N$ and $\overline{\mathbb Q}_\xi^N$ can be coupled with high probability; i.e., the trace process jumps only to its neighbors with high probability, in the following sense. Denote by $\color{blue} \tau_N$ the first (random) time such that a jump $\xi \to \zeta$ occurs in the trace process with $\zeta \notin \mathcal N (\xi)$. More precisely, define
%\[
%\mathfrak a_N(t) := \sum_{s \in (0,t]} \sum_{\xi \in \mathcal E_N } \sum_{\zeta \in \mathcal E_N \setminus \mathcal N(\xi) } {\bf 1} \left\{ \eta_N^{\mathcal E_N}(s-) = \xi , \enspace \eta_N^{\mathcal E_N}(s) = \zeta \right\}, \qquad t \ge 0,
%\]
%where the summation in $s \in (0,t]$ is well defined a.s. since each trajectory is a.s. a c\`adl\`ag path. Then, define
%\[
%\tau_N := \inf \{ t>0 : \mathfrak a_N(t) = 1 \}.
%\]
\[
\tau_N := \inf \left\{ t>0 :  \exists \, \xi \in \mathcal E_N \quad \text{such that} \quad
\eta_N^{\mathcal E_N}(t-) = \xi, \quad
\eta_N^{\mathcal E_N}(t) \in \mathcal E_N \setminus \mathcal N(\xi) \right\}.
\]
Clearly, $\tau_N < \infty$ almost surely.

\begin{thm}\label{thm2}
For each fixed $T>0$,
\[
\lim_{N \to \infty} \sup_{\xi \in \mathcal E_N} \mathbb Q_{\xi}^N ( \tau_N \le T) = 0.
\]
\end{thm}

Theorem \ref{thm2} will be proved in Section \ref{sec4}.

\subsection{\label{sec1.4}Labeled Trace Process}

To track the movement of each condensate in the trace process, we would like to label them from $1$ to $k$ such that each label follows its corresponding condensate along the process. This cannot be done in a canonical way since there exists a small probability (tending to $0$) that a series of atypical transitions occur and trigger the condensates to evolve in a non-trackable way.\footnote{For instance, suppose that $L=6$ and the process starts from $(3,0,4,0,0,0) \in \mathcal E_N$. If the process visits \[ \mathcal E_N \ni (\bm3,0,\bm4,0,0,0) \to (2,1,4,0,0,0) \to (2,1,3,1,0,0) \to (2,1,3,0,1,0) \to (\bm2,0,\bm4,0,\bm1,0) \in \mathcal E_N,\] respectively, then a new condensate emerges at the fifth site out of nowhere. Notice that the second and third jumps above happen with negligible probability.}
However, if we restrict ourselves in the situation that only the jumps to neighbors occur, as described in Definition \ref{def:typ-jump}, then we may preserve the order of labels in a way that the labels follow the correct condensates along the dynamics.

Precisely, recall $\widehat{\mathbb I}_L^k$ from Definition \ref{def:IL-ILhat} and define a labeled trace set $\color{blue} \widehat{\mathcal E}_N$ as
\begin{equation}\label{eq:ENhat}
\widehat{\mathcal E}_N := \left\{ {\color{blue} \widehat\xi_{\bm n}^{\bm x}} := (\bm x,\bm n) \in \widehat{\mathbb I}_L^k \times \mathbb N^k : n_i = n_j \quad \text{if} \enspace x_i = x_j , \quad \sum_{ i \in \mathscr A_{\bm x }} n_i = N \right\} ,
\end{equation}
where the collection ${\color{blue} \mathscr{A}_{\bm x}} \subset \llbracket 1,k \rrbracket$ is defined as any maximal collection of $i \in \llbracket 1,k \rrbracket$ such that $x_i$, $i \in \mathscr A_{\bm x}$ are all different. Any choice of $\mathscr A_{\bm x}$ does not alter the definition since $n_i = n_j$ if $x_i = x_j$. One can easily notice that $\mathbb I_L^k \times {\bf N}_N^k \subset \widehat{\mathcal E}_N$. Each element $\widehat\xi_{\bm n}^{\bm x} \in \widehat{\mathcal E}_N$ contains the information of \emph{ordered} positions and sizes of the $k$ initial condensates. If two condensates coalesce, their positions become equal and their sizes are updated to be the same merged size.

For each $\widehat\xi_{\bm n}^{\bm x} \in \widehat{\mathcal E}_N$, define ${\color{blue} \Psi_N}(\widehat\xi_{\bm n}^{\bm x}) \in \mathcal E_N$ as
\begin{equation}\label{eq:PsiN}
\Psi_N (\widehat\xi_{\bm n}^{\bm x}) := \sum_{i \in \mathscr A_{\bm x}} n_{i}\,\delta^{x_{i}} .
\end{equation}
Indeed, $\Psi_N(\widehat\xi_{\bm n}^{\bm x}) \in \mathcal E_N$ since the elements of $\widehat{\mathbb I}_N^k$ are isolated and $\sum_{i \in \mathscr A_{\bm x}} n_i = N$. In addition, if $(\bm x,\bm n) \in \mathbb I_L^k \times {\bf N}_N^k$ then clearly $\Psi_N(\widehat\xi_{\bm n}^{\bm x}) = \xi_{\bm n}^{\bm x}$ (cf. \eqref{eq:xinx}). Thus, without any confusion, we may denote $\Psi_N(\widehat\xi_{\bm n}^{\bm x})$ as $\xi_{\bm n}^{\bm x}$.

\begin{figure}
\begin{tikzpicture}[scale=0.5]
\draw[white] (-2.7,0)--(2.7,0);
\draw[] (0,0) ellipse (2 and 1);
\foreach \i in {0,...,15} { \draw[thick] ({2*cos(\i*22.5)},{1*sin(\i*22.5)-0.1})--({2*cos(\i*22.5)},{1*sin(\i*22.5)}); }

\foreach \i in {1} { \foreach \j in {0,...,3} { \draw[very thick,blue] ({2*cos(\i*22.5)},{1*sin(\i*22.5)-0.2})--({2*cos(\i*22.5)},{1*sin(\i*22.5)});
\fill[blue!50!white] ({2*cos(\i*22.5)},{1*sin(\i*22.5)+0.15+0.3*\j}) circle (0.15); \draw[blue] ({2*cos(\i*22.5)},{1*sin(\i*22.5)+0.15+0.3*\j}) circle (0.15); }
\draw ({2*cos(\i*22.5)},{1*sin(\i*22.5)-0.15}) node[below]{\tiny \color{blue} $\bm{x_4},\bm{x_5}$}; }
\foreach \i in {6} { \foreach \j in {0,...,2} { \draw[very thick,teal] ({2*cos(\i*22.5)},{1*sin(\i*22.5)-0.2})--({2*cos(\i*22.5)},{1*sin(\i*22.5)});
\fill[teal!50!white] ({2*cos(\i*22.5)},{1*sin(\i*22.5)+0.15+0.3*\j}) circle (0.15); \draw[teal] ({2*cos(\i*22.5)},{1*sin(\i*22.5)+0.15+0.3*\j}) circle (0.15); }
\draw ({2*cos(\i*22.5)},{1*sin(\i*22.5)-0.15}) node[below]{\tiny \color{teal} $\bm{x_6}$}; }
\foreach \i in {11} { \foreach \j in {0,...,4} { \draw[very thick,red] ({2*cos(\i*22.5)},{1*sin(\i*22.5)-0.2})--({2*cos(\i*22.5)},{1*sin(\i*22.5)});
\fill[red!50!white] ({2*cos(\i*22.5)},{1*sin(\i*22.5)+0.15+0.3*\j}) circle (0.15); \draw[red] ({2*cos(\i*22.5)},{1*sin(\i*22.5)+0.15+0.3*\j}) circle (0.15); }
\draw ({2*cos(\i*22.5)},{1*sin(\i*22.5)-0.15}) node[below]{\tiny \color{red} $\bm{x_1},\bm{x_2}$}; }
\foreach \i in {13} { \foreach \j in {0,...,3} { \draw[very thick,red] ({2*cos(\i*22.5)},{1*sin(\i*22.5)-0.2})--({2*cos(\i*22.5)},{1*sin(\i*22.5)});
\fill[red!50!white] ({2*cos(\i*22.5)},{1*sin(\i*22.5)+0.15+0.3*\j}) circle (0.15); \draw[red] ({2*cos(\i*22.5)},{1*sin(\i*22.5)+0.15+0.3*\j}) circle (0.15); }
\draw ({2*cos(\i*22.5)},{1*sin(\i*22.5)-0.15}) node[below]{\tiny \color{red} $\bm{x_3}$}; }

\begin{scope}[shift={(6,0)}]
\draw[white] (-2.7,0)--(2.7,0);
\draw[] (0,0) ellipse (2 and 1);
\foreach \i in {0,...,15} { \draw[thick] ({2*cos(\i*22.5)},{1*sin(\i*22.5)-0.1})--({2*cos(\i*22.5)},{1*sin(\i*22.5)}); }

\foreach \i in {0} { \foreach \j in {0,...,3} { \draw[very thick,blue] ({2*cos(\i*22.5)},{1*sin(\i*22.5)-0.2})--({2*cos(\i*22.5)},{1*sin(\i*22.5)});
\fill[blue!50!white] ({2*cos(\i*22.5)},{1*sin(\i*22.5)+0.15+0.3*\j}) circle (0.15); \draw[blue] ({2*cos(\i*22.5)},{1*sin(\i*22.5)+0.15+0.3*\j}) circle (0.15); }
\draw ({2*cos(\i*22.5)},{1*sin(\i*22.5)-0.15}) node[below]{\tiny \color{blue} $\bm{x_4},\bm{x_5}$}; }
\foreach \i in {6} { \foreach \j in {0,...,2} { \draw[very thick,teal] ({2*cos(\i*22.5)},{1*sin(\i*22.5)-0.2})--({2*cos(\i*22.5)},{1*sin(\i*22.5)});
\fill[teal!50!white] ({2*cos(\i*22.5)},{1*sin(\i*22.5)+0.15+0.3*\j}) circle (0.15); \draw[teal] ({2*cos(\i*22.5)},{1*sin(\i*22.5)+0.15+0.3*\j}) circle (0.15); }
\draw ({2*cos(\i*22.5)},{1*sin(\i*22.5)-0.15}) node[below]{\tiny \color{teal} $\bm{x_6}$}; }
\foreach \i in {11} { \foreach \j in {0,...,4} { \draw[very thick,red] ({2*cos(\i*22.5)},{1*sin(\i*22.5)-0.2})--({2*cos(\i*22.5)},{1*sin(\i*22.5)});
\fill[red!50!white] ({2*cos(\i*22.5)},{1*sin(\i*22.5)+0.15+0.3*\j}) circle (0.15); \draw[red] ({2*cos(\i*22.5)},{1*sin(\i*22.5)+0.15+0.3*\j}) circle (0.15); }
\draw ({2*cos(\i*22.5)},{1*sin(\i*22.5)-0.15}) node[below]{\tiny \color{red} $\bm{x_1},\bm{x_2}$}; }
\foreach \i in {13} { \foreach \j in {0,...,3} { \draw[very thick,red] ({2*cos(\i*22.5)},{1*sin(\i*22.5)-0.2})--({2*cos(\i*22.5)},{1*sin(\i*22.5)});
\fill[red!50!white] ({2*cos(\i*22.5)},{1*sin(\i*22.5)+0.15+0.3*\j}) circle (0.15); \draw[red] ({2*cos(\i*22.5)},{1*sin(\i*22.5)+0.15+0.3*\j}) circle (0.15); }
\draw ({2*cos(\i*22.5)},{1*sin(\i*22.5)-0.15}) node[below]{\tiny \color{red} $\bm{x_3}$}; }
\end{scope}

\begin{scope}[shift={(12,0)}]
\draw[white] (-2.7,0)--(2.7,0);
\draw[] (0,0) ellipse (2 and 1);
\foreach \i in {0,...,15} { \draw[thick] ({2*cos(\i*22.5)},{1*sin(\i*22.5)-0.1})--({2*cos(\i*22.5)},{1*sin(\i*22.5)}); }

\foreach \i in {0} { \foreach \j in {0,...,3} { \draw[very thick,blue] ({2*cos(\i*22.5)},{1*sin(\i*22.5)-0.2})--({2*cos(\i*22.5)},{1*sin(\i*22.5)});
\fill[blue!50!white] ({2*cos(\i*22.5)},{1*sin(\i*22.5)+0.15+0.3*\j}) circle (0.15); \draw[blue] ({2*cos(\i*22.5)},{1*sin(\i*22.5)+0.15+0.3*\j}) circle (0.15); }
\draw ({2*cos(\i*22.5)},{1*sin(\i*22.5)-0.15}) node[below]{\tiny \color{blue} $\bm{x_4},\bm{x_5}$}; }
\foreach \i in {6} { \foreach \j in {0,...,2} { \draw[very thick,teal] ({2*cos(\i*22.5)},{1*sin(\i*22.5)-0.2})--({2*cos(\i*22.5)},{1*sin(\i*22.5)});
\fill[teal!50!white] ({2*cos(\i*22.5)},{1*sin(\i*22.5)+0.15+0.3*\j}) circle (0.15); \draw[teal] ({2*cos(\i*22.5)},{1*sin(\i*22.5)+0.15+0.3*\j}) circle (0.15); }
\draw ({2*cos(\i*22.5)},{1*sin(\i*22.5)-0.15}) node[below]{\tiny \color{teal} $\bm{x_6}$}; }
\foreach \i in {11} { \foreach \j in {0,...,6} { \draw[very thick,red] ({2*cos(\i*22.5)},{1*sin(\i*22.5)-0.2})--({2*cos(\i*22.5)},{1*sin(\i*22.5)});
\fill[red!50!white] ({2*cos(\i*22.5)},{1*sin(\i*22.5)+0.15+0.3*\j}) circle (0.15); \draw[red] ({2*cos(\i*22.5)},{1*sin(\i*22.5)+0.15+0.3*\j}) circle (0.15); }
\draw ({2*cos(\i*22.5)},{1*sin(\i*22.5)-0.15}) node[below]{\tiny \color{red} $\bm{x_1},\bm{x_2}$}; }
\foreach \i in {13} { \foreach \j in {0,...,1} { \draw[very thick,red] ({2*cos(\i*22.5)},{1*sin(\i*22.5)-0.2})--({2*cos(\i*22.5)},{1*sin(\i*22.5)});
\fill[red!50!white] ({2*cos(\i*22.5)},{1*sin(\i*22.5)+0.15+0.3*\j}) circle (0.15); \draw[red] ({2*cos(\i*22.5)},{1*sin(\i*22.5)+0.15+0.3*\j}) circle (0.15); }
\draw ({2*cos(\i*22.5)},{1*sin(\i*22.5)-0.15}) node[below]{\tiny \color{red} $\bm{x_3}$}; }
\end{scope}

\begin{scope}[shift={(18,0)}]
\draw[white] (-2.7,0)--(2.7,0);
\draw[] (0,0) ellipse (2 and 1);
\foreach \i in {0,...,15} { \draw[thick] ({2*cos(\i*22.5)},{1*sin(\i*22.5)-0.1})--({2*cos(\i*22.5)},{1*sin(\i*22.5)}); }

\foreach \i in {0} { \foreach \j in {0,...,3} { \draw[very thick,blue] ({2*cos(\i*22.5)},{1*sin(\i*22.5)-0.2})--({2*cos(\i*22.5)},{1*sin(\i*22.5)});
\fill[blue!50!white] ({2*cos(\i*22.5)},{1*sin(\i*22.5)+0.15+0.3*\j}) circle (0.15); \draw[blue] ({2*cos(\i*22.5)},{1*sin(\i*22.5)+0.15+0.3*\j}) circle (0.15); }
\draw ({2*cos(\i*22.5)},{1*sin(\i*22.5)-0.15}) node[below]{\tiny \color{blue} $\bm{x_4},\bm{x_5}$}; }
\foreach \i in {7} { \foreach \j in {0,...,2} { \draw[very thick,teal] ({2*cos(\i*22.5)},{1*sin(\i*22.5)-0.2})--({2*cos(\i*22.5)},{1*sin(\i*22.5)});
\fill[teal!50!white] ({2*cos(\i*22.5)},{1*sin(\i*22.5)+0.15+0.3*\j}) circle (0.15); \draw[teal] ({2*cos(\i*22.5)},{1*sin(\i*22.5)+0.15+0.3*\j}) circle (0.15); }
\draw ({2*cos(\i*22.5)},{1*sin(\i*22.5)-0.15}) node[below]{\tiny \color{teal} $\bm{x_6}$}; }
\foreach \i in {11} { \foreach \j in {0,...,6} { \draw[very thick,red] ({2*cos(\i*22.5)},{1*sin(\i*22.5)-0.2})--({2*cos(\i*22.5)},{1*sin(\i*22.5)});
\fill[red!50!white] ({2*cos(\i*22.5)},{1*sin(\i*22.5)+0.15+0.3*\j}) circle (0.15); \draw[red] ({2*cos(\i*22.5)},{1*sin(\i*22.5)+0.15+0.3*\j}) circle (0.15); }
\draw ({2*cos(\i*22.5)},{1*sin(\i*22.5)-0.15}) node[below]{\tiny \color{red} $\bm{x_1},\bm{x_2}$}; }
\foreach \i in {13} { \foreach \j in {0,...,1} { \draw[very thick,red] ({2*cos(\i*22.5)},{1*sin(\i*22.5)-0.2})--({2*cos(\i*22.5)},{1*sin(\i*22.5)});
\fill[red!50!white] ({2*cos(\i*22.5)},{1*sin(\i*22.5)+0.15+0.3*\j}) circle (0.15); \draw[red] ({2*cos(\i*22.5)},{1*sin(\i*22.5)+0.15+0.3*\j}) circle (0.15); }
\draw ({2*cos(\i*22.5)},{1*sin(\i*22.5)-0.15}) node[below]{\tiny \color{red} $\bm{x_3}$}; }
\end{scope}

\begin{scope}[shift={(24,0)}]
\draw[white] (-2.7,0)--(2.7,0);
\draw[] (0,0) ellipse (2 and 1);
\foreach \i in {0,...,15} { \draw[thick] ({2*cos(\i*22.5)},{1*sin(\i*22.5)-0.1})--({2*cos(\i*22.5)},{1*sin(\i*22.5)}); }

\foreach \i in {0} { \foreach \j in {0,...,3} { \draw[very thick,blue] ({2*cos(\i*22.5)},{1*sin(\i*22.5)-0.2})--({2*cos(\i*22.5)},{1*sin(\i*22.5)});
\fill[blue!50!white] ({2*cos(\i*22.5)},{1*sin(\i*22.5)+0.15+0.3*\j}) circle (0.15); \draw[blue] ({2*cos(\i*22.5)},{1*sin(\i*22.5)+0.15+0.3*\j}) circle (0.15); }
\draw ({2*cos(\i*22.5)},{1*sin(\i*22.5)-0.15}) node[below]{\tiny \color{blue} $\bm{x_4},\bm{x_5}$}; }
\foreach \i in {7} { \foreach \j in {0,...,2} { \draw[very thick,teal] ({2*cos(\i*22.5)},{1*sin(\i*22.5)-0.2})--({2*cos(\i*22.5)},{1*sin(\i*22.5)});
\fill[teal!50!white] ({2*cos(\i*22.5)},{1*sin(\i*22.5)+0.15+0.3*\j}) circle (0.15); \draw[teal] ({2*cos(\i*22.5)},{1*sin(\i*22.5)+0.15+0.3*\j}) circle (0.15); }
\draw ({2*cos(\i*22.5)},{1*sin(\i*22.5)-0.15}) node[below]{\tiny \color{teal} $\bm{x_6}$}; }
\foreach \i in {12} { \foreach \j in {0,...,8} { \draw[very thick,red] ({2*cos(\i*22.5)},{1*sin(\i*22.5)-0.2})--({2*cos(\i*22.5)},{1*sin(\i*22.5)});
\fill[red!50!white] ({2*cos(\i*22.5)},{1*sin(\i*22.5)+0.15+0.3*\j}) circle (0.15); \draw[red] ({2*cos(\i*22.5)},{1*sin(\i*22.5)+0.15+0.3*\j}) circle (0.15); }
\draw ({2*cos(\i*22.5)},{1*sin(\i*22.5)-0.15}) node[below]{\tiny \color{red} $\bm{x_1},\bm{x_2},\bm{x_3}$}; }
\end{scope}

\begin{scope}[shift={(0,-4)}]
\draw[white] (-2.7,0)--(2.7,0);
\draw[] (0,0) ellipse (2 and 1);
\foreach \i in {0,...,15} { \draw[thick] ({2*cos(\i*22.5)},{1*sin(\i*22.5)-0.1})--({2*cos(\i*22.5)},{1*sin(\i*22.5)}); }

\foreach \i in {1} { \foreach \j in {0,...,3} { \draw[very thick,blue] ({2*cos(\i*22.5)},{1*sin(\i*22.5)-0.2})--({2*cos(\i*22.5)},{1*sin(\i*22.5)});
\fill[blue!50!white] ({2*cos(\i*22.5)},{1*sin(\i*22.5)+0.15+0.3*\j}) circle (0.15); \draw[blue] ({2*cos(\i*22.5)},{1*sin(\i*22.5)+0.15+0.3*\j}) circle (0.15); } }
\foreach \i in {6} { \foreach \j in {0,...,2} { \draw[very thick,teal] ({2*cos(\i*22.5)},{1*sin(\i*22.5)-0.2})--({2*cos(\i*22.5)},{1*sin(\i*22.5)});
\fill[teal!50!white] ({2*cos(\i*22.5)},{1*sin(\i*22.5)+0.15+0.3*\j}) circle (0.15); \draw[teal] ({2*cos(\i*22.5)},{1*sin(\i*22.5)+0.15+0.3*\j}) circle (0.15); } }
\foreach \i in {11} { \foreach \j in {0,...,4} { \draw[very thick,red] ({2*cos(\i*22.5)},{1*sin(\i*22.5)-0.2})--({2*cos(\i*22.5)},{1*sin(\i*22.5)});
\fill[red!50!white] ({2*cos(\i*22.5)},{1*sin(\i*22.5)+0.15+0.3*\j}) circle (0.15); \draw[red] ({2*cos(\i*22.5)},{1*sin(\i*22.5)+0.15+0.3*\j}) circle (0.15); } }
\foreach \i in {13} { \foreach \j in {0,...,3} { \draw[very thick,red] ({2*cos(\i*22.5)},{1*sin(\i*22.5)-0.2})--({2*cos(\i*22.5)},{1*sin(\i*22.5)});
\fill[red!50!white] ({2*cos(\i*22.5)},{1*sin(\i*22.5)+0.15+0.3*\j}) circle (0.15); \draw[red] ({2*cos(\i*22.5)},{1*sin(\i*22.5)+0.15+0.3*\j}) circle (0.15); } }
\end{scope}

\begin{scope}[shift={(6,-4)}]
\draw[white] (-2.7,0)--(2.7,0);
\draw[] (0,0) ellipse (2 and 1);
\foreach \i in {0,...,15} { \draw[thick] ({2*cos(\i*22.5)},{1*sin(\i*22.5)-0.1})--({2*cos(\i*22.5)},{1*sin(\i*22.5)}); }

\foreach \i in {0} { \foreach \j in {0,...,3} { \draw[very thick,blue] ({2*cos(\i*22.5)},{1*sin(\i*22.5)-0.2})--({2*cos(\i*22.5)},{1*sin(\i*22.5)});
\fill[blue!50!white] ({2*cos(\i*22.5)},{1*sin(\i*22.5)+0.15+0.3*\j}) circle (0.15); \draw[blue] ({2*cos(\i*22.5)},{1*sin(\i*22.5)+0.15+0.3*\j}) circle (0.15); } }
\foreach \i in {6} { \foreach \j in {0,...,2} { \draw[very thick,teal] ({2*cos(\i*22.5)},{1*sin(\i*22.5)-0.2})--({2*cos(\i*22.5)},{1*sin(\i*22.5)});
\fill[teal!50!white] ({2*cos(\i*22.5)},{1*sin(\i*22.5)+0.15+0.3*\j}) circle (0.15); \draw[teal] ({2*cos(\i*22.5)},{1*sin(\i*22.5)+0.15+0.3*\j}) circle (0.15); } }
\foreach \i in {11} { \foreach \j in {0,...,4} { \draw[very thick,red] ({2*cos(\i*22.5)},{1*sin(\i*22.5)-0.2})--({2*cos(\i*22.5)},{1*sin(\i*22.5)});
\fill[red!50!white] ({2*cos(\i*22.5)},{1*sin(\i*22.5)+0.15+0.3*\j}) circle (0.15); \draw[red] ({2*cos(\i*22.5)},{1*sin(\i*22.5)+0.15+0.3*\j}) circle (0.15); } }
\foreach \i in {13} { \foreach \j in {0,...,3} { \draw[very thick,red] ({2*cos(\i*22.5)},{1*sin(\i*22.5)-0.2})--({2*cos(\i*22.5)},{1*sin(\i*22.5)});
\fill[red!50!white] ({2*cos(\i*22.5)},{1*sin(\i*22.5)+0.15+0.3*\j}) circle (0.15); \draw[red] ({2*cos(\i*22.5)},{1*sin(\i*22.5)+0.15+0.3*\j}) circle (0.15); } }
\end{scope}

\begin{scope}[shift={(12,-4)}]
\draw[white] (-2.7,0)--(2.7,0);
\draw[] (0,0) ellipse (2 and 1);
\foreach \i in {0,...,15} { \draw[thick] ({2*cos(\i*22.5)},{1*sin(\i*22.5)-0.1})--({2*cos(\i*22.5)},{1*sin(\i*22.5)}); }

\foreach \i in {0} { \foreach \j in {0,...,3} { \draw[very thick,blue] ({2*cos(\i*22.5)},{1*sin(\i*22.5)-0.2})--({2*cos(\i*22.5)},{1*sin(\i*22.5)});
\fill[blue!50!white] ({2*cos(\i*22.5)},{1*sin(\i*22.5)+0.15+0.3*\j}) circle (0.15); \draw[blue] ({2*cos(\i*22.5)},{1*sin(\i*22.5)+0.15+0.3*\j}) circle (0.15); } }
\foreach \i in {6} { \foreach \j in {0,...,2} { \draw[very thick,teal] ({2*cos(\i*22.5)},{1*sin(\i*22.5)-0.2})--({2*cos(\i*22.5)},{1*sin(\i*22.5)});
\fill[teal!50!white] ({2*cos(\i*22.5)},{1*sin(\i*22.5)+0.15+0.3*\j}) circle (0.15); \draw[teal] ({2*cos(\i*22.5)},{1*sin(\i*22.5)+0.15+0.3*\j}) circle (0.15); } }
\foreach \i in {11} { \foreach \j in {0,...,6} { \draw[very thick,red] ({2*cos(\i*22.5)},{1*sin(\i*22.5)-0.2})--({2*cos(\i*22.5)},{1*sin(\i*22.5)});
\fill[red!50!white] ({2*cos(\i*22.5)},{1*sin(\i*22.5)+0.15+0.3*\j}) circle (0.15); \draw[red] ({2*cos(\i*22.5)},{1*sin(\i*22.5)+0.15+0.3*\j}) circle (0.15); } }
\foreach \i in {13} { \foreach \j in {0,...,1} { \draw[very thick,red] ({2*cos(\i*22.5)},{1*sin(\i*22.5)-0.2})--({2*cos(\i*22.5)},{1*sin(\i*22.5)});
\fill[red!50!white] ({2*cos(\i*22.5)},{1*sin(\i*22.5)+0.15+0.3*\j}) circle (0.15); \draw[red] ({2*cos(\i*22.5)},{1*sin(\i*22.5)+0.15+0.3*\j}) circle (0.15); } }
\end{scope}

\begin{scope}[shift={(18,-4)}]
\draw[white] (-2.7,0)--(2.7,0);
\draw[] (0,0) ellipse (2 and 1);
\foreach \i in {0,...,15} { \draw[thick] ({2*cos(\i*22.5)},{1*sin(\i*22.5)-0.1})--({2*cos(\i*22.5)},{1*sin(\i*22.5)}); }

\foreach \i in {0} { \foreach \j in {0,...,3} { \draw[very thick,blue] ({2*cos(\i*22.5)},{1*sin(\i*22.5)-0.2})--({2*cos(\i*22.5)},{1*sin(\i*22.5)});
\fill[blue!50!white] ({2*cos(\i*22.5)},{1*sin(\i*22.5)+0.15+0.3*\j}) circle (0.15); \draw[blue] ({2*cos(\i*22.5)},{1*sin(\i*22.5)+0.15+0.3*\j}) circle (0.15); } }
\foreach \i in {7} { \foreach \j in {0,...,2} { \draw[very thick,teal] ({2*cos(\i*22.5)},{1*sin(\i*22.5)-0.2})--({2*cos(\i*22.5)},{1*sin(\i*22.5)});
\fill[teal!50!white] ({2*cos(\i*22.5)},{1*sin(\i*22.5)+0.15+0.3*\j}) circle (0.15); \draw[teal] ({2*cos(\i*22.5)},{1*sin(\i*22.5)+0.15+0.3*\j}) circle (0.15); } }
\foreach \i in {11} { \foreach \j in {0,...,6} { \draw[very thick,red] ({2*cos(\i*22.5)},{1*sin(\i*22.5)-0.2})--({2*cos(\i*22.5)},{1*sin(\i*22.5)});
\fill[red!50!white] ({2*cos(\i*22.5)},{1*sin(\i*22.5)+0.15+0.3*\j}) circle (0.15); \draw[red] ({2*cos(\i*22.5)},{1*sin(\i*22.5)+0.15+0.3*\j}) circle (0.15); } }
\foreach \i in {13} { \foreach \j in {0,...,1} { \draw[very thick,red] ({2*cos(\i*22.5)},{1*sin(\i*22.5)-0.2})--({2*cos(\i*22.5)},{1*sin(\i*22.5)});
\fill[red!50!white] ({2*cos(\i*22.5)},{1*sin(\i*22.5)+0.15+0.3*\j}) circle (0.15); \draw[red] ({2*cos(\i*22.5)},{1*sin(\i*22.5)+0.15+0.3*\j}) circle (0.15); } }
\end{scope}

\begin{scope}[shift={(24,-4)}]
\draw[white] (-2.7,0)--(2.7,0);
\draw[] (0,0) ellipse (2 and 1);
\foreach \i in {0,...,15} { \draw[thick] ({2*cos(\i*22.5)},{1*sin(\i*22.5)-0.1})--({2*cos(\i*22.5)},{1*sin(\i*22.5)}); }

\foreach \i in {0} { \foreach \j in {0,...,3} { \draw[very thick,blue] ({2*cos(\i*22.5)},{1*sin(\i*22.5)-0.2})--({2*cos(\i*22.5)},{1*sin(\i*22.5)});
\fill[blue!50!white] ({2*cos(\i*22.5)},{1*sin(\i*22.5)+0.15+0.3*\j}) circle (0.15); \draw[blue] ({2*cos(\i*22.5)},{1*sin(\i*22.5)+0.15+0.3*\j}) circle (0.15); } }
\foreach \i in {7} { \foreach \j in {0,...,2} { \draw[very thick,teal] ({2*cos(\i*22.5)},{1*sin(\i*22.5)-0.2})--({2*cos(\i*22.5)},{1*sin(\i*22.5)});
\fill[teal!50!white] ({2*cos(\i*22.5)},{1*sin(\i*22.5)+0.15+0.3*\j}) circle (0.15); \draw[teal] ({2*cos(\i*22.5)},{1*sin(\i*22.5)+0.15+0.3*\j}) circle (0.15); } }
\foreach \i in {12} { \foreach \j in {0,...,8} { \draw[very thick,red] ({2*cos(\i*22.5)},{1*sin(\i*22.5)-0.2})--({2*cos(\i*22.5)},{1*sin(\i*22.5)});
\fill[red!50!white] ({2*cos(\i*22.5)},{1*sin(\i*22.5)+0.15+0.3*\j}) circle (0.15); \draw[red] ({2*cos(\i*22.5)},{1*sin(\i*22.5)+0.15+0.3*\j}) circle (0.15); } }
\end{scope}

\draw[very thick,-latex] (2.5,0)--(3.5,0); \draw (3,0.1) node[above]{\scriptsize A};
\draw[very thick,-latex] (8.5,0)--(9.5,0); \draw (9,0.1) node[above]{\scriptsize B};
\draw[very thick,-latex] (14.5,0)--(15.5,0); \draw (15,0.1) node[above]{\scriptsize A};
\draw[very thick,-latex] (20.5,0)--(21.5,0); \draw (21,0.1) node[above]{\scriptsize B};
\draw[very thick,-latex] (2.5,-4)--(3.5,-4); \draw (3,-3.9) node[above]{\scriptsize A};
\draw[very thick,-latex] (8.5,-4)--(9.5,-4); \draw (9,-3.9) node[above]{\scriptsize B};
\draw[very thick,-latex] (14.5,-4)--(15.5,-4); \draw (15,-3.9) node[above]{\scriptsize A};
\draw[very thick,-latex] (20.5,-4)--(21.5,-4); \draw (21,-3.9) node[above]{\scriptsize B};
\end{tikzpicture}\caption{\label{fig1.2}A trajectory of the labeled trace process on $\widehat{\mathcal E}_N$ (up) and its projection to $\mathcal E_N$ (down), which becomes a trajectory of the trace process restricted to neighbor jumps only.}
\end{figure}

Next, define a transition rate function ${\color{blue} \widehat R_N} : \widehat{\mathcal E}_N \times \widehat{\mathcal E}_N \to [0,\infty)$ as follows. Refer to Figure \ref{fig1.2} for an illustration.

\begin{defn}\label{def:lab-trace}
Fix a configuration $\widehat\xi_{\bm n}^{\bm x} \in \widehat{\mathcal E}_N$.

\begin{itemize}
\item \textbf{(Type A)} Suppose that $x_{i-1} \ne x_i = \cdots = x_{i+j-1} \ne x_{i+j}$. First, suppose that $x_{i+j} \ne x_i + 2$.
Let $\bm y := \bm x + (\bm e_i + \cdots + \bm e_{i+j-1})$, which belongs to $\widehat{\mathbb I}_L^k$ since $x_{i+j} \ne x_i +2$. Define
\begin{equation}\label{eq:RNhat}
\widehat R_N ( \widehat\xi_{\bm n}^{\bm x} , \widehat\xi_{\bm n}^{\bm y} ) := R_N (\xi_{\bm n}^{\bm x} , \xi_{\bm n}^{\bm y} ) .
\end{equation}
Similarly, suppose alternatively that $x_{i-1} \ne x_i -2$. Let $\bm y := \bm x - (\bm e_i + \cdots + \bm e_{i+j-1}) \in \widehat{\mathbb I}_L^k$ and define as in \eqref{eq:RNhat}.
\item \textbf{(Type B)} If $x_{i-1} \ne x_i = \cdots = x_{i+j-1}$, $x_{i+j} = \cdots = x_{i+j+j'-1} = x_i+2$ (which should be understood in modulo $L$), and $x_{i+j+j'} \ne x_{i+j}$, then define
\begin{equation}\label{eq:RNhat-2}
\widehat R_N ( \widehat\xi_{\bm n}^{\bm x} , \widehat\xi_{\bm m}^{\bm y} ) := R_N (\xi_{\bm n}^{\bm x} , \xi_{\bm m}^{\bm y} ) ,
\end{equation}
if $\widehat\xi_{\bm m}^{\bm y} \in \widehat{\mathcal E}_N$ falls into one of the following two categories:
\begin{itemize}
\item $y_i = \cdots = y_{i+j+j'-1} \in \{x_i , x_i+1 , x_i+2\}$, $m_i = \cdots = m_{i+j+j'-1} = n_i + n_{i+j}$, and $y_p = x_p$, $m_p = n_p$ for any other $p \in \llbracket 1 , k \rrbracket$;
\item $\bm y = \bm x$, $m_i = \cdots = m_{i+j-1} = m$, $m_{i+j} = \cdots = m_{i+j+j'-1} = n_i + n_{i+j} - m$ for some $m \in \llbracket 1 , n_i + n_{i+j} -1 \rrbracket$, and $m_p = n_p$ for any other $p \in \llbracket 1 , k \rrbracket$.
\end{itemize}
\item Define $\widehat R_N (\widehat\xi_{\bm n}^{\bm x} , \widehat\xi  ) : = 0$ for any other $\widehat\xi \in \widehat{\mathcal E}_N$ not described above.
\end{itemize}
\end{defn}

The transition rate $\widehat R_N(\cdot,\cdot)$ defines a continuous-time Markov chain $\color{blue} \widehat\eta_N (t)$ on $\widehat{\mathcal E}_N$, which has only one ergodic component $\{ \widehat\xi_N^x : x \in \mathbb T_L \}$ where $\widehat\xi_N^x := ((x,N))_{i \in \llbracket 1,k \rrbracket} \in \widehat{\mathcal E}_N$.
Denote by $\color{blue} \widehat{\mathcal L}_N$ the corresponding generator, and by $\color{blue} \widehat{\mathbb Q}_{\widehat\xi}^N$ its law on $D([0,\infty);\widehat{\mathcal E}_N)$ which starts from $\widehat\xi$.

Recall $\Psi_N : \widehat{\mathcal E}_N \to \mathcal E_N$ from \eqref{eq:PsiN}. With a slight abuse of notation, we may regard $\Psi_N$ also as a mapping from $D([0,\infty);\widehat{\mathcal E}_N)$ to $D([0,\infty);\mathcal E_N)$. In this regard, the pushforward measure $\widehat{\mathbb Q}_{\widehat\xi}^N \circ \Psi_N^{-1}$ is well defined on the unlabeled trajectory space $D([0,\infty);\mathcal E_N)$.
Moreover, by comparing Definitions \ref{def:typ-jump} and \ref{def:lab-trace}, it is not difficult to see that
\begin{equation}\label{eq:proj-law}
\widehat{\mathbb Q}_{\widehat\xi_{\bm n}^{\bm x}}^{N} \circ \Psi_N^{-1} = \overline{\mathbb Q}_{\xi_{\bm n}^{\bm x}}^{N} ,
\end{equation}
i.e., the trajectory in $\mathcal E_N$ generated by the law $\widehat{\mathbb Q}_{\widehat\xi_{\bm n}^{\bm x}}^N$ via projection $\Psi_N$ follows the neighbor-jump-restricted trace process law $\overline{\mathbb Q}_{\xi_{\bm n}^{\bm x}}^N$.

\subsection{\label{sec1.5}Coalescing Brownian Motions}

Our next objective is to describe the evolution of the positions of the condensates. For this purpose, we define a diffusion process representing $k$ coalescing Brownian motions on the continuum torus $\mathbb T := \mathbb R / \mathbb Z = [0,1)$, inspired from \cite[Section 2]{DLZ04}.

For a topological space $A$ and $n \in \mathbb N_0$, denote by $\color{blue} C^{n}(A)$ the
space of $n$ times continuously differentiable functions on $A$. In particular,
$C(A):=C^{0}(A)$ is the space of continuous functions on $A$.

For each $\ell\in\llbracket1,k\rrbracket$, define
\[
{ \color{blue} \mathbb{T}_{\circ}^{\ell}} :=\left\{ (u_{1},\dots,u_{\ell})\in\mathbb{T}^{\ell}:u_{i}\ne u_{j} \quad \text{for all}\enspace i\ne j\right\} .
\]
Note that $\mathbb{T}_{\circ}^{\ell}$ consists of $(\ell-1)!$ connected
open subsets of $\mathbb{T}^{\ell}$, and that the closure of $\mathbb{T}_{\circ}^{\ell}$
is exactly $\mathbb{T}^{\ell}$. Clearly, its boundary can be written as
\[
\partial\mathbb{T}_{\circ}^{\ell} = \left\{ (u_{1},\dots,u_{\ell})\in\mathbb{T}^{\ell}:u_{i}=u_{j} \quad \text{for some}\enspace i\ne j\right\} .
\]

Next, we describe the domain of the generator. Define ${\color{blue} \mathfrak D_{0}^{(1)}}:=C^{3}(\mathbb{T})$.\footnote{The reason why we choose the $C^3$ space instead of $C^2$ as in \cite[cond. (2.A)]{DLZ04} is that we need a uniform control on the third derivatives of the functions in the generator; see \eqref{eq:Taylor-3rd-used}.}
For each $\ell\ge2$, denote by ${\color{blue} \mathfrak D_{0}^{(\ell)}}$ the
set of functions ${\bf f}\in C(\mathbb{T}^{\ell})$ that satisfies
\begin{itemize}
\item ${\bf f}\in C^{3}(\mathbb{T}_{\circ}^{\ell})$;
\item all partial derivatives of ${\bf f}$ up to the third degree are extended
continuously to $\mathbb{T}^{\ell}$ such that 
\begin{equation}
\partial_{i}\partial_{j}{\bf f}=0\qquad\text{on}\quad\left\{ \bm{u}\in\partial\mathbb{T}_{\circ}^{\ell}:u_{i}=u_{j}\right\} \qquad\text{for all}\quad i\ne j.\label{eq:didj}
\end{equation}
\end{itemize}

For $1\le i<j\le\ell$, define an operator $\psi_{i,j}:C(\mathbb{T}^{\ell})\to C(\mathbb{T}^{\ell-1})$
as
\begin{equation}\label{eq:psi-ij}
\psi_{i,j}{\bf f}(u_{1},\dots,u_{\ell-1})
:={\bf f}(u_1, \dots,u_{i},\dots,u_{j-1},u_{i},u_{j},\dots,
u_{\ell-1} ),
\end{equation}
where $u_{i}$ appears twice on the right-hand side at the $i$-th
and $j$-th coordinates. For instance, if $\ell=3$, $\psi_{1,2}{\bf f}(u,v)={\bf f}(u,u,v)$,
$\psi_{1,3}{\bf f}(u,v)={\bf f}(u,v,u)$, and $\psi_{2,3}{\bf f}(u,v)={\bf f}(u,v,v)$.

Now, denote by ${\color{blue} \mathfrak D^{(k)}}$ the collection of functions
${\bf f}\in C(\mathbb{T}^{k})$ such that
\[
\psi_{i_{\ell},j_{\ell}} \circ \cdots \circ \psi_{i_{1},j_{1}}{\bf f}\in \mathfrak D_{0}^{(k-\ell)}
\]
for all $\ell\in\llbracket0,k-1\rrbracket$ and $1\le i_{m}<j_{m}\le k-m+1$
for $m\in\llbracket1,\ell\rrbracket$. The generator ${\color{blue} \mathfrak L_{\rho}}$
acting on $C(\mathbb{T}^{k})$ is defined as
\begin{equation}\label{eq:Lrho}
\mathfrak L_{\rho}{\bf f}:=\rho^{2}\Delta{\bf f}\qquad\text{for all}\quad{\bf f}\in \mathfrak D^{(k)}.
\end{equation}
Here, $\Delta{\bf f}$ on $\partial\mathbb{T}_{\circ}^{k}$ should
be understood as the continuous extension from $\mathbb{T}_{\circ}^{k}$.

The pair $(\mathfrak L_{\rho},\mathfrak{D}^{(k)})$ defines a well-posed martingale problem associated
to $k$ coalescing Brownian motions on $\mathbb T$,
as stated in the following theorem, whose proof is deferred to Appendix \ref{secA}.
\begin{thm}
\label{thm3}For any $\bm{u} \in {\mathbb T}^{k}$,
there exists a unique probability measure $\color{blue} {\bf Q}_{\bm{u}}$
on $D([0,\infty);{\mathbb T}^{k})$ such that ${\bf Q}_{\bm u} [ \omega :  \omega_0 = \bm u ] = 1$
and
\[
\left( {\bf f}(\omega_{t})+\int_{0}^{t}e^{-\lambda s}\,(\lambda-\mathfrak L_\rho){\bf f}(\omega_{s})\,{\rm d}s \right)_{t \ge 0}
\]
is a ${\bf Q}_{ \bm{u} }$-martingale for any ${\bf f}\in\mathfrak{D}^{(k)}$ and $\lambda > 0$.
\end{thm}

\subsection{\label{sec1.6}Convergence of Condensate Movements to
Coalescing Brownian Motions}

We are ready to state our final main result. Recall the set $\widehat{\mathcal{E}}_{N}$
from \eqref{eq:ENhat}. Define a projection $\Phi_{N} : \widehat{\mathcal{E}}_{N}\to{\mathbb T}^{k}$ which forgets the mass information:
\begin{equation}
\Phi_{N}(\widehat\xi_{\bm{n}}^{\bm{x}}):= \frac{ \bm x } L \in [0,1)^k = {\mathbb T}^{k}.
\label{eq:PhiN}
\end{equation}
Let us also regard, with a notational abuse, that the projection $\Phi_N$ maps $D([0,\infty);\widehat{\mathcal E}_N)$ to $D([0,\infty);\mathbb T^k)$.

\begin{thm}
\label{thm4}Let $\widehat\xi_{\bm{n}}^{\bm{x}}\in \widehat{\mathcal{E}}_{N}$ and assume that
$L^{-1}\bm x \xrightarrow{N\to\infty} \bm u \in \mathbb{T}^k$. Then, the laws $\widehat{\mathbb Q}_{\widehat\xi_{\bm{n}}^{\bm{x}}}^{N} \circ \Phi_N^{-1}$ on $D([0,\infty);{\mathbb T}^k)$ converge weakly in the Skorokhod topology to
${\bf Q}_{  \bm u  }$ as $N\to\infty$.
\end{thm}

\begin{rem}\label{rem:prop}
The main theorems explain the typical saturation mechanism of the
condensing symmetric inclusion process as follows. By Theorem
\ref{thm1}, the process stays at a condensed configuration in
$\mathcal E_N$ at almost all times. By Theorem \ref{thm2}, the
trajectory of the process undergoes only the typical jumps (of types A
or B defined in Section \ref{sec1.3}) with high probability, so the
condensates can be labeled in order by \eqref{eq:proj-law}, and each
mass information remains attached to each condensate before collision
and adds up afterwards. Finally, Theorem \ref{thm4} indicates that the
movements of the condensates are close to the process of $k$
coalescing Brownian motions, which is well posed by Theorem
\ref{thm3}. In conclusion, as $N \to \infty$, the $k$ condensates
perform \emph{coalescing Brownian motions with a coagulating mass
mechanism.}  It is worth mentioning that, at any time of the process
at a condensed configuration, the dynamical mass information of each
condensate can be recovered from the history of the saturation regime
with high probability.
\end{rem}

\begin{rem}
The case of $k=1$ was proved in \cite[Theorem 3.24]{KS21} under a suboptimal condition $d_{N}N^{4}\log N\ll1$. Our main theorems generalize this result to the case when the process
starts from two or more condensates.
Previously in \cite{KS21}, it was impossible to consider two or more condensates due to the following
two reasons. First, the coalescing mechanism between two condensates was not well understood at the moment, which is now resolved in this paper with aid from recent works \cite{KSau24a,KSau24b}.
Second, the jumping mechanism between less stable configurations with two or more condensates (in terms of the stationary distribution) was difficult to analyze via potential theory. In this paper, we successfully overcome this by applying the so-called \emph{resolvent approach} to metastability \cite{KL26,LMS25}.
\end{rem}

\begin{rem}
Here we collect some potential future directions of research. A
natural first extension would be to consider the same saturation
regime of the symmetric inclusion process on the discrete torus of
dimension $d \ge 2$. In this case the situation would become
qualitatively different, since the diffusive time order of the random
walk is still $N^2$, whereas the time order of the coalescence of two
condensates becomes $N^2 \log N$ for $d=2$ and $N^d$ for $d \ge 3$,
both strictly bigger than $N^2$. Thus, in the thermodynamic limit, we
expect the limit diffusion to be the non-coalescing $k$ independent
Brownian motions in $\mathbb T^d$.

Alternatively, one may consider the saturation regime of the
\emph{asymmetric} inclusion process on the torus. In this case, it was
predicted \cite{CCG14} that the corresponding time-scale would be
$\frac{N}{d_N}$ instead of $\theta_N = \frac{N^2}{d_N}$. The
asymmetric jumps in the thermodynamic limit in $\mathbb T$ (or
$\mathbb T^d$) are expected to be deterministic, but the saturation
regime would remain stochastic due to the random choice of the
condensate to move next. After the saturation is completed, the limit
dynamics in the stationary regime becomes deterministic, as verified
in \cite[Theorem 3.22]{KS21}.

%One needs to first address the discrepancy in \cite{CCG14,KS21} regarding the time-scale of the saturation/stationary regimes, which was predicted as $\frac{N}{d_N}$ in \cite{CCG14} and calculated as $\frac1{d_N}$ in \cite{KS21}.
\end{rem}

The analysis of the bulk part, i.e., the jumping mechanism of a condensate that is at least distance three away from the others, is relatively straightforward since it does not interact with the other condensates. Thus, the union of typical trajectories (cf. $\mathcal A_{\bm n}^{\bm x}$ in Definition \ref{def:Anx-def-1}) is essentially one-dimensional. The main technical difficulty arises from the edge part, i.e., two condensates that are exactly two units apart. In this case, the interaction between the two condensates makes the set of typical trajectories (cf. $\mathcal A_{\bm n}^{\bm x}$ in Definition \ref{def:Anx-def-2}) two-dimensional, and it is practically impossible to track all possible jumps of the condensates. Alternatively, we use the symmetry structure of the microscopic jumps (see \eqref{eq:coupling-2}) to upper bound the time it takes for coalescence and prove that it is negligible compared to the time-scale $\theta_N$ (see Lemmas \ref{lem:tube-typ} and \ref{lem:tube-typ-2}). Combining these two estimates for the bulk and edge parts proves the flatness of the resolvent solutions in Lemmas \ref{lem:res-ind-hyp} and \ref{lem:res-ind-hyp-label}, which is key to identifying the dynamics of the coalescing particle condensates.

The rest of the article is organized as follows. 
In Section \ref{sec2}, we study the local configurations near the condensed configurations in $\mathcal E_N$, via the so-called tube of typical trajectories $\mathcal A_N$. The results therein will be repeatedly exploited in the remainder.
In Section \ref{sec3}, we prove Theorem \ref{thm1} via a comparison principle of the macro/microscopic resolvent solutions.
In Section \ref{sec4}, we prove Theorem \ref{thm2}.
In Section \ref{sec5}, we present another resolvent comparison result, now suitable for characterizing the transitions between the labeled condensed configurations.
Finallly, in Section \ref{sec6}, we make use of this resolvent condition to prove Theorem \ref{thm4}.
%we prove the resolvent flatness condition, Theorem \ref{thm:res}.
%In Section \ref{sec7}, we prove the uniqueness of the martingale
%problem presented in Theorem \ref{thm3}. In Section \ref{sec6},
%we prove that the collection $\{{\bf Q}_{\xi_{\bm{n}}^{\bm{x}}}^{N}:N\ge1\}$
%in Theorem \ref{thm4} is tight and any limit point solves the
%martingale problem in Theorem \ref{thm3}. Along with the uniqueness
%from Section \ref{sec7}, this completes the proof of both Theorems
%\ref{thm3} and \ref{thm4}. \textbf{TBD: mass part?}
Appendix \ref{secA} presents some general properties of the coalescing Brownian motions
defined in Section \ref{sec1.5}, and especially, the proof of Theorem \ref{thm3}. Appendix \ref{secB} collects
simple estimates on two types of one-dimensional random walks that
are exploited in Section \ref{sec2}.

\section{\label{sec2}Tube of Typical Trajectories}

In this section, we begin the analysis of particle condensates transitions.
Recall from Definition \ref{def:IL-ILhat} and \eqref{eq:EN} the definitions of $\mathbb I_L^\ell$, $\ell \in \llbracket 1,k \rrbracket$, and $\mathcal E_N$.
Notice that if $\bm x \in \mathbb I_L^\ell$, $d \, (x_i , x_j) \ge 2$ for all $i \ne j$, where ${\color{blue} d} \, (\cdot,\cdot)$ denotes the canonical distance on $\mathbb T_L$.
Decompose $\mathbb{I}_{L}^{\ell} := \mathbb J_{L}^{\ell}\cup\mathbb K_{L}^{\ell}$
where
\[
{\color{blue} \mathbb J_{L}^{\ell}}:=\left\{ \bm{x} \in \mathbb I_{L}^{\ell} : d \, (x_i, x_j) \ge 3 \quad \text{for all} \enspace i \ne j \right\} ,\qquad
{\color{blue} \mathbb K_{L}^{\ell}}:= \mathbb{I}_{L}^{\ell} \setminus\mathbb J_{L}^{\ell}.
\]
Then, define
\begin{equation}\label{eq:JNKN}
{\color{blue} \mathcal J_N^\ell } := \bigcup_{\bm x \in \mathbb J_L^\ell} \bigcup_{\bm n \in {\bf N}_N^\ell} \{ \xi_{\bm n}^{\bm x} \}, \qquad
{\color{blue} \mathcal K_N^\ell } := \bigcup_{\bm x \in \mathbb K_L^\ell} \bigcup_{\bm n \in {\bf N}_N^\ell} \{ \xi_{\bm n}^{\bm x} \} \qquad \text{so that} \qquad
\mathcal E_N^\ell = \mathcal J_N^\ell \cup \mathcal K_N^\ell .
\end{equation}

Next, we define sets $\mathcal A_{\bm n}^{\bm x}$ for each $\xi_{\bm n}^{\bm x} \in \mathcal E_N$, dividing into two cases.

\begin{defn}[Collection $\mathcal A_{\bm n}^{\bm x}$ for $\xi_{\bm n}^{\bm x} \in \mathcal J_N^\ell$]\label{def:Anx-def-1}
For each $\xi_{\bm n}^{\bm x} \in \mathcal J_N^\ell$, denote by $\color{blue} \mathcal A_{\bm n}^{\bm x}$ the collection of all configurations that can be attained from consecutive transitions starting from $\xi_{\bm n}^{\bm x}$ such that,
if the initial particle jump occurs at $x_i \to x_i +a$, where $a \in \{1,-1\}$, then from this point only jumps between $x_i$ and $x_i +a$ are allowed. The set $\mathcal A_{\bm n}^{\bm x}$ is called a \emph{tube} from $\xi_{\bm n}^{\bm x}$ to $\mathcal E_N$, in the sense that it is defined by the initial state $\xi_{\bm n}^{\bm x}$ and the ones attained from typical jumps after the first jump. It is clear that
\[
\mathcal A_{\bm n}^{\bm x} = \bigcup_{i = 1}^\ell \left( \mathcal A_{\bm n}^{\bm x,i,+} \cup \mathcal A_{\bm n}^{\bm x,i,-} \right) ,
\]
where
\[
{\color{blue} \mathcal{A}_{\bm n}^{\bm{x},i,\pm} } := \left\{ \eta\in\Omega_{N}:\eta_{x_{i}}+\eta_{x_{i} \pm 1}=n_{i} \quad \text{and} \quad \eta_{x_{j}}=n_{j} \quad \text{for}\enspace j\ne i\right\} .
\]
The set $\mathcal{A}_{\bm n}^{\bm{x}}$ collects configurations
visited during a typical transition from $\xi_{\bm{n}}^{\bm{x}}$
to its neighbors, that are, $\xi_{\bm{n}}^{\bm{x}\pm\bm e_{i}}$
for $i\in\llbracket1,\ell\rrbracket$. Note that the configurations $\xi_{\bm n}^{\bm x \pm \bm e_i}$ are still in $\mathcal E_N^\ell$ since $\bm x \in \mathbb J_L^\ell$.
Starting from $\xi_{\bm n}^{\bm x}$ and during the time in which the process remains in $\mathcal A_{\bm n}^{\bm x}$, it returns to $\mathcal E_N$ at a configuration which belongs either to $\mathcal J_N^\ell$ or to $\mathcal K_N^\ell$.
\end{defn}

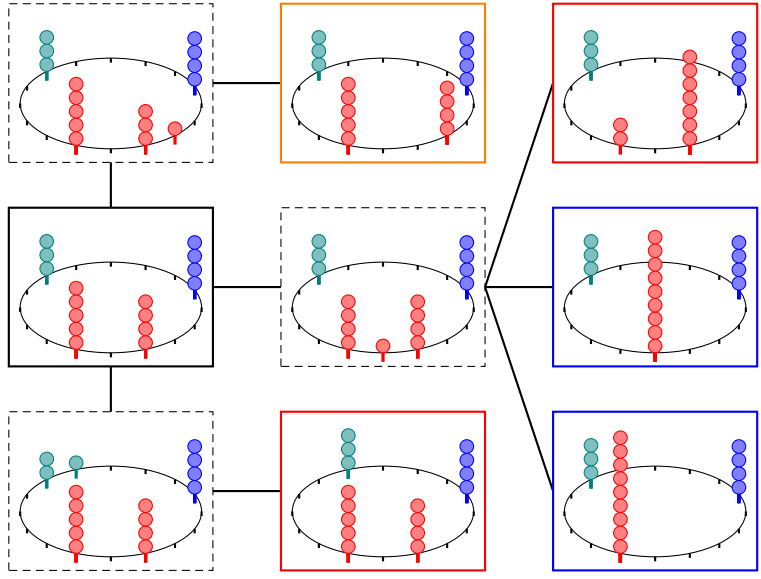
\begin{figure}
\begin{tikzpicture}[scale=0.6]
\draw[thick] (0,2.2)--(0,3.2); \draw[thick] (2.25,4.95)--(3.75,4.95);
\draw[thick] (0,-1.3)--(0,-2.3); \draw[thick] (2.25,-4.05)--(3.75,-4.05);
\draw[thick] (2.25,0.45)--(3.75,0.45);
\draw[thick] (8.25,0.45)--(9.75,4.95); \draw[thick] (8.25,0.45)--(9.75,0.45); \draw[thick] (8.25,0.45)--(9.75,-4.05);

\draw[thick] (-2.25,-1.3) rectangle (2.25,2.2);
\draw[] (0,0) ellipse (2 and 1);
\foreach \i in {0,...,15} { \draw[thick] ({2*cos(\i*22.5)},{1*sin(\i*22.5)-0.1})--({2*cos(\i*22.5)},{1*sin(\i*22.5)}); }

\foreach \i in {1} { \foreach \j in {0,...,3} { \draw[very thick,blue] ({2*cos(\i*22.5)},{1*sin(\i*22.5)-0.2})--({2*cos(\i*22.5)},{1*sin(\i*22.5)});
\fill[blue!50!white] ({2*cos(\i*22.5)},{1*sin(\i*22.5)+0.15+0.3*\j}) circle (0.15); \draw[blue] ({2*cos(\i*22.5)},{1*sin(\i*22.5)+0.15+0.3*\j}) circle (0.15); } }
\foreach \i in {6} { \foreach \j in {0,...,2} { \draw[very thick,teal] ({2*cos(\i*22.5)},{1*sin(\i*22.5)-0.2})--({2*cos(\i*22.5)},{1*sin(\i*22.5)});
\fill[teal!50!white] ({2*cos(\i*22.5)},{1*sin(\i*22.5)+0.15+0.3*\j}) circle (0.15); \draw[teal] ({2*cos(\i*22.5)},{1*sin(\i*22.5)+0.15+0.3*\j}) circle (0.15); } }
\foreach \i in {11} { \foreach \j in {0,...,4} { \draw[very thick,red] ({2*cos(\i*22.5)},{1*sin(\i*22.5)-0.2})--({2*cos(\i*22.5)},{1*sin(\i*22.5)});
\fill[red!50!white] ({2*cos(\i*22.5)},{1*sin(\i*22.5)+0.15+0.3*\j}) circle (0.15); \draw[red] ({2*cos(\i*22.5)},{1*sin(\i*22.5)+0.15+0.3*\j}) circle (0.15); } }
\foreach \i in {13} { \foreach \j in {0,...,3} { \draw[very thick,red] ({2*cos(\i*22.5)},{1*sin(\i*22.5)-0.2})--({2*cos(\i*22.5)},{1*sin(\i*22.5)});
\fill[red!50!white] ({2*cos(\i*22.5)},{1*sin(\i*22.5)+0.15+0.3*\j}) circle (0.15); \draw[red] ({2*cos(\i*22.5)},{1*sin(\i*22.5)+0.15+0.3*\j}) circle (0.15); } }

\begin{scope}[shift={(0,-4.5)}]
\draw[densely dashed] (-2.25,-1.3) rectangle (2.25,2.2);
\draw[] (0,0) ellipse (2 and 1);
\foreach \i in {0,...,15} { \draw[thick] ({2*cos(\i*22.5)},{1*sin(\i*22.5)-0.1})--({2*cos(\i*22.5)},{1*sin(\i*22.5)}); }

\foreach \i in {1} { \foreach \j in {0,...,3} { \draw[very thick,blue] ({2*cos(\i*22.5)},{1*sin(\i*22.5)-0.2})--({2*cos(\i*22.5)},{1*sin(\i*22.5)});
\fill[blue!50!white] ({2*cos(\i*22.5)},{1*sin(\i*22.5)+0.15+0.3*\j}) circle (0.15); \draw[blue] ({2*cos(\i*22.5)},{1*sin(\i*22.5)+0.15+0.3*\j}) circle (0.15); } }
\foreach \i in {5} { \foreach \j in {0,...,0} { \draw[very thick,teal] ({2*cos(\i*22.5)},{1*sin(\i*22.5)-0.2})--({2*cos(\i*22.5)},{1*sin(\i*22.5)});
\fill[teal!50!white] ({2*cos(\i*22.5)},{1*sin(\i*22.5)+0.15+0.3*\j}) circle (0.15); \draw[teal] ({2*cos(\i*22.5)},{1*sin(\i*22.5)+0.15+0.3*\j}) circle (0.15); } }
\foreach \i in {6} { \foreach \j in {0,...,1} { \draw[very thick,teal] ({2*cos(\i*22.5)},{1*sin(\i*22.5)-0.2})--({2*cos(\i*22.5)},{1*sin(\i*22.5)});
\fill[teal!50!white] ({2*cos(\i*22.5)},{1*sin(\i*22.5)+0.15+0.3*\j}) circle (0.15); \draw[teal] ({2*cos(\i*22.5)},{1*sin(\i*22.5)+0.15+0.3*\j}) circle (0.15); } }
\foreach \i in {11} { \foreach \j in {0,...,4} { \draw[very thick,red] ({2*cos(\i*22.5)},{1*sin(\i*22.5)-0.2})--({2*cos(\i*22.5)},{1*sin(\i*22.5)});
\fill[red!50!white] ({2*cos(\i*22.5)},{1*sin(\i*22.5)+0.15+0.3*\j}) circle (0.15); \draw[red] ({2*cos(\i*22.5)},{1*sin(\i*22.5)+0.15+0.3*\j}) circle (0.15); } }
\foreach \i in {13} { \foreach \j in {0,...,3} { \draw[very thick,red] ({2*cos(\i*22.5)},{1*sin(\i*22.5)-0.2})--({2*cos(\i*22.5)},{1*sin(\i*22.5)});
\fill[red!50!white] ({2*cos(\i*22.5)},{1*sin(\i*22.5)+0.15+0.3*\j}) circle (0.15); \draw[red] ({2*cos(\i*22.5)},{1*sin(\i*22.5)+0.15+0.3*\j}) circle (0.15); } }
\end{scope}

\begin{scope}[shift={(0,4.5)}]
\draw[densely dashed] (-2.25,-1.3) rectangle (2.25,2.2);
\draw[] (0,0) ellipse (2 and 1);
\foreach \i in {0,...,15} { \draw[thick] ({2*cos(\i*22.5)},{1*sin(\i*22.5)-0.1})--({2*cos(\i*22.5)},{1*sin(\i*22.5)}); }

\foreach \i in {1} { \foreach \j in {0,...,3} { \draw[very thick,blue] ({2*cos(\i*22.5)},{1*sin(\i*22.5)-0.2})--({2*cos(\i*22.5)},{1*sin(\i*22.5)});
\fill[blue!50!white] ({2*cos(\i*22.5)},{1*sin(\i*22.5)+0.15+0.3*\j}) circle (0.15); \draw[blue] ({2*cos(\i*22.5)},{1*sin(\i*22.5)+0.15+0.3*\j}) circle (0.15); } }
\foreach \i in {6} { \foreach \j in {0,...,2} { \draw[very thick,teal] ({2*cos(\i*22.5)},{1*sin(\i*22.5)-0.2})--({2*cos(\i*22.5)},{1*sin(\i*22.5)});
\fill[teal!50!white] ({2*cos(\i*22.5)},{1*sin(\i*22.5)+0.15+0.3*\j}) circle (0.15); \draw[teal] ({2*cos(\i*22.5)},{1*sin(\i*22.5)+0.15+0.3*\j}) circle (0.15); } }
\foreach \i in {11} { \foreach \j in {0,...,4} { \draw[very thick,red] ({2*cos(\i*22.5)},{1*sin(\i*22.5)-0.2})--({2*cos(\i*22.5)},{1*sin(\i*22.5)});
\fill[red!50!white] ({2*cos(\i*22.5)},{1*sin(\i*22.5)+0.15+0.3*\j}) circle (0.15); \draw[red] ({2*cos(\i*22.5)},{1*sin(\i*22.5)+0.15+0.3*\j}) circle (0.15); } }
\foreach \i in {13} { \foreach \j in {0,...,2} { \draw[very thick,red] ({2*cos(\i*22.5)},{1*sin(\i*22.5)-0.2})--({2*cos(\i*22.5)},{1*sin(\i*22.5)});
\fill[red!50!white] ({2*cos(\i*22.5)},{1*sin(\i*22.5)+0.15+0.3*\j}) circle (0.15); \draw[red] ({2*cos(\i*22.5)},{1*sin(\i*22.5)+0.15+0.3*\j}) circle (0.15); } }
\foreach \i in {14} { \foreach \j in {0,...,0} { \draw[very thick,red] ({2*cos(\i*22.5)},{1*sin(\i*22.5)-0.2})--({2*cos(\i*22.5)},{1*sin(\i*22.5)});
\fill[red!50!white] ({2*cos(\i*22.5)},{1*sin(\i*22.5)+0.15+0.3*\j}) circle (0.15); \draw[red] ({2*cos(\i*22.5)},{1*sin(\i*22.5)+0.15+0.3*\j}) circle (0.15); } }
\end{scope}

\begin{scope}[shift={(6,0)}]
\draw[densely dashed] (-2.25,-1.3) rectangle (2.25,2.2);
\draw[] (0,0) ellipse (2 and 1);
\foreach \i in {0,...,15} { \draw[thick] ({2*cos(\i*22.5)},{1*sin(\i*22.5)-0.1})--({2*cos(\i*22.5)},{1*sin(\i*22.5)}); }

\foreach \i in {1} { \foreach \j in {0,...,3} { \draw[very thick,blue] ({2*cos(\i*22.5)},{1*sin(\i*22.5)-0.2})--({2*cos(\i*22.5)},{1*sin(\i*22.5)});
\fill[blue!50!white] ({2*cos(\i*22.5)},{1*sin(\i*22.5)+0.15+0.3*\j}) circle (0.15); \draw[blue] ({2*cos(\i*22.5)},{1*sin(\i*22.5)+0.15+0.3*\j}) circle (0.15); } }
\foreach \i in {6} { \foreach \j in {0,...,2} { \draw[very thick,teal] ({2*cos(\i*22.5)},{1*sin(\i*22.5)-0.2})--({2*cos(\i*22.5)},{1*sin(\i*22.5)});
\fill[teal!50!white] ({2*cos(\i*22.5)},{1*sin(\i*22.5)+0.15+0.3*\j}) circle (0.15); \draw[teal] ({2*cos(\i*22.5)},{1*sin(\i*22.5)+0.15+0.3*\j}) circle (0.15); } }
\foreach \i in {11} { \foreach \j in {0,...,3} { \draw[very thick,red] ({2*cos(\i*22.5)},{1*sin(\i*22.5)-0.2})--({2*cos(\i*22.5)},{1*sin(\i*22.5)});
\fill[red!50!white] ({2*cos(\i*22.5)},{1*sin(\i*22.5)+0.15+0.3*\j}) circle (0.15); \draw[red] ({2*cos(\i*22.5)},{1*sin(\i*22.5)+0.15+0.3*\j}) circle (0.15); } }
\foreach \i in {12} { \foreach \j in {0,...,0} { \draw[very thick,red] ({2*cos(\i*22.5)},{1*sin(\i*22.5)-0.2})--({2*cos(\i*22.5)},{1*sin(\i*22.5)});
\fill[red!50!white] ({2*cos(\i*22.5)},{1*sin(\i*22.5)+0.15+0.3*\j}) circle (0.15); \draw[red] ({2*cos(\i*22.5)},{1*sin(\i*22.5)+0.15+0.3*\j}) circle (0.15); } }
\foreach \i in {13} { \foreach \j in {0,...,3} { \draw[very thick,red] ({2*cos(\i*22.5)},{1*sin(\i*22.5)-0.2})--({2*cos(\i*22.5)},{1*sin(\i*22.5)});
\fill[red!50!white] ({2*cos(\i*22.5)},{1*sin(\i*22.5)+0.15+0.3*\j}) circle (0.15); \draw[red] ({2*cos(\i*22.5)},{1*sin(\i*22.5)+0.15+0.3*\j}) circle (0.15); } }
\end{scope}

\begin{scope}[shift={(6,-4.5)}]
\draw[red,thick] (-2.25,-1.3) rectangle (2.25,2.2);
\draw[] (0,0) ellipse (2 and 1);
\foreach \i in {0,...,15} { \draw[thick] ({2*cos(\i*22.5)},{1*sin(\i*22.5)-0.1})--({2*cos(\i*22.5)},{1*sin(\i*22.5)}); }

\foreach \i in {1} { \foreach \j in {0,...,3} { \draw[very thick,blue] ({2*cos(\i*22.5)},{1*sin(\i*22.5)-0.2})--({2*cos(\i*22.5)},{1*sin(\i*22.5)});
\fill[blue!50!white] ({2*cos(\i*22.5)},{1*sin(\i*22.5)+0.15+0.3*\j}) circle (0.15); \draw[blue] ({2*cos(\i*22.5)},{1*sin(\i*22.5)+0.15+0.3*\j}) circle (0.15); } }
\foreach \i in {5} { \foreach \j in {0,...,2} { \draw[very thick,teal] ({2*cos(\i*22.5)},{1*sin(\i*22.5)-0.2})--({2*cos(\i*22.5)},{1*sin(\i*22.5)});
\fill[teal!50!white] ({2*cos(\i*22.5)},{1*sin(\i*22.5)+0.15+0.3*\j}) circle (0.15); \draw[teal] ({2*cos(\i*22.5)},{1*sin(\i*22.5)+0.15+0.3*\j}) circle (0.15); } }
\foreach \i in {11} { \foreach \j in {0,...,4} { \draw[very thick,red] ({2*cos(\i*22.5)},{1*sin(\i*22.5)-0.2})--({2*cos(\i*22.5)},{1*sin(\i*22.5)});
\fill[red!50!white] ({2*cos(\i*22.5)},{1*sin(\i*22.5)+0.15+0.3*\j}) circle (0.15); \draw[red] ({2*cos(\i*22.5)},{1*sin(\i*22.5)+0.15+0.3*\j}) circle (0.15); } }
\foreach \i in {13} { \foreach \j in {0,...,3} { \draw[very thick,red] ({2*cos(\i*22.5)},{1*sin(\i*22.5)-0.2})--({2*cos(\i*22.5)},{1*sin(\i*22.5)});
\fill[red!50!white] ({2*cos(\i*22.5)},{1*sin(\i*22.5)+0.15+0.3*\j}) circle (0.15); \draw[red] ({2*cos(\i*22.5)},{1*sin(\i*22.5)+0.15+0.3*\j}) circle (0.15); } }
\end{scope}

\begin{scope}[shift={(6,4.5)}]
\draw[orange,thick] (-2.25,-1.3) rectangle (2.25,2.2);
\draw[] (0,0) ellipse (2 and 1);
\foreach \i in {0,...,15} { \draw[thick] ({2*cos(\i*22.5)},{1*sin(\i*22.5)-0.1})--({2*cos(\i*22.5)},{1*sin(\i*22.5)}); }

\foreach \i in {1} { \foreach \j in {0,...,3} { \draw[very thick,blue] ({2*cos(\i*22.5)},{1*sin(\i*22.5)-0.2})--({2*cos(\i*22.5)},{1*sin(\i*22.5)});
\fill[blue!50!white] ({2*cos(\i*22.5)},{1*sin(\i*22.5)+0.15+0.3*\j}) circle (0.15); \draw[blue] ({2*cos(\i*22.5)},{1*sin(\i*22.5)+0.15+0.3*\j}) circle (0.15); } }
\foreach \i in {6} { \foreach \j in {0,...,2} { \draw[very thick,teal] ({2*cos(\i*22.5)},{1*sin(\i*22.5)-0.2})--({2*cos(\i*22.5)},{1*sin(\i*22.5)});
\fill[teal!50!white] ({2*cos(\i*22.5)},{1*sin(\i*22.5)+0.15+0.3*\j}) circle (0.15); \draw[teal] ({2*cos(\i*22.5)},{1*sin(\i*22.5)+0.15+0.3*\j}) circle (0.15); } }
\foreach \i in {11} { \foreach \j in {0,...,4} { \draw[very thick,red] ({2*cos(\i*22.5)},{1*sin(\i*22.5)-0.2})--({2*cos(\i*22.5)},{1*sin(\i*22.5)});
\fill[red!50!white] ({2*cos(\i*22.5)},{1*sin(\i*22.5)+0.15+0.3*\j}) circle (0.15); \draw[red] ({2*cos(\i*22.5)},{1*sin(\i*22.5)+0.15+0.3*\j}) circle (0.15); } }
\foreach \i in {14} { \foreach \j in {0,...,3} { \draw[very thick,red] ({2*cos(\i*22.5)},{1*sin(\i*22.5)-0.2})--({2*cos(\i*22.5)},{1*sin(\i*22.5)});
\fill[red!50!white] ({2*cos(\i*22.5)},{1*sin(\i*22.5)+0.15+0.3*\j}) circle (0.15); \draw[red] ({2*cos(\i*22.5)},{1*sin(\i*22.5)+0.15+0.3*\j}) circle (0.15); } }
\end{scope}

\begin{scope}[shift={(12,-4.5)}]
\draw[blue,thick] (-2.25,-1.3) rectangle (2.25,2.2);
\draw[] (0,0) ellipse (2 and 1);
\foreach \i in {0,...,15} { \draw[thick] ({2*cos(\i*22.5)},{1*sin(\i*22.5)-0.1})--({2*cos(\i*22.5)},{1*sin(\i*22.5)}); }

\foreach \i in {1} { \foreach \j in {0,...,3} { \draw[very thick,blue] ({2*cos(\i*22.5)},{1*sin(\i*22.5)-0.2})--({2*cos(\i*22.5)},{1*sin(\i*22.5)});
\fill[blue!50!white] ({2*cos(\i*22.5)},{1*sin(\i*22.5)+0.15+0.3*\j}) circle (0.15); \draw[blue] ({2*cos(\i*22.5)},{1*sin(\i*22.5)+0.15+0.3*\j}) circle (0.15); } }
\foreach \i in {6} { \foreach \j in {0,...,2} { \draw[very thick,teal] ({2*cos(\i*22.5)},{1*sin(\i*22.5)-0.2})--({2*cos(\i*22.5)},{1*sin(\i*22.5)});
\fill[teal!50!white] ({2*cos(\i*22.5)},{1*sin(\i*22.5)+0.15+0.3*\j}) circle (0.15); \draw[teal] ({2*cos(\i*22.5)},{1*sin(\i*22.5)+0.15+0.3*\j}) circle (0.15); } }
\foreach \i in {11} { \foreach \j in {0,...,8} { \draw[very thick,red] ({2*cos(\i*22.5)},{1*sin(\i*22.5)-0.2})--({2*cos(\i*22.5)},{1*sin(\i*22.5)});
\fill[red!50!white] ({2*cos(\i*22.5)},{1*sin(\i*22.5)+0.15+0.3*\j}) circle (0.15); \draw[red] ({2*cos(\i*22.5)},{1*sin(\i*22.5)+0.15+0.3*\j}) circle (0.15); } }
\end{scope}

\begin{scope}[shift={(12,0)}]
\draw[blue,thick] (-2.25,-1.3) rectangle (2.25,2.2);
\draw[] (0,0) ellipse (2 and 1);
\foreach \i in {0,...,15} { \draw[thick] ({2*cos(\i*22.5)},{1*sin(\i*22.5)-0.1})--({2*cos(\i*22.5)},{1*sin(\i*22.5)}); }

\foreach \i in {1} { \foreach \j in {0,...,3} { \draw[very thick,blue] ({2*cos(\i*22.5)},{1*sin(\i*22.5)-0.2})--({2*cos(\i*22.5)},{1*sin(\i*22.5)});
\fill[blue!50!white] ({2*cos(\i*22.5)},{1*sin(\i*22.5)+0.15+0.3*\j}) circle (0.15); \draw[blue] ({2*cos(\i*22.5)},{1*sin(\i*22.5)+0.15+0.3*\j}) circle (0.15); } }
\foreach \i in {6} { \foreach \j in {0,...,2} { \draw[very thick,teal] ({2*cos(\i*22.5)},{1*sin(\i*22.5)-0.2})--({2*cos(\i*22.5)},{1*sin(\i*22.5)});
\fill[teal!50!white] ({2*cos(\i*22.5)},{1*sin(\i*22.5)+0.15+0.3*\j}) circle (0.15); \draw[teal] ({2*cos(\i*22.5)},{1*sin(\i*22.5)+0.15+0.3*\j}) circle (0.15); } }
\foreach \i in {12} { \foreach \j in {0,...,8} { \draw[very thick,red] ({2*cos(\i*22.5)},{1*sin(\i*22.5)-0.2})--({2*cos(\i*22.5)},{1*sin(\i*22.5)});
\fill[red!50!white] ({2*cos(\i*22.5)},{1*sin(\i*22.5)+0.15+0.3*\j}) circle (0.15); \draw[red] ({2*cos(\i*22.5)},{1*sin(\i*22.5)+0.15+0.3*\j}) circle (0.15); } }
\end{scope}

\begin{scope}[shift={(12,4.5)}]
\draw[red,thick] (-2.25,-1.3) rectangle (2.25,2.2);
\draw[] (0,0) ellipse (2 and 1);
\foreach \i in {0,...,15} { \draw[thick] ({2*cos(\i*22.5)},{1*sin(\i*22.5)-0.1})--({2*cos(\i*22.5)},{1*sin(\i*22.5)}); }

\foreach \i in {1} { \foreach \j in {0,...,3} { \draw[very thick,blue] ({2*cos(\i*22.5)},{1*sin(\i*22.5)-0.2})--({2*cos(\i*22.5)},{1*sin(\i*22.5)});
\fill[blue!50!white] ({2*cos(\i*22.5)},{1*sin(\i*22.5)+0.15+0.3*\j}) circle (0.15); \draw[blue] ({2*cos(\i*22.5)},{1*sin(\i*22.5)+0.15+0.3*\j}) circle (0.15); } }
\foreach \i in {6} { \foreach \j in {0,...,2} { \draw[very thick,teal] ({2*cos(\i*22.5)},{1*sin(\i*22.5)-0.2})--({2*cos(\i*22.5)},{1*sin(\i*22.5)});
\fill[teal!50!white] ({2*cos(\i*22.5)},{1*sin(\i*22.5)+0.15+0.3*\j}) circle (0.15); \draw[teal] ({2*cos(\i*22.5)},{1*sin(\i*22.5)+0.15+0.3*\j}) circle (0.15); } }
\foreach \i in {11} { \foreach \j in {0,...,1} { \draw[very thick,red] ({2*cos(\i*22.5)},{1*sin(\i*22.5)-0.2})--({2*cos(\i*22.5)},{1*sin(\i*22.5)});
\fill[red!50!white] ({2*cos(\i*22.5)},{1*sin(\i*22.5)+0.15+0.3*\j}) circle (0.15); \draw[red] ({2*cos(\i*22.5)},{1*sin(\i*22.5)+0.15+0.3*\j}) circle (0.15); } }
\foreach \i in {13} { \foreach \j in {0,...,6} { \draw[very thick,red] ({2*cos(\i*22.5)},{1*sin(\i*22.5)-0.2})--({2*cos(\i*22.5)},{1*sin(\i*22.5)});
\fill[red!50!white] ({2*cos(\i*22.5)},{1*sin(\i*22.5)+0.15+0.3*\j}) circle (0.15); \draw[red] ({2*cos(\i*22.5)},{1*sin(\i*22.5)+0.15+0.3*\j}) circle (0.15); } }
\end{scope}
\end{tikzpicture}\caption{\label{fig2.1}Example of configurations in $\mathcal A_{\bm n}^{\bm x}$ where $\xi_{\bm n}^{\bm x} \in \mathcal K_N^\ell$ with $\ell = 4$. Precisely, 
the initial configuration $\xi_{\bm n}^{\bm x}$ (with normal boundary),
configurations in $\mathcal B_N^\ell$ (with dashed boundary),
$\mathcal K_N^\ell$ (with red boundary),
$\mathcal J_N^\ell$ (with orange boundary),
and $\mathcal E_N^{\ell-1}$ (with blue boundary) are presented.}
\end{figure}

\begin{defn}[Collection $\mathcal A_{\bm n}^{\bm x}$ for $\xi_{\bm n}^{\bm x} \in \mathcal K_N^\ell$]\label{def:Anx-def-2}
Fix $\bm{x} \in \mathbb K_L^{\ell}$ and $\bm n \in {\bf N}_N^\ell$ such that $\xi_{\bm n}^{\bm x} \in \mathcal K_N^\ell$.
Denote by $\color{blue}  \mathcal{A}_{\bm{n}}^{\bm{x}}$ the collection of all configurations
that can be attained from consecutive transitions starting from $\xi_{\bm{n}}^{\bm{x}}$
according to the following two rules:
\begin{itemize}
\item If the initial particle jump occurs as $x_i \to x_i +1$ and $x_{i+1} \ne x_i +2$, then after this jump only jumps between $x_i$ and $x_i +1$ are allowed. Similarly if initially $x_i \to x_i -1$ occurs where $x_{i-1} \ne x_i - 2$, then after that only jumps between $x_i \leftrightarrow x_i -1$ are allowed.
\item If the initial jump occurs as $x_i \to x_i +1$ and $x_{i+1} = x_i + 2$, then after that only jumps among $x_i , x_i +1 , x_i +2$ are allowed. Similarly if initially $x_i \to x_i -1$ occurs where $x_{i-1} = x_i -2$, then after that only jumps between $x_i \leftrightarrow x_i-1 \leftrightarrow x_i-2$ are allowed.
\end{itemize}
We may represent
\[
\mathcal A_{\bm n}^{\bm x} = \bigcup_{i=1}^\ell \left( \mathcal A_{\bm n}^{\bm x, i,+} \cup \mathcal A_{\bm n}^{\bm x, i,-} \right) ,
\]
where $\color{blue} \mathcal A_{\bm n}^{\bm x,i,+}$ is defined as
\[
\begin{cases}
\left\{ \eta\in\Omega_{N}:\eta_{x_{i}}+\eta_{x_{i} + 1}=n_{i}, \quad \eta_{x_{j}}=n_{j} \quad \text{for}\enspace j\ne i\right\} & \text{if} \enspace {x_{i + 1} \ne x_i + 2}, \\
\left\{ \eta\in\Omega_{N}:\eta_{x_{i}}+\eta_{x_{i} + 1} + \eta_{x_i + 2} = n_{i} + n_{i+1}, \quad \eta_{x_{j}}=n_{j} \quad \text{for}\enspace j\ne i,i+1 \right\} & \text{if} \enspace x_{i + 1} = x_i + 2,
\end{cases}
\]
and $\color{blue} \mathcal A_{\bm n}^{\bm x,i,-}$ is defined as
\[
\begin{cases}
\left\{ \eta\in\Omega_{N}:\eta_{x_{i}}+\eta_{x_{i} - 1}=n_{i}, \quad \eta_{x_{j}}=n_{j} \quad \text{for}\enspace j\ne i\right\} & \text{if} \enspace {x_{i - 1} \ne x_i - 2}, \\
\left\{ \eta\in\Omega_{N}:\eta_{x_{i}}+\eta_{x_{i} - 1} + \eta_{x_i - 2} = n_{i} + n_{i-1}, \quad \eta_{x_{j}}=n_{j} \quad \text{for}\enspace j\ne i,i-1 \right\} & \text{if} \enspace x_{i - 1} = x_i - 2.
\end{cases}
\]
After a jump from $\xi_{\bm n}^{\bm x}$, provided that the process stays within $\mathcal A_{\bm n}^{\bm x}$, there are three types of elements in $\mathcal E_N$ that the process can hit. Refer to Figure \ref{fig2.1}.
\begin{itemize}
\item It returns to $\mathcal K_N^\ell$; i.e., the number of occupied sites returns to $\ell$ and there still exists a pair of occupied sites with distance $2$. An example is the initial configuration $\xi_{\bm n}^{\bm x}$.
\item It ends up in $\mathcal J_N^\ell$; i.e., the number of occupied sites becomes $\ell$, but now all occupied sites are at least distance $3$ away from each other.
\item It hits $\mathcal E_N^{\ell-1}$. This happens if particles in a pair of sites with distance $2$ merge into a single condensate.
\end{itemize}
\end{defn}

In any cases, define ${\color{blue} \mathcal B_{\bm n}^{\bm x}} := \mathcal A_{\bm n}^{\bm x} \setminus \mathcal E_N$ for each $\xi_{\bm n}^{\bm x} \in \mathcal E_N^\ell$. Collect
\[
{\color{blue} \mathcal A_N^\ell } := \bigcup_{\xi_{\bm n}^{\bm x} \in \mathcal E_N^\ell} \mathcal A_{\bm n}^{\bm x}, \qquad
{\color{blue} \mathcal A_N} := \bigcup_{\ell=1}^k \mathcal A_N^\ell, \qquad
{\color{blue} \mathcal B_N^\ell} := \bigcup_{\xi_{\bm n}^{\bm x} \in \mathcal E_N^\ell} \mathcal B_{\bm n}^{\bm x}, \qquad \text{and} \qquad
{\color{blue} \mathcal B_N} := \bigcup_{\ell=1}^k \mathcal B_N^\ell .
\]
By definition, $\mathcal B_N \cap \mathcal E_N = \emptyset$ and $\mathcal A_N = \mathcal E_N \cup \mathcal B_N$.
Moreover, by comparing Definition \ref{def:typ-jump} and Definitions \ref{def:Anx-def-1} and \ref{def:Anx-def-2}, we obtain that
\begin{equation}\label{eq:Nnx-Anx}
\mathcal N_{\bm n}^{\bm x} = ( \mathcal A_{\bm n}^{\bm x} \cap \mathcal E_N ) \setminus \{ \xi_{\bm n}^{\bm x} \}.
\end{equation}

The collection $\mathcal A_N$ is indeed a tube of typical trajectories, in the sense of the following lemma.
Define $\color{blue} \mathfrak T_N$ as the first time when a jump from $\mathcal B_{\bm n}^{\bm x}$ to $\Omega_N \setminus \mathcal A_{\bm n}^{\bm x}$ occurs for some $\xi_{\bm n}^{\bm x} \in \mathcal E_N$.
More precisely, define
%\[
%\mathfrak b_N(t) := \sum_{s \in (0,t]} \sum_{\xi_{\bm n}^{\bm x} \in \mathcal E_N} \sum_{\zeta \in \mathcal B_{\bm n}^{\bm x}} \sum_{\zeta' \in \Omega_N \setminus \mathcal A_{\bm n}^{\bm x} } {\bf 1} \left\{ \eta_N(s-) = \zeta , \enspace \eta_N(s) = \zeta' \right\} \qquad \text{for} \quad t \ge 0,
%\]
%where the summation in $s \in (0,t]$ is well defined almost surely. Then define
%\begin{equation}\label{eq:TN-def}
%\mathfrak T_N := \inf \{ t>0 : \mathfrak b_N(t) = 1 \},
%\end{equation}
\begin{equation}\label{eq:TN-def}
\mathfrak T_N := \inf \left\{ t>0 :  \exists \, \xi_{\bm n}^{\bm x} \in
\mathcal E_N \quad \text{such that} \quad
\eta_N(t-) \in \mathcal B_{\bm n}^{\bm x}, \quad
\eta_N(t) \in \Omega_N \setminus \mathcal A_{\bm n}^{\bm x} \right\},
\end{equation}
which is a finite value almost surely. Recall that $r_N : \Omega_N \times \Omega_N \to [0,\infty)$ denotes the jump rate function of the original inclusion process \eqref{eq:inc-gen}. Define ${\color{blue} r_N'}:\Omega_N \times \Omega_N \to [0,\infty)$ as
\begin{equation}\label{eq:rN'-def}
r_N' \left(\zeta,\zeta' \right) := \begin{cases} 0 & \text{if} \enspace \zeta \in \mathcal B_{\bm n}^{\bm x}, \enspace \zeta' \in \Omega_N \setminus \mathcal A_{\bm n}^{\bm x} \enspace \text{for some} \enspace \xi_{\bm n}^{\bm x} \in \mathcal E_N, \\
r_N \left( \zeta, \zeta' \right) & \text{otherwise}. \end{cases}
\end{equation}
Denote by $\color{blue} \sigma_N(t)$ the process on $\Omega_N$ defined via $r_N'(\cdot,\cdot)$, and by $\color{blue} {\bf P}_\xi^N$ its law starting from $\xi \in \mathcal E_N$. Mind that this is different from annihilating all jumps from $\mathcal B_N$ to $\Omega_N \setminus \mathcal A_N$. For instance consider $L=4$, $\xi_{\bm n}^{\bm x} = (0,0,5,0)$, and $(0,1,4,0) \in \mathcal B_{\bm n}^{\bm x}$. A jump $(0,1,4,0) \to (1,0,4,0)$ is not allowed for $\sigma_N(t)$ since $(1,0,4,0) \notin \mathcal A_{\bm n}^{\bm x}$, but allowed in this alternative dynamics since $(1,0,4,0) \in \mathcal A_N \setminus \mathcal A_{\bm n}^{\bm x}$.

\begin{lem}\label{lem:tube-typ}
For any given $T>0$,
\begin{equation}\label{eq:tube-typ}
\lim_{N \to \infty} \sup_{\xi \in \mathcal E_N} \mathbb P_\xi^N \left( \mathfrak T_N \le T \right) = 0.
\end{equation}
\end{lem}

\begin{proof}
Before presenting a rigorous proof, we sketch the key ideas. A jump from $\mathcal B_{\bm n}^{\bm x}$ to $\Omega_N \setminus \mathcal A_{\bm n}^{\bm x}$ for some $\xi_{\bm n}^{\bm x} \in \mathcal E_N$ must be a particle jump to an empty site.\footnote{A jump to an empty site does not guarantee the other way around. To see this, consider e.g. that $L=4$ and $\xi_{\bm n}^{\bm x} = (1,0,4,0)$, such that $(0,1,4,0) \in \mathcal B_{\bm n}^{\bm x}$. If the next jump occurs as $(0,1,4,0) \to (1,0,4,0)$, then a jump to an empty site occurred but the process remains inside $\mathcal A_{\bm n}^{\bm x}$.}
According to \eqref{eq:inc-gen}, since each site has exactly two neighbors, the total jump rate from a configuration in $\mathcal B_{\bm n}^{\bm x}$ to $\Omega_N \setminus \mathcal A_{\bm n}^{\bm x}$ is thus bounded above by $2 \theta_N d_N N = 2N^3$.
%Therefore, for each $\xi \in \mathcal E_N$, we may bound the probability in \eqref{eq:tube-typ} from above
%by considering the restricted process (see below) in $\mathcal A_N$ and calculating the probability that an independent clock of rate
%$2N^3$ rings, during the time in which the process remains in $\mathcal B_N$, until time $T$.
Since the first jump from a configuration in $\mathcal E_N$ to $\mathcal B_{\bm n}^{\bm x}$ has rate $O(N^3)$ and the corresponding rate to return to $\mathcal E_N$ is much faster, each time interval between the moments when the process stays in $\mathcal E_N$ has a scale of $N^{-3}$. This implies that after $O(N^3)$ such trials we reach the time limit $T$. Thus, we prove that in each trial, the probability to observe a particle jump to an empty site before returning to $\mathcal E_N$ has a scale strictly smaller than $N^{-3}$, which would conclude the proof via strong Markov property.

\begin{figure}
\begin{tikzpicture}[scale=0.5]
\draw[-latex] (2.3,0)--(3.2,0); \draw[-latex] (7.8,0)--(8.7,0); \draw[-latex] (13.3,0)--(14.2,0); \draw[-latex] (18.8,0)--(19.7,0);
\draw[-latex] (2.3,-5)--(3.2,-5); \draw[-latex] (7.8,-5)--(8.7,-5); \draw[-latex] (13.3,-5)--(14.2,-5); \draw[-latex] (18.8,-5)--(19.7,-5);
\draw[-latex] (24.3,-5)--(25.2,-5); \draw (26,-5) node{$\cdots$};

\draw[] (0,0) ellipse (2 and 1);
\foreach \i in {0,...,15} { \draw[thick] ({2*cos(\i*22.5)},{1*sin(\i*22.5)-0.1})--({2*cos(\i*22.5)},{1*sin(\i*22.5)}); }

\foreach \i in {1} { \foreach \j in {0,...,3} { \draw[very thick,blue] ({2*cos(\i*22.5)},{1*sin(\i*22.5)-0.2})--({2*cos(\i*22.5)},{1*sin(\i*22.5)});
\fill[blue!50!white] ({2*cos(\i*22.5)},{1*sin(\i*22.5)+0.15+0.3*\j}) circle (0.15); \draw[blue] ({2*cos(\i*22.5)},{1*sin(\i*22.5)+0.15+0.3*\j}) circle (0.15); } }
\foreach \i in {6} { \foreach \j in {0,...,2} { \draw[very thick,teal] ({2*cos(\i*22.5)},{1*sin(\i*22.5)-0.2})--({2*cos(\i*22.5)},{1*sin(\i*22.5)});
\fill[teal!50!white] ({2*cos(\i*22.5)},{1*sin(\i*22.5)+0.15+0.3*\j}) circle (0.15); \draw[teal] ({2*cos(\i*22.5)},{1*sin(\i*22.5)+0.15+0.3*\j}) circle (0.15); } }
\foreach \i in {11} { \foreach \j in {0,...,4} { \draw[very thick,red] ({2*cos(\i*22.5)},{1*sin(\i*22.5)-0.2})--({2*cos(\i*22.5)},{1*sin(\i*22.5)});
\fill[red!50!white] ({2*cos(\i*22.5)},{1*sin(\i*22.5)+0.15+0.3*\j}) circle (0.15); \draw[red] ({2*cos(\i*22.5)},{1*sin(\i*22.5)+0.15+0.3*\j}) circle (0.15); } }
\foreach \i in {13} { \foreach \j in {0,...,3} { \draw[very thick,red] ({2*cos(\i*22.5)},{1*sin(\i*22.5)-0.2})--({2*cos(\i*22.5)},{1*sin(\i*22.5)});
\fill[red!50!white] ({2*cos(\i*22.5)},{1*sin(\i*22.5)+0.15+0.3*\j}) circle (0.15); \draw[red] ({2*cos(\i*22.5)},{1*sin(\i*22.5)+0.15+0.3*\j}) circle (0.15); } }
\draw (0,-2.5) node[above]{$\mathfrak t_0$};

\begin{scope}[shift={(5.5,0)}]
\draw[] (0,0) ellipse (2 and 1);
\foreach \i in {0,...,15} { \draw[thick] ({2*cos(\i*22.5)},{1*sin(\i*22.5)-0.1})--({2*cos(\i*22.5)},{1*sin(\i*22.5)}); }

\foreach \i in {1} { \foreach \j in {0,...,3} { \draw[very thick,blue] ({2*cos(\i*22.5)},{1*sin(\i*22.5)-0.2})--({2*cos(\i*22.5)},{1*sin(\i*22.5)});
\fill[blue!50!white] ({2*cos(\i*22.5)},{1*sin(\i*22.5)+0.15+0.3*\j}) circle (0.15); \draw[blue] ({2*cos(\i*22.5)},{1*sin(\i*22.5)+0.15+0.3*\j}) circle (0.15); } }
\foreach \i in {6} { \foreach \j in {0,...,1} { \draw[very thick,teal] ({2*cos(\i*22.5)},{1*sin(\i*22.5)-0.2})--({2*cos(\i*22.5)},{1*sin(\i*22.5)});
\fill[teal!50!white] ({2*cos(\i*22.5)},{1*sin(\i*22.5)+0.15+0.3*\j}) circle (0.15); \draw[teal] ({2*cos(\i*22.5)},{1*sin(\i*22.5)+0.15+0.3*\j}) circle (0.15); } }
\foreach \i in {7} { \foreach \j in {0,...,0} { \draw[very thick,teal] ({2*cos(\i*22.5)},{1*sin(\i*22.5)-0.2})--({2*cos(\i*22.5)},{1*sin(\i*22.5)});
\fill[teal!50!white] ({2*cos(\i*22.5)},{1*sin(\i*22.5)+0.15+0.3*\j}) circle (0.15); \draw[teal] ({2*cos(\i*22.5)},{1*sin(\i*22.5)+0.15+0.3*\j}) circle (0.15); } }
\foreach \i in {11} { \foreach \j in {0,...,4} { \draw[very thick,red] ({2*cos(\i*22.5)},{1*sin(\i*22.5)-0.2})--({2*cos(\i*22.5)},{1*sin(\i*22.5)});
\fill[red!50!white] ({2*cos(\i*22.5)},{1*sin(\i*22.5)+0.15+0.3*\j}) circle (0.15); \draw[red] ({2*cos(\i*22.5)},{1*sin(\i*22.5)+0.15+0.3*\j}) circle (0.15); } }
\foreach \i in {13} { \foreach \j in {0,...,3} { \draw[very thick,red] ({2*cos(\i*22.5)},{1*sin(\i*22.5)-0.2})--({2*cos(\i*22.5)},{1*sin(\i*22.5)});
\fill[red!50!white] ({2*cos(\i*22.5)},{1*sin(\i*22.5)+0.15+0.3*\j}) circle (0.15); \draw[red] ({2*cos(\i*22.5)},{1*sin(\i*22.5)+0.15+0.3*\j}) circle (0.15); } }
\draw (0,-2.5) node[above]{$\mathfrak s_0$};
\end{scope}

\begin{scope}[shift={(11,0)}]
\draw[] (0,0) ellipse (2 and 1);
\foreach \i in {0,...,15} { \draw[thick] ({2*cos(\i*22.5)},{1*sin(\i*22.5)-0.1})--({2*cos(\i*22.5)},{1*sin(\i*22.5)}); }

\foreach \i in {1} { \foreach \j in {0,...,3} { \draw[very thick,blue] ({2*cos(\i*22.5)},{1*sin(\i*22.5)-0.2})--({2*cos(\i*22.5)},{1*sin(\i*22.5)});
\fill[blue!50!white] ({2*cos(\i*22.5)},{1*sin(\i*22.5)+0.15+0.3*\j}) circle (0.15); \draw[blue] ({2*cos(\i*22.5)},{1*sin(\i*22.5)+0.15+0.3*\j}) circle (0.15); } }
\foreach \i in {7} { \foreach \j in {0,...,2} { \draw[very thick,teal] ({2*cos(\i*22.5)},{1*sin(\i*22.5)-0.2})--({2*cos(\i*22.5)},{1*sin(\i*22.5)});
\fill[teal!50!white] ({2*cos(\i*22.5)},{1*sin(\i*22.5)+0.15+0.3*\j}) circle (0.15); \draw[teal] ({2*cos(\i*22.5)},{1*sin(\i*22.5)+0.15+0.3*\j}) circle (0.15); } }
\foreach \i in {11} { \foreach \j in {0,...,4} { \draw[very thick,red] ({2*cos(\i*22.5)},{1*sin(\i*22.5)-0.2})--({2*cos(\i*22.5)},{1*sin(\i*22.5)});
\fill[red!50!white] ({2*cos(\i*22.5)},{1*sin(\i*22.5)+0.15+0.3*\j}) circle (0.15); \draw[red] ({2*cos(\i*22.5)},{1*sin(\i*22.5)+0.15+0.3*\j}) circle (0.15); } }
\foreach \i in {13} { \foreach \j in {0,...,3} { \draw[very thick,red] ({2*cos(\i*22.5)},{1*sin(\i*22.5)-0.2})--({2*cos(\i*22.5)},{1*sin(\i*22.5)});
\fill[red!50!white] ({2*cos(\i*22.5)},{1*sin(\i*22.5)+0.15+0.3*\j}) circle (0.15); \draw[red] ({2*cos(\i*22.5)},{1*sin(\i*22.5)+0.15+0.3*\j}) circle (0.15); } }
\draw (0,-2.5) node[above]{$\mathfrak t_1$};
\end{scope}

\begin{scope}[shift={(16.5,0)}]
\draw[] (0,0) ellipse (2 and 1);
\foreach \i in {0,...,15} { \draw[thick] ({2*cos(\i*22.5)},{1*sin(\i*22.5)-0.1})--({2*cos(\i*22.5)},{1*sin(\i*22.5)}); }

\foreach \i in {1} { \foreach \j in {0,...,3} { \draw[very thick,blue] ({2*cos(\i*22.5)},{1*sin(\i*22.5)-0.2})--({2*cos(\i*22.5)},{1*sin(\i*22.5)});
\fill[blue!50!white] ({2*cos(\i*22.5)},{1*sin(\i*22.5)+0.15+0.3*\j}) circle (0.15); \draw[blue] ({2*cos(\i*22.5)},{1*sin(\i*22.5)+0.15+0.3*\j}) circle (0.15); } }
\foreach \i in {7} { \foreach \j in {0,...,2} { \draw[very thick,teal] ({2*cos(\i*22.5)},{1*sin(\i*22.5)-0.2})--({2*cos(\i*22.5)},{1*sin(\i*22.5)});
\fill[teal!50!white] ({2*cos(\i*22.5)},{1*sin(\i*22.5)+0.15+0.3*\j}) circle (0.15); \draw[teal] ({2*cos(\i*22.5)},{1*sin(\i*22.5)+0.15+0.3*\j}) circle (0.15); } }
\foreach \i in {11} { \foreach \j in {0,...,3} { \draw[very thick,red] ({2*cos(\i*22.5)},{1*sin(\i*22.5)-0.2})--({2*cos(\i*22.5)},{1*sin(\i*22.5)});
\fill[red!50!white] ({2*cos(\i*22.5)},{1*sin(\i*22.5)+0.15+0.3*\j}) circle (0.15); \draw[red] ({2*cos(\i*22.5)},{1*sin(\i*22.5)+0.15+0.3*\j}) circle (0.15); } }
\foreach \i in {12} { \foreach \j in {0,...,0} { \draw[very thick,red] ({2*cos(\i*22.5)},{1*sin(\i*22.5)-0.2})--({2*cos(\i*22.5)},{1*sin(\i*22.5)});
\fill[red!50!white] ({2*cos(\i*22.5)},{1*sin(\i*22.5)+0.15+0.3*\j}) circle (0.15); \draw[red] ({2*cos(\i*22.5)},{1*sin(\i*22.5)+0.15+0.3*\j}) circle (0.15); } }
\foreach \i in {13} { \foreach \j in {0,...,3} { \draw[very thick,red] ({2*cos(\i*22.5)},{1*sin(\i*22.5)-0.2})--({2*cos(\i*22.5)},{1*sin(\i*22.5)});
\fill[red!50!white] ({2*cos(\i*22.5)},{1*sin(\i*22.5)+0.15+0.3*\j}) circle (0.15); \draw[red] ({2*cos(\i*22.5)},{1*sin(\i*22.5)+0.15+0.3*\j}) circle (0.15); } }
\draw (0,-2.5) node[above]{$\mathfrak s_1$};
\end{scope}

\begin{scope}[shift={(22,0)}]
\draw[] (0,0) ellipse (2 and 1);
\foreach \i in {0,...,15} { \draw[thick] ({2*cos(\i*22.5)},{1*sin(\i*22.5)-0.1})--({2*cos(\i*22.5)},{1*sin(\i*22.5)}); }

\foreach \i in {1} { \foreach \j in {0,...,3} { \draw[very thick,blue] ({2*cos(\i*22.5)},{1*sin(\i*22.5)-0.2})--({2*cos(\i*22.5)},{1*sin(\i*22.5)});
\fill[blue!50!white] ({2*cos(\i*22.5)},{1*sin(\i*22.5)+0.15+0.3*\j}) circle (0.15); \draw[blue] ({2*cos(\i*22.5)},{1*sin(\i*22.5)+0.15+0.3*\j}) circle (0.15); } }
\foreach \i in {7} { \foreach \j in {0,...,2} { \draw[very thick,teal] ({2*cos(\i*22.5)},{1*sin(\i*22.5)-0.2})--({2*cos(\i*22.5)},{1*sin(\i*22.5)});
\fill[teal!50!white] ({2*cos(\i*22.5)},{1*sin(\i*22.5)+0.15+0.3*\j}) circle (0.15); \draw[teal] ({2*cos(\i*22.5)},{1*sin(\i*22.5)+0.15+0.3*\j}) circle (0.15); } }
\foreach \i in {11} { \foreach \j in {0,...,6} { \draw[very thick,red] ({2*cos(\i*22.5)},{1*sin(\i*22.5)-0.2})--({2*cos(\i*22.5)},{1*sin(\i*22.5)});
\fill[red!50!white] ({2*cos(\i*22.5)},{1*sin(\i*22.5)+0.15+0.3*\j}) circle (0.15); \draw[red] ({2*cos(\i*22.5)},{1*sin(\i*22.5)+0.15+0.3*\j}) circle (0.15); } }
\foreach \i in {13} { \foreach \j in {0,...,1} { \draw[very thick,red] ({2*cos(\i*22.5)},{1*sin(\i*22.5)-0.2})--({2*cos(\i*22.5)},{1*sin(\i*22.5)});
\fill[red!50!white] ({2*cos(\i*22.5)},{1*sin(\i*22.5)+0.15+0.3*\j}) circle (0.15); \draw[red] ({2*cos(\i*22.5)},{1*sin(\i*22.5)+0.15+0.3*\j}) circle (0.15); } }
\draw (0,-2.5) node[above]{$\mathfrak t_2$};
\end{scope}

\begin{scope}[shift={(5.5,-5)}]
\draw[] (0,0) ellipse (2 and 1);
\foreach \i in {0,...,15} { \draw[thick] ({2*cos(\i*22.5)},{1*sin(\i*22.5)-0.1})--({2*cos(\i*22.5)},{1*sin(\i*22.5)}); }

\foreach \i in {1} { \foreach \j in {0,...,2} { \draw[very thick,blue] ({2*cos(\i*22.5)},{1*sin(\i*22.5)-0.2})--({2*cos(\i*22.5)},{1*sin(\i*22.5)});
\fill[blue!50!white] ({2*cos(\i*22.5)},{1*sin(\i*22.5)+0.15+0.3*\j}) circle (0.15); \draw[blue] ({2*cos(\i*22.5)},{1*sin(\i*22.5)+0.15+0.3*\j}) circle (0.15); } }
\foreach \i in {2} { \foreach \j in {0,...,0} { \draw[very thick,blue] ({2*cos(\i*22.5)},{1*sin(\i*22.5)-0.2})--({2*cos(\i*22.5)},{1*sin(\i*22.5)});
\fill[blue!50!white] ({2*cos(\i*22.5)},{1*sin(\i*22.5)+0.15+0.3*\j}) circle (0.15); \draw[blue] ({2*cos(\i*22.5)},{1*sin(\i*22.5)+0.15+0.3*\j}) circle (0.15); } }
\foreach \i in {7} { \foreach \j in {0,...,2} { \draw[very thick,teal] ({2*cos(\i*22.5)},{1*sin(\i*22.5)-0.2})--({2*cos(\i*22.5)},{1*sin(\i*22.5)});
\fill[teal!50!white] ({2*cos(\i*22.5)},{1*sin(\i*22.5)+0.15+0.3*\j}) circle (0.15); \draw[teal] ({2*cos(\i*22.5)},{1*sin(\i*22.5)+0.15+0.3*\j}) circle (0.15); } }
\foreach \i in {11} { \foreach \j in {0,...,6} { \draw[very thick,red] ({2*cos(\i*22.5)},{1*sin(\i*22.5)-0.2})--({2*cos(\i*22.5)},{1*sin(\i*22.5)});
\fill[red!50!white] ({2*cos(\i*22.5)},{1*sin(\i*22.5)+0.15+0.3*\j}) circle (0.15); \draw[red] ({2*cos(\i*22.5)},{1*sin(\i*22.5)+0.15+0.3*\j}) circle (0.15); } }
\foreach \i in {13} { \foreach \j in {0,...,1} { \draw[very thick,red] ({2*cos(\i*22.5)},{1*sin(\i*22.5)-0.2})--({2*cos(\i*22.5)},{1*sin(\i*22.5)});
\fill[red!50!white] ({2*cos(\i*22.5)},{1*sin(\i*22.5)+0.15+0.3*\j}) circle (0.15); \draw[red] ({2*cos(\i*22.5)},{1*sin(\i*22.5)+0.15+0.3*\j}) circle (0.15); } }
\draw (0,-2.5) node[above]{$\mathfrak s_2$};
\end{scope}

\begin{scope}[shift={(11,-5)}]
\draw[] (0,0) ellipse (2 and 1);
\foreach \i in {0,...,15} { \draw[thick] ({2*cos(\i*22.5)},{1*sin(\i*22.5)-0.1})--({2*cos(\i*22.5)},{1*sin(\i*22.5)}); }

\foreach \i in {1} { \foreach \j in {0,...,3} { \draw[very thick,blue] ({2*cos(\i*22.5)},{1*sin(\i*22.5)-0.2})--({2*cos(\i*22.5)},{1*sin(\i*22.5)});
\fill[blue!50!white] ({2*cos(\i*22.5)},{1*sin(\i*22.5)+0.15+0.3*\j}) circle (0.15); \draw[blue] ({2*cos(\i*22.5)},{1*sin(\i*22.5)+0.15+0.3*\j}) circle (0.15); } }
\foreach \i in {7} { \foreach \j in {0,...,2} { \draw[very thick,teal] ({2*cos(\i*22.5)},{1*sin(\i*22.5)-0.2})--({2*cos(\i*22.5)},{1*sin(\i*22.5)});
\fill[teal!50!white] ({2*cos(\i*22.5)},{1*sin(\i*22.5)+0.15+0.3*\j}) circle (0.15); \draw[teal] ({2*cos(\i*22.5)},{1*sin(\i*22.5)+0.15+0.3*\j}) circle (0.15); } }
\foreach \i in {11} { \foreach \j in {0,...,6} { \draw[very thick,red] ({2*cos(\i*22.5)},{1*sin(\i*22.5)-0.2})--({2*cos(\i*22.5)},{1*sin(\i*22.5)});
\fill[red!50!white] ({2*cos(\i*22.5)},{1*sin(\i*22.5)+0.15+0.3*\j}) circle (0.15); \draw[red] ({2*cos(\i*22.5)},{1*sin(\i*22.5)+0.15+0.3*\j}) circle (0.15); } }
\foreach \i in {13} { \foreach \j in {0,...,1} { \draw[very thick,red] ({2*cos(\i*22.5)},{1*sin(\i*22.5)-0.2})--({2*cos(\i*22.5)},{1*sin(\i*22.5)});
\fill[red!50!white] ({2*cos(\i*22.5)},{1*sin(\i*22.5)+0.15+0.3*\j}) circle (0.15); \draw[red] ({2*cos(\i*22.5)},{1*sin(\i*22.5)+0.15+0.3*\j}) circle (0.15); } }
\draw (0,-2.5) node[above]{$\mathfrak t_3$};
\end{scope}

\begin{scope}[shift={(16.5,-5)}]
\draw[] (0,0) ellipse (2 and 1);
\foreach \i in {0,...,15} { \draw[thick] ({2*cos(\i*22.5)},{1*sin(\i*22.5)-0.1})--({2*cos(\i*22.5)},{1*sin(\i*22.5)}); }

\foreach \i in {1} { \foreach \j in {0,...,3} { \draw[very thick,blue] ({2*cos(\i*22.5)},{1*sin(\i*22.5)-0.2})--({2*cos(\i*22.5)},{1*sin(\i*22.5)});
\fill[blue!50!white] ({2*cos(\i*22.5)},{1*sin(\i*22.5)+0.15+0.3*\j}) circle (0.15); \draw[blue] ({2*cos(\i*22.5)},{1*sin(\i*22.5)+0.15+0.3*\j}) circle (0.15); } }
\foreach \i in {7} { \foreach \j in {0,...,2} { \draw[very thick,teal] ({2*cos(\i*22.5)},{1*sin(\i*22.5)-0.2})--({2*cos(\i*22.5)},{1*sin(\i*22.5)});
\fill[teal!50!white] ({2*cos(\i*22.5)},{1*sin(\i*22.5)+0.15+0.3*\j}) circle (0.15); \draw[teal] ({2*cos(\i*22.5)},{1*sin(\i*22.5)+0.15+0.3*\j}) circle (0.15); } }
\foreach \i in {11} { \foreach \j in {0,...,6} { \draw[very thick,red] ({2*cos(\i*22.5)},{1*sin(\i*22.5)-0.2})--({2*cos(\i*22.5)},{1*sin(\i*22.5)});
\fill[red!50!white] ({2*cos(\i*22.5)},{1*sin(\i*22.5)+0.15+0.3*\j}) circle (0.15); \draw[red] ({2*cos(\i*22.5)},{1*sin(\i*22.5)+0.15+0.3*\j}) circle (0.15); } }
\foreach \i in {12} { \foreach \j in {0,...,0} { \draw[very thick,red] ({2*cos(\i*22.5)},{1*sin(\i*22.5)-0.2})--({2*cos(\i*22.5)},{1*sin(\i*22.5)});
\fill[red!50!white] ({2*cos(\i*22.5)},{1*sin(\i*22.5)+0.15+0.3*\j}) circle (0.15); \draw[red] ({2*cos(\i*22.5)},{1*sin(\i*22.5)+0.15+0.3*\j}) circle (0.15); } }
\foreach \i in {13} { \foreach \j in {0,...,0} { \draw[very thick,red] ({2*cos(\i*22.5)},{1*sin(\i*22.5)-0.2})--({2*cos(\i*22.5)},{1*sin(\i*22.5)});
\fill[red!50!white] ({2*cos(\i*22.5)},{1*sin(\i*22.5)+0.15+0.3*\j}) circle (0.15); \draw[red] ({2*cos(\i*22.5)},{1*sin(\i*22.5)+0.15+0.3*\j}) circle (0.15); } }
\draw (0,-2.5) node[above]{$\mathfrak s_3$};
\end{scope}

\begin{scope}[shift={(22,-5)}]
\draw[] (0,0) ellipse (2 and 1);
\foreach \i in {0,...,15} { \draw[thick] ({2*cos(\i*22.5)},{1*sin(\i*22.5)-0.1})--({2*cos(\i*22.5)},{1*sin(\i*22.5)}); }

\foreach \i in {1} { \foreach \j in {0,...,3} { \draw[very thick,blue] ({2*cos(\i*22.5)},{1*sin(\i*22.5)-0.2})--({2*cos(\i*22.5)},{1*sin(\i*22.5)});
\fill[blue!50!white] ({2*cos(\i*22.5)},{1*sin(\i*22.5)+0.15+0.3*\j}) circle (0.15); \draw[blue] ({2*cos(\i*22.5)},{1*sin(\i*22.5)+0.15+0.3*\j}) circle (0.15); } }
\foreach \i in {7} { \foreach \j in {0,...,2} { \draw[very thick,teal] ({2*cos(\i*22.5)},{1*sin(\i*22.5)-0.2})--({2*cos(\i*22.5)},{1*sin(\i*22.5)});
\fill[teal!50!white] ({2*cos(\i*22.5)},{1*sin(\i*22.5)+0.15+0.3*\j}) circle (0.15); \draw[teal] ({2*cos(\i*22.5)},{1*sin(\i*22.5)+0.15+0.3*\j}) circle (0.15); } }
\foreach \i in {12} { \foreach \j in {0,...,8} { \draw[very thick,red] ({2*cos(\i*22.5)},{1*sin(\i*22.5)-0.2})--({2*cos(\i*22.5)},{1*sin(\i*22.5)});
\fill[red!50!white] ({2*cos(\i*22.5)},{1*sin(\i*22.5)+0.15+0.3*\j}) circle (0.15); \draw[red] ({2*cos(\i*22.5)},{1*sin(\i*22.5)+0.15+0.3*\j}) circle (0.15); } }
\draw (0,-2.5) node[above]{$\mathfrak t_4$};
\end{scope}
\end{tikzpicture}\caption{\label{fig2.2}Consecutive hitting times $0=\mathfrak t_0 < \mathfrak s_0 < \cdots < \mathfrak s_3 < \mathfrak t_4 < \cdots$ where $\ell = 4$.}
\end{figure}

Rigorously, define a sequence of increasing stopping times
\begin{equation}\label{eq:sigma-stopping}
{\color{blue} \mathfrak t_{0} } < {\color{blue} \mathfrak t_{1} } < {\color{blue} \mathfrak t_{2} } < \cdots
\end{equation}
as follows.
Let $\mathfrak t_0 = 0$, and let $\mathfrak t_{i+1} > \mathfrak t_i$ be the first return time to $\mathcal E_N$ after time $\mathfrak t_i$. In addition, define
\begin{equation}\label{eq:nu-stopping}
{\color{blue} \mathfrak s_0 } < {\color{blue} \mathfrak s_1 } < {\color{blue} \mathfrak s_2 } < \cdots
\end{equation}
as each $\mathfrak s_i$ being the first jump time after $\mathfrak t_i$. Then, the process stays in $\mathcal B_N$ at times in the intervals
\[
\bigcup_{i=0}^\infty \, [\mathfrak s_i , \mathfrak t_{i+1}).
\]
See Figure \ref{fig2.2}.
Define ${\color{blue} \mathfrak n} := \inf \, \{ i \ge 1 : \mathfrak t_i > T \}$.
The probability that an exponential random variable with rate $\lambda$
rings within a time period of $M$ is exactly $1-e^{-\lambda M}$.
Thus,
\begin{equation}\label{eq:change-prob-1}
\mathbb P_\xi^N (\mathfrak T_N \le T ) \le {\rm E}^{{\bf P}_\xi^N} \left[ 1- e^{- 2N^3 \sum_{i=1}^{\mathfrak n} (\mathfrak t_i - \mathfrak s_{i-1}) } \right] ,
\end{equation}
where, recall, ${\bf P}_\xi^N$ was defined after \eqref{eq:rN'-def}.
First, observe that $\mathfrak n > m$ if and only if $\mathfrak t_m \le T$, which implies that at least $m$ different jumps from $\mathcal E_N$ occur until time $T$. Since any jump from $\mathcal E_N$ has rate $2 \theta_N d_N N = 2N^3$ and since the sum of $m$ independent exponential random variables of parameter $\lambda$ has a Gamma distribution of parameters $m$ and $\lambda$,
\[
{\bf P}_\xi^N ( \mathfrak n > m) \le P( \Gamma_{m,N} \le T) = \int_0^{2N^3T} \frac{t^{m-1} e^{-t}}{(m-1)!} \, {\rm d}t =: {\color{blue} \alpha_m},
\]
where $\Gamma_{m,N} \sim {\rm Gamma} \, (m, 2N^3)$. Integrating by parts gives $\alpha_m = - e^{-2N^3T} \frac{(2N^3T)^{m-1}}{(m-1)!} + \alpha_{m-1}$. Hence, as $\alpha_1 = 1-e^{-2N^3T}$,
\[
\alpha_m = 1 - e^{-2N^3T} \sum_{n=0}^{m-1} \frac{(2N^3T)^n}{n!} \le \frac{(2N^3T)^m}{m!} ,
\]
where the inequality follows from a standard Taylor estimate. Thus, we can take a sequence of integers ${\color{blue} m_{1,N}} \simeq 2eT N^3$ such that\footnote{Here, $f_N \simeq g_N$ if $\lim_{N \to \infty} f_N / g_N = 1$ and $f_N=o(g_N)$ if $\lim_{N\to\infty}f_{N}/g_{N}=0$.} by Stirling's formula,
\[
\sup_{\xi \in \mathcal E_N} {\bf P}_\xi^N ( \mathfrak n > m_{1,N} ) \le \frac {(2N^3T)^{m_{1,N}}}{m_{1,N}!} \le \frac {(2N^3T)^{m_{1,N}}} {m_{1,N}^{m_{1,N}} \sqrt{2\pi m_{1,N}} / e^{m_{1,N}} }  = o(1).
\]
Substituting this bound in \eqref{eq:change-prob-1} yields that
\begin{equation}\label{eq:change-prob-2}
\begin{aligned}
\mathbb P_\xi^N (\mathfrak T_N \le T ) & \le {\rm E}^{{\bf P}_\xi^N} \left[ \left( 1- e^{- 2N^3 \sum_{i=1}^{\mathfrak n} ( \mathfrak t_i - \mathfrak s_{i-1} ) } \right) {\bf 1} \{ \mathfrak n \le m_{1,N} \} \right] + {\bf P}_\xi^N ( \mathfrak n > m_{1,N} ) \\
& \le {\rm E}^{{\bf P}_\xi^N} \left[ 1- e^{- 2N^3 \sum_{i=1}^{m_{1,N}} ( \mathfrak t_i - \mathfrak s_{i-1} ) } \right] + o(1) ,
\end{aligned}
\end{equation}
where the error is uniform over all $\xi \in \mathcal E_N$.

Our next objective is to estimate the expectation in the right-hand side of \eqref{eq:change-prob-2}.
By the strong Markov property at each stopping time $\mathfrak t_i$ for $i \in \llbracket 1,m_{1,N} \rrbracket$,
\begin{equation}\label{eq:change-prob-3}
{\rm E}^{{\bf P}_\xi^N} \left[ 1- e^{- 2N^3 \sum_{i=1}^{m_{1,N}} (\mathfrak t_i - \mathfrak s_{i-1}) } \right] \le 1 -  \left( \inf_{\zeta \in \mathcal E_N} {\rm E}^{{\bf P}_\zeta^N } \left[ e^{-2N^3 ( \mathfrak t_1 - \mathfrak s_0 )} \right] \right)^{m_{1,N}} .
\end{equation}
For the expectation on the right-hand side of \eqref{eq:change-prob-3}, apply again the strong Markov property at $\mathfrak s_0$ to obtain that
\begin{equation}\label{eq:calc-going}
{\rm E}^{{\bf P}_\zeta^N } \left[ e^{-2N^3 ( \mathfrak t_1 - \mathfrak s_0 )} \right]
= \sum_{\eta \in \mathcal B_N} {\bf P}_\zeta^N ( \sigma_N ( \mathfrak s_0 ) = \eta ) \, {\rm E}^{{\bf P}_\eta^N} \left[ e^{-2N^3 \mathfrak t_1} \right].
\end{equation}

First, suppose that $\eta \in \mathcal B_N$ is obtained from $\zeta$ by a particle's jump $x \to y$ such that $x,y$ are two isolated occupied sites. Let $\zeta_x = n$ such that $(\eta_x , \eta_y ) = (n-1,1)$. Note that $r_N(\zeta,\eta) = \theta_N d_N n$. In this case, the dynamics of $\sigma_{N}(\cdot)$ starting from
$\eta$ until time $\mathfrak t_1$ corresponds to a one-dimensional random
walk on $\llbracket0,n\rrbracket$ starting from $1$ and stopped
upon hitting $\{0,n\}$, whose jump rates are
\begin{equation}
r(i,i+1)= \theta_N (n-i)(i+d_{N}),\quad r(i,i-1)= \theta_N i(n-i+d_{N})\qquad\text{for}\quad i\in\llbracket1,n-1\rrbracket.\label{eq:coupling-1}
\end{equation}
Thus, applying Lemma \ref{lem:B1} and since $n\le N$ and $e^{-\gamma}\ge1-\gamma$,
\begin{equation}\label{eq:RW-bound-1}
\begin{aligned}
{\rm E}^{{\bf P}_{\eta}^{N}} \left[e^{-2 N^3 \mathfrak t_1} \right]
\ge1-2 N^3 \,  {\rm E}^{{\bf P}_{\eta}^{N}} [ \mathfrak t_1 ] & \ge 1 - 2N^3 \, e^{d_N (1+\log N)} \, \frac{2 (1+\log N)}{\theta_N n}   \\
& = 1 - \frac {4 d_N N^3} {r_N (\zeta,\eta)} \,e^{d_{N}(1+\log N)}\,(1+\log N).
\end{aligned}
\end{equation}

On the other hand, suppose that $\eta \in \mathcal B_N$ is obtained from $\zeta$ by $x \to y$, where the other neighbor $z$ of $y$ is also occupied by $\zeta$. Let $(\zeta_x , \zeta_z) = (n,n')$ such that $(\eta_x,\eta_y,\eta_z) = (n-1,1,n')$
(so that, again, $r_N (\zeta,\eta) = \theta_N d_N n$).
In this case, the hitting time $\mathfrak t_1$ equals the hitting time of $\{0,n+n'\}$ of the
one-dimensional random walk on $\llbracket0,n+n'\rrbracket$ starting
from $1$ whose rates are
\begin{equation}
r(i,i+1) = \theta_N (n+n'-i)(i+d_{N}),\quad r(i,i-1) = \theta_N i(n+n'-i+ 2d_{N}),\quad i\in\llbracket1,n+n'-1\rrbracket.\label{eq:coupling-2}
\end{equation}
Indeed, the two dynamics can be coupled by mapping the number
of particles at $y$ to the location of the walk on $\llbracket0,n+n'\rrbracket$.
At the hitting time $\mathfrak t_1$, the random walk sits either at $0$ (when
$y$ becomes empty) or $n+n'$ (when all the particles gather at $y$), thus the hitting time of $\{0,n+n'\}$
in the random walk is exactly $\mathfrak t_1$. By Lemma \ref{lem:B4},
\begin{equation}\label{eq:RW-bound-2}
\begin{aligned}
{\rm E}^{{\bf P}_{\eta}^{N}} \left[e^{-2 N^3 \mathfrak t_1} \right]
\ge1-2 N^3 \,  {\rm E}^{{\bf P}_{\eta}^{N}} [ \mathfrak t_1 ] & \ge 1 - 2N^3 \, e^{d_N (1+\log N)} \, \frac{2 (1+\log N)}{\theta_N n}   \\
& = 1 - \frac {4 d_N N^3} {r_N (\zeta,\eta)} \,e^{d_{N}(1+\log N)}\,(1+\log N).
\end{aligned}
\end{equation}
Substituting \eqref{eq:RW-bound-1} and \eqref{eq:RW-bound-2} in \eqref{eq:calc-going} yields that
\[
{\rm E}^{{\bf P}_\zeta^N} \left[ e^{-2N^3 (\mathfrak t_1 - \mathfrak s_0) } \right] \ge 1 - \sum_{\eta \in \mathcal B_N : \, r_N(\zeta,\eta) > 0} \, \frac {4 d_N N^3 } { 2 \theta_N d_N N} \,e^{d_{N}(1+\log N)}\,(1+\log N) ,
\]
where the constant $2\theta_N d_N N$ in the denominator indicates the holding rate of $\zeta \in \mathcal E_N$. Simplifying and noting that the number of configurations $\eta \in \mathcal B_N$ with $r_N(\zeta,\eta)>0$ is at most $2k$,
\begin{equation}\label{eq:one-jump-bound}
{\rm E}^{{\bf P}_\zeta^N} \left[ e^{-2N^3 (\mathfrak t_1 - \mathfrak s_0) } \right] \ge 1 - 4k d_N \,e^{d_{N}(1+\log N)}\,(1+\log N) .
\end{equation}
Substituting this to \eqref{eq:change-prob-3}, and applying Bernoulli's inequality, we obtain that
\begin{equation}\label{eq:expect-bound}
\begin{aligned}
{\rm E}^{{\bf P}_\xi^N} \left[ 1- e^{- 2N^3 \sum_{i=1}^{m_{1,N}} ( \mathfrak t_i - \mathfrak s_{i-1} ) } \right] & \le 1 -  \left( 1 - 4k d_N \,e^{d_{N}(1+\log N)}\,(1+\log N) \right)^{m_{1,N}} \\
& \le 4m_{1,N} k d_N \,e^{d_{N}(1+\log N)}\,(1+\log N).
\end{aligned}
\end{equation}
Since $m_{1,N} \simeq 2eTN^3$, \eqref{eq:dN-cond} indicates that the last term is $o(1)$. Combining this with \eqref{eq:change-prob-2} proves Lemma \ref{lem:tube-typ}.
\end{proof}

The next lemma is a variant of Lemma \ref{lem:tube-typ}. Recall from Definition \ref{def:Anx-def-2} that starting from a configuration in $\mathcal K_N^\ell$, if the process stays in $\mathcal A_N$, its first return to $\mathcal E_N$ is either at $\mathcal K_N^\ell$, $\mathcal J_N^\ell$, or $\mathcal E_N^{\ell-1}$. We argue that it hits either $\mathcal J_N^\ell$ or $\mathcal E_N^{\ell-1}$ before time $\mathfrak T_N$ w.h.p.
Let $\color{blue} H_A$ be the first hitting time of a set $A$.

\begin{lem}
\label{lem:tube-typ-2}For each $\ell \in \llbracket 1,k \rrbracket$,
\begin{equation}\label{eq:tube-typ-2}
\lim_{N\to\infty} \sup_{\xi \in \mathcal K_N^\ell} \mathbb{P}_{\xi}^{N} \left( H_{ \mathcal J_N^\ell \cup \mathcal E_N^{\ell-1} } > \mathfrak T_N \right) = 0.
\end{equation}
\end{lem}

\begin{proof}
As explained in the proof of Lemma \ref{lem:tube-typ}, we may bound the probability in \eqref{eq:tube-typ-2} above by the probability that an independent clock of rate $2N^3$ rings, during when the process $\sigma_N(t)$ stays in $\mathcal B_N$, but now before when the process hits $\mathcal J_N^\ell \cup \mathcal E_N^{\ell-1}$.

Recall from \eqref{eq:sigma-stopping} and \eqref{eq:nu-stopping} the definition of the stopping times $\mathfrak t_i , \mathfrak s_i$, $i \ge 0$.
Define ${\color{blue} \mathfrak n'} :=\inf \, \{ i \ge 1 : \sigma_N ( \mathfrak t_i ) \in \mathcal J_N^\ell \cup \mathcal E_N^{\ell-1} \}$. Then,
\begin{equation}\label{eq:3.4-1}
\mathbb{P}_{\xi}^{N} \left( H_{ \mathcal J_N^\ell \cup \mathcal E_N^{\ell-1} } > \mathfrak T_N \right)
\le {\rm E}^{{\bf P}_\xi^N} \left[ 1- e^{- 2N^3 \sum_{i=1}^{\mathfrak n'} ( \mathfrak t_i - \mathfrak s_{i-1} ) } \right] ,
\end{equation}
where, recall, ${\bf P}_\xi^N$ denotes the law of the restricted process $\sigma_N(t)$ on $\mathcal A_N$ starting from $\xi \in \mathcal K_N^\ell$.

To calculate the distribution of $\mathfrak n'$, note that
starting from any $\zeta \in \mathcal K_N^\ell$, there
exists at least one pair of occupied sites $x,z\in\mathbb{T}_{L}$ of $\zeta$
such that $x,y,z$ are consecutive sites in $\mathbb{T}_{L}$. The
probability that the next configuration $\eta$ is obtained by a particle's jump of type either $x \to y$ or $z \to y$ is exactly
\[
\frac{(\zeta_{x}+\zeta_{z}) \, \theta_N d_{N}}{2N^3}=\frac{\zeta_{x}+\zeta_{z}}{2N},
\]
since the total holding rate at $\zeta$ is $2N^3$ and
the two jumps that send a particle from $\{x,z\}$ to $y$ 
have total rate $(\zeta_{x}+\zeta_{z}) \, \theta_N d_{N}$. By Lemma \ref{lem:B4} via the coupling presented in \eqref{eq:coupling-2},
from that configuration $\eta$, 
the probability to reach $\mathcal E_N^{\ell-1}$ (and thus $\mathcal E_N^{\ell-1} \cup \mathcal J_N^\ell$)
at time $\mathfrak t_1$ is bounded below by
\[
\frac{1}{(\zeta_{x}+\zeta_{z})\,e^{2d_{N}(1+\log N)}}.
\]
Thus,
\[
\inf_{\zeta \in \mathcal K_N^\ell} {\bf P}_{\zeta}^{N} \left[\sigma_{N}(\mathfrak t_1)\in \mathcal E_N^{\ell-1} \cup \mathcal J_N^\ell \right]
\ge \frac{\zeta_x + \zeta_z}{2N} \times \frac{1}{(\zeta_x + \zeta_z) \, e^{2d_N (1+\log N)}} = \frac{1}{2N\,e^{2d_{N}(1+\log N)}}.
\]
Therefore, by the strong Markov property,
\begin{equation}
{\bf P}_{\xi}^{N}\,[\mathfrak n' > m]
\le \left(1 - \inf_{\zeta \in \mathcal K_N^\ell} {\bf P}_{\zeta}^{N} \left[\sigma_{N}(\mathfrak t_1)\in \mathcal E_N^{\ell-1} \cup \mathcal J_N^\ell \right] \right)^{m}
\le \left(1-\frac{1}{2N\,e^{2d_{N}(1+\log N)}}\right)^{m}.\label{eq:H-UB}
\end{equation}
By taking integers ${\color{blue} m_{2,N}} \simeq N \log N$, we then obtain via \eqref{eq:dN-cond} that
\begin{equation}\label{eq:G-bound}
\sup_{\xi \in \mathcal K_N^\ell} {\bf P}_\xi^N [ \mathfrak n' > m_{2,N} ] \le
\left(1-\frac{1}{2N\,e^{2d_{N}(1+\log N)}}\right)^{m_{2,N}} \simeq \left(1-\frac{1}{2N\,e^{2d_{N}(1+\log N)}}\right)^{N\log N} = o(1).
\end{equation}
Thus, as it was done in \eqref{eq:change-prob-2}, we substitute this to \eqref{eq:3.4-1} and obtain
\begin{equation}\label{eq:3.4-2}
\begin{aligned}
\mathbb{P}_{\xi}^{N} \left( H_{ \mathcal J_N^\ell \cup \mathcal E_N^{\ell-1} } > \mathfrak T_N \right) & \le {\rm E}^{{\bf P}_\xi^N} \left[ \left( 1- e^{- 2N^3 \sum_{i=1}^{\mathfrak n'} (\mathfrak t_i - \mathfrak s_{i-1}) } \right) {\bf 1} \{ \mathfrak n' \le m_{2,N} \} \right] + {\bf P}_\xi^N ( \mathfrak n' > m_{2,N} ) \\
& \le {\rm E}^{{\bf P}_\xi^N} \left[ 1- e^{- 2N^3 \sum_{i=1}^{m_{2,N}} (\mathfrak t_i - \mathfrak s_{i-1}) } \right] + o(1) ,
\end{aligned}
\end{equation}
where the error is uniform over all $\xi \in \mathcal K_N^\ell$. For the expectation on the right-hand side, we use \eqref{eq:expect-bound} to bound it:
\[
{\rm E}^{{\bf P}_\xi^N} \left[ 1- e^{- 2N^3 \sum_{i=1}^{m_{2,N}} (\mathfrak t_i - \mathfrak s_{i-1}) } \right]
\le 4m_{2,N} k d_N \,e^{d_{N}(1+\log N)}\,(1+\log N).
\]
Since $d_N N \log^2 N = o(1)$ by \eqref{eq:dN-cond}, we conclude the proof via \eqref{eq:3.4-2}.
\end{proof}

Next, we present a lower bound of the probability of hitting $ \mathcal E_N^{\ell-1} $
before $ \mathcal J_N^\ell $, starting from a configuration $\xi \in \mathcal K_N^\ell$.
This constant $p$, which depends on $k$, will appear later in \eqref{eq:key-2}.

\begin{lem}
\label{lem:escape-p}
There exists a universal constant $p = p(k) \in(0,1)$
such that
\[
\liminf_{N\to\infty} \inf_{\xi \in \mathcal K_N^\ell} \mathbb{P}_{\xi}^{N}\left[H_{\mathcal E_N^{\ell-1} }<H_{\mathcal J_N^\ell }\right] \ge p.
\]
\end{lem}

\begin{proof}
Recall that ${\bf P}_{\xi}^{N}$ is the law of the
restricted process in $\mathcal A_N$. By Lemma \ref{lem:tube-typ-2},
\begin{equation}
\inf_{\xi \in \mathcal K_N^\ell} \mathbb{P}_{\xi}^{N}\left[H_{\mathcal E_N^{\ell-1} }<H_{\mathcal J_N^\ell }\right] \ge \inf_{\xi \in \mathcal K_N^\ell} {\bf P}_{\xi}^{N}\left[H_{\mathcal E_N^{\ell-1} }<H_{\mathcal J_N^\ell }\right] +o (1).\label{eq:stays-inside}
\end{equation}

\begin{figure}
\begin{tikzpicture}[scale=0.7]
\draw[densely dotted,->] (2.5,0)--(4,-1); \draw[densely dotted,->] (2.5,0.5)--(4,1.5);
\draw[densely dotted,->] (0.25,2)--(1,3); \draw[densely dotted,->] (-0.25,2)--(-1,3);
\draw[densely dotted,->] (-2.5,0)--(-4,-1); \draw[densely dotted,->] (-2.5,0.5)--(-4,1.5);
\draw[very thick,-latex] (0,-1.5)--(0,-2.5);
\draw (3.25,-0.5) node[above right]{\footnotesize $\simeq d_N$}; \draw (3.25,1) node[below right]{\footnotesize $\simeq d_N$};
\draw (0.625,2.5) node[below right]{\footnotesize $\simeq d_N$}; \draw (-0.625,2.5) node[above right]{\footnotesize $\simeq d_N$};
\draw (-3.25,-0.5) node[below right]{\footnotesize $\simeq d_N$}; \draw (-3.25,1) node[above right]{\footnotesize $\simeq d_N$};

\draw[] (0,0) ellipse (2 and 1);
\foreach \i in {0,...,15} { \draw[thick] ({2*cos(\i*22.5)},{1*sin(\i*22.5)-0.1})--({2*cos(\i*22.5)},{1*sin(\i*22.5)}); }

\foreach \i in {1} { \foreach \j in {0,...,3} { \draw[very thick,blue] ({2*cos(\i*22.5)},{1*sin(\i*22.5)-0.2})--({2*cos(\i*22.5)},{1*sin(\i*22.5)});
\fill[blue!50!white] ({2*cos(\i*22.5)},{1*sin(\i*22.5)+0.15+0.3*\j}) circle (0.15); \draw[blue] ({2*cos(\i*22.5)},{1*sin(\i*22.5)+0.15+0.3*\j}) circle (0.15); } }
\foreach \i in {6} { \foreach \j in {0,...,2} { \draw[very thick,teal] ({2*cos(\i*22.5)},{1*sin(\i*22.5)-0.2})--({2*cos(\i*22.5)},{1*sin(\i*22.5)});
\fill[teal!50!white] ({2*cos(\i*22.5)},{1*sin(\i*22.5)+0.15+0.3*\j}) circle (0.15); \draw[teal] ({2*cos(\i*22.5)},{1*sin(\i*22.5)+0.15+0.3*\j}) circle (0.15); } }
\foreach \i in {11} { \foreach \j in {0,...,4} { \draw[very thick,red] ({2*cos(\i*22.5)},{1*sin(\i*22.5)-0.2})--({2*cos(\i*22.5)},{1*sin(\i*22.5)});
\fill[red!50!white] ({2*cos(\i*22.5)},{1*sin(\i*22.5)+0.15+0.3*\j}) circle (0.15); \draw[red] ({2*cos(\i*22.5)},{1*sin(\i*22.5)+0.15+0.3*\j}) circle (0.15); } }
\foreach \i in {13} { \foreach \j in {0,...,3} { \draw[very thick,red] ({2*cos(\i*22.5)},{1*sin(\i*22.5)-0.2})--({2*cos(\i*22.5)},{1*sin(\i*22.5)});
\fill[red!50!white] ({2*cos(\i*22.5)},{1*sin(\i*22.5)+0.15+0.3*\j}) circle (0.15); \draw[red] ({2*cos(\i*22.5)},{1*sin(\i*22.5)+0.15+0.3*\j}) circle (0.15); } }

\begin{scope}[shift={(6,-1)},scale=0.7]
\draw[] (0,0) ellipse (2 and 1);
\foreach \i in {0,...,15} { \draw[thick] ({2*cos(\i*22.5)},{1*sin(\i*22.5)-0.1})--({2*cos(\i*22.5)},{1*sin(\i*22.5)}); }

\foreach \i in {0} { \foreach \j in {0,...,3} { \draw[very thick,blue] ({2*cos(\i*22.5)},{1*sin(\i*22.5)-0.2})--({2*cos(\i*22.5)},{1*sin(\i*22.5)});
\fill[blue!50!white] ({2*cos(\i*22.5)},{1*sin(\i*22.5)+0.15+0.3*\j}) circle (0.15); \draw[blue] ({2*cos(\i*22.5)},{1*sin(\i*22.5)+0.15+0.3*\j}) circle (0.15); } }
\foreach \i in {6} { \foreach \j in {0,...,2} { \draw[very thick,teal] ({2*cos(\i*22.5)},{1*sin(\i*22.5)-0.2})--({2*cos(\i*22.5)},{1*sin(\i*22.5)});
\fill[teal!50!white] ({2*cos(\i*22.5)},{1*sin(\i*22.5)+0.15+0.3*\j}) circle (0.15); \draw[teal] ({2*cos(\i*22.5)},{1*sin(\i*22.5)+0.15+0.3*\j}) circle (0.15); } }
\foreach \i in {11} { \foreach \j in {0,...,4} { \draw[very thick,red] ({2*cos(\i*22.5)},{1*sin(\i*22.5)-0.2})--({2*cos(\i*22.5)},{1*sin(\i*22.5)});
\fill[red!50!white] ({2*cos(\i*22.5)},{1*sin(\i*22.5)+0.15+0.3*\j}) circle (0.15); \draw[red] ({2*cos(\i*22.5)},{1*sin(\i*22.5)+0.15+0.3*\j}) circle (0.15); } }
\foreach \i in {13} { \foreach \j in {0,...,3} { \draw[very thick,red] ({2*cos(\i*22.5)},{1*sin(\i*22.5)-0.2})--({2*cos(\i*22.5)},{1*sin(\i*22.5)});
\fill[red!50!white] ({2*cos(\i*22.5)},{1*sin(\i*22.5)+0.15+0.3*\j}) circle (0.15); \draw[red] ({2*cos(\i*22.5)},{1*sin(\i*22.5)+0.15+0.3*\j}) circle (0.15); } }
\end{scope}

\begin{scope}[shift={(6,2)},scale=0.7]
\draw[] (0,0) ellipse (2 and 1);
\foreach \i in {0,...,15} { \draw[thick] ({2*cos(\i*22.5)},{1*sin(\i*22.5)-0.1})--({2*cos(\i*22.5)},{1*sin(\i*22.5)}); }

\foreach \i in {2} { \foreach \j in {0,...,3} { \draw[very thick,blue] ({2*cos(\i*22.5)},{1*sin(\i*22.5)-0.2})--({2*cos(\i*22.5)},{1*sin(\i*22.5)});
\fill[blue!50!white] ({2*cos(\i*22.5)},{1*sin(\i*22.5)+0.15+0.3*\j}) circle (0.15); \draw[blue] ({2*cos(\i*22.5)},{1*sin(\i*22.5)+0.15+0.3*\j}) circle (0.15); } }
\foreach \i in {6} { \foreach \j in {0,...,2} { \draw[very thick,teal] ({2*cos(\i*22.5)},{1*sin(\i*22.5)-0.2})--({2*cos(\i*22.5)},{1*sin(\i*22.5)});
\fill[teal!50!white] ({2*cos(\i*22.5)},{1*sin(\i*22.5)+0.15+0.3*\j}) circle (0.15); \draw[teal] ({2*cos(\i*22.5)},{1*sin(\i*22.5)+0.15+0.3*\j}) circle (0.15); } }
\foreach \i in {11} { \foreach \j in {0,...,4} { \draw[very thick,red] ({2*cos(\i*22.5)},{1*sin(\i*22.5)-0.2})--({2*cos(\i*22.5)},{1*sin(\i*22.5)});
\fill[red!50!white] ({2*cos(\i*22.5)},{1*sin(\i*22.5)+0.15+0.3*\j}) circle (0.15); \draw[red] ({2*cos(\i*22.5)},{1*sin(\i*22.5)+0.15+0.3*\j}) circle (0.15); } }
\foreach \i in {13} { \foreach \j in {0,...,3} { \draw[very thick,red] ({2*cos(\i*22.5)},{1*sin(\i*22.5)-0.2})--({2*cos(\i*22.5)},{1*sin(\i*22.5)});
\fill[red!50!white] ({2*cos(\i*22.5)},{1*sin(\i*22.5)+0.15+0.3*\j}) circle (0.15); \draw[red] ({2*cos(\i*22.5)},{1*sin(\i*22.5)+0.15+0.3*\j}) circle (0.15); } }
\end{scope}

\begin{scope}[shift={(2.5,4)},scale=0.7]
\draw[] (0,0) ellipse (2 and 1);
\foreach \i in {0,...,15} { \draw[thick] ({2*cos(\i*22.5)},{1*sin(\i*22.5)-0.1})--({2*cos(\i*22.5)},{1*sin(\i*22.5)}); }

\foreach \i in {1} { \foreach \j in {0,...,3} { \draw[very thick,blue] ({2*cos(\i*22.5)},{1*sin(\i*22.5)-0.2})--({2*cos(\i*22.5)},{1*sin(\i*22.5)});
\fill[blue!50!white] ({2*cos(\i*22.5)},{1*sin(\i*22.5)+0.15+0.3*\j}) circle (0.15); \draw[blue] ({2*cos(\i*22.5)},{1*sin(\i*22.5)+0.15+0.3*\j}) circle (0.15); } }
\foreach \i in {5} { \foreach \j in {0,...,2} { \draw[very thick,teal] ({2*cos(\i*22.5)},{1*sin(\i*22.5)-0.2})--({2*cos(\i*22.5)},{1*sin(\i*22.5)});
\fill[teal!50!white] ({2*cos(\i*22.5)},{1*sin(\i*22.5)+0.15+0.3*\j}) circle (0.15); \draw[teal] ({2*cos(\i*22.5)},{1*sin(\i*22.5)+0.15+0.3*\j}) circle (0.15); } }
\foreach \i in {11} { \foreach \j in {0,...,4} { \draw[very thick,red] ({2*cos(\i*22.5)},{1*sin(\i*22.5)-0.2})--({2*cos(\i*22.5)},{1*sin(\i*22.5)});
\fill[red!50!white] ({2*cos(\i*22.5)},{1*sin(\i*22.5)+0.15+0.3*\j}) circle (0.15); \draw[red] ({2*cos(\i*22.5)},{1*sin(\i*22.5)+0.15+0.3*\j}) circle (0.15); } }
\foreach \i in {13} { \foreach \j in {0,...,3} { \draw[very thick,red] ({2*cos(\i*22.5)},{1*sin(\i*22.5)-0.2})--({2*cos(\i*22.5)},{1*sin(\i*22.5)});
\fill[red!50!white] ({2*cos(\i*22.5)},{1*sin(\i*22.5)+0.15+0.3*\j}) circle (0.15); \draw[red] ({2*cos(\i*22.5)},{1*sin(\i*22.5)+0.15+0.3*\j}) circle (0.15); } }
\end{scope}

\begin{scope}[shift={(-2.5,4)},scale=0.7]
\draw[] (0,0) ellipse (2 and 1);
\foreach \i in {0,...,15} { \draw[thick] ({2*cos(\i*22.5)},{1*sin(\i*22.5)-0.1})--({2*cos(\i*22.5)},{1*sin(\i*22.5)}); }

\foreach \i in {1} { \foreach \j in {0,...,3} { \draw[very thick,blue] ({2*cos(\i*22.5)},{1*sin(\i*22.5)-0.2})--({2*cos(\i*22.5)},{1*sin(\i*22.5)});
\fill[blue!50!white] ({2*cos(\i*22.5)},{1*sin(\i*22.5)+0.15+0.3*\j}) circle (0.15); \draw[blue] ({2*cos(\i*22.5)},{1*sin(\i*22.5)+0.15+0.3*\j}) circle (0.15); } }
\foreach \i in {7} { \foreach \j in {0,...,2} { \draw[very thick,teal] ({2*cos(\i*22.5)},{1*sin(\i*22.5)-0.2})--({2*cos(\i*22.5)},{1*sin(\i*22.5)});
\fill[teal!50!white] ({2*cos(\i*22.5)},{1*sin(\i*22.5)+0.15+0.3*\j}) circle (0.15); \draw[teal] ({2*cos(\i*22.5)},{1*sin(\i*22.5)+0.15+0.3*\j}) circle (0.15); } }
\foreach \i in {11} { \foreach \j in {0,...,4} { \draw[very thick,red] ({2*cos(\i*22.5)},{1*sin(\i*22.5)-0.2})--({2*cos(\i*22.5)},{1*sin(\i*22.5)});
\fill[red!50!white] ({2*cos(\i*22.5)},{1*sin(\i*22.5)+0.15+0.3*\j}) circle (0.15); \draw[red] ({2*cos(\i*22.5)},{1*sin(\i*22.5)+0.15+0.3*\j}) circle (0.15); } }
\foreach \i in {13} { \foreach \j in {0,...,3} { \draw[very thick,red] ({2*cos(\i*22.5)},{1*sin(\i*22.5)-0.2})--({2*cos(\i*22.5)},{1*sin(\i*22.5)});
\fill[red!50!white] ({2*cos(\i*22.5)},{1*sin(\i*22.5)+0.15+0.3*\j}) circle (0.15); \draw[red] ({2*cos(\i*22.5)},{1*sin(\i*22.5)+0.15+0.3*\j}) circle (0.15); } }
\end{scope}

\begin{scope}[shift={(-6,2)},scale=0.7]
\draw[] (0,0) ellipse (2 and 1);
\foreach \i in {0,...,15} { \draw[thick] ({2*cos(\i*22.5)},{1*sin(\i*22.5)-0.1})--({2*cos(\i*22.5)},{1*sin(\i*22.5)}); }

\foreach \i in {1} { \foreach \j in {0,...,3} { \draw[very thick,blue] ({2*cos(\i*22.5)},{1*sin(\i*22.5)-0.2})--({2*cos(\i*22.5)},{1*sin(\i*22.5)});
\fill[blue!50!white] ({2*cos(\i*22.5)},{1*sin(\i*22.5)+0.15+0.3*\j}) circle (0.15); \draw[blue] ({2*cos(\i*22.5)},{1*sin(\i*22.5)+0.15+0.3*\j}) circle (0.15); } }
\foreach \i in {6} { \foreach \j in {0,...,2} { \draw[very thick,teal] ({2*cos(\i*22.5)},{1*sin(\i*22.5)-0.2})--({2*cos(\i*22.5)},{1*sin(\i*22.5)});
\fill[teal!50!white] ({2*cos(\i*22.5)},{1*sin(\i*22.5)+0.15+0.3*\j}) circle (0.15); \draw[teal] ({2*cos(\i*22.5)},{1*sin(\i*22.5)+0.15+0.3*\j}) circle (0.15); } }
\foreach \i in {10} { \foreach \j in {0,...,4} { \draw[very thick,red] ({2*cos(\i*22.5)},{1*sin(\i*22.5)-0.2})--({2*cos(\i*22.5)},{1*sin(\i*22.5)});
\fill[red!50!white] ({2*cos(\i*22.5)},{1*sin(\i*22.5)+0.15+0.3*\j}) circle (0.15); \draw[red] ({2*cos(\i*22.5)},{1*sin(\i*22.5)+0.15+0.3*\j}) circle (0.15); } }
\foreach \i in {13} { \foreach \j in {0,...,3} { \draw[very thick,red] ({2*cos(\i*22.5)},{1*sin(\i*22.5)-0.2})--({2*cos(\i*22.5)},{1*sin(\i*22.5)});
\fill[red!50!white] ({2*cos(\i*22.5)},{1*sin(\i*22.5)+0.15+0.3*\j}) circle (0.15); \draw[red] ({2*cos(\i*22.5)},{1*sin(\i*22.5)+0.15+0.3*\j}) circle (0.15); } }
\end{scope}

\begin{scope}[shift={(-6,-1)},scale=0.7]
\draw[] (0,0) ellipse (2 and 1);
\foreach \i in {0,...,15} { \draw[thick] ({2*cos(\i*22.5)},{1*sin(\i*22.5)-0.1})--({2*cos(\i*22.5)},{1*sin(\i*22.5)}); }

\foreach \i in {1} { \foreach \j in {0,...,3} { \draw[very thick,blue] ({2*cos(\i*22.5)},{1*sin(\i*22.5)-0.2})--({2*cos(\i*22.5)},{1*sin(\i*22.5)});
\fill[blue!50!white] ({2*cos(\i*22.5)},{1*sin(\i*22.5)+0.15+0.3*\j}) circle (0.15); \draw[blue] ({2*cos(\i*22.5)},{1*sin(\i*22.5)+0.15+0.3*\j}) circle (0.15); } }
\foreach \i in {6} { \foreach \j in {0,...,2} { \draw[very thick,teal] ({2*cos(\i*22.5)},{1*sin(\i*22.5)-0.2})--({2*cos(\i*22.5)},{1*sin(\i*22.5)});
\fill[teal!50!white] ({2*cos(\i*22.5)},{1*sin(\i*22.5)+0.15+0.3*\j}) circle (0.15); \draw[teal] ({2*cos(\i*22.5)},{1*sin(\i*22.5)+0.15+0.3*\j}) circle (0.15); } }
\foreach \i in {11} { \foreach \j in {0,...,4} { \draw[very thick,red] ({2*cos(\i*22.5)},{1*sin(\i*22.5)-0.2})--({2*cos(\i*22.5)},{1*sin(\i*22.5)});
\fill[red!50!white] ({2*cos(\i*22.5)},{1*sin(\i*22.5)+0.15+0.3*\j}) circle (0.15); \draw[red] ({2*cos(\i*22.5)},{1*sin(\i*22.5)+0.15+0.3*\j}) circle (0.15); } }
\foreach \i in {14} { \foreach \j in {0,...,3} { \draw[very thick,red] ({2*cos(\i*22.5)},{1*sin(\i*22.5)-0.2})--({2*cos(\i*22.5)},{1*sin(\i*22.5)});
\fill[red!50!white] ({2*cos(\i*22.5)},{1*sin(\i*22.5)+0.15+0.3*\j}) circle (0.15); \draw[red] ({2*cos(\i*22.5)},{1*sin(\i*22.5)+0.15+0.3*\j}) circle (0.15); } }
\end{scope}

\begin{scope}[shift={(-4,-4)},scale=0.7]
\draw[] (0,0) ellipse (2 and 1);
\foreach \i in {0,...,15} { \draw[thick] ({2*cos(\i*22.5)},{1*sin(\i*22.5)-0.1})--({2*cos(\i*22.5)},{1*sin(\i*22.5)}); }

\foreach \i in {1} { \foreach \j in {0,...,3} { \draw[very thick,blue] ({2*cos(\i*22.5)},{1*sin(\i*22.5)-0.2})--({2*cos(\i*22.5)},{1*sin(\i*22.5)});
\fill[blue!50!white] ({2*cos(\i*22.5)},{1*sin(\i*22.5)+0.15+0.3*\j}) circle (0.15); \draw[blue] ({2*cos(\i*22.5)},{1*sin(\i*22.5)+0.15+0.3*\j}) circle (0.15); } }
\foreach \i in {6} { \foreach \j in {0,...,2} { \draw[very thick,teal] ({2*cos(\i*22.5)},{1*sin(\i*22.5)-0.2})--({2*cos(\i*22.5)},{1*sin(\i*22.5)});
\fill[teal!50!white] ({2*cos(\i*22.5)},{1*sin(\i*22.5)+0.15+0.3*\j}) circle (0.15); \draw[teal] ({2*cos(\i*22.5)},{1*sin(\i*22.5)+0.15+0.3*\j}) circle (0.15); } }
\foreach \i in {11} { \foreach \j in {0,...,8} { \draw[very thick,red] ({2*cos(\i*22.5)},{1*sin(\i*22.5)-0.2})--({2*cos(\i*22.5)},{1*sin(\i*22.5)});
\fill[red!50!white] ({2*cos(\i*22.5)},{1*sin(\i*22.5)+0.15+0.3*\j}) circle (0.15); \draw[red] ({2*cos(\i*22.5)},{1*sin(\i*22.5)+0.15+0.3*\j}) circle (0.15); } }
\end{scope}

\begin{scope}[shift={(0,-4)},scale=0.7]
\draw[] (0,0) ellipse (2 and 1);
\foreach \i in {0,...,15} { \draw[thick] ({2*cos(\i*22.5)},{1*sin(\i*22.5)-0.1})--({2*cos(\i*22.5)},{1*sin(\i*22.5)}); }

\foreach \i in {1} { \foreach \j in {0,...,3} { \draw[very thick,blue] ({2*cos(\i*22.5)},{1*sin(\i*22.5)-0.2})--({2*cos(\i*22.5)},{1*sin(\i*22.5)});
\fill[blue!50!white] ({2*cos(\i*22.5)},{1*sin(\i*22.5)+0.15+0.3*\j}) circle (0.15); \draw[blue] ({2*cos(\i*22.5)},{1*sin(\i*22.5)+0.15+0.3*\j}) circle (0.15); } }
\foreach \i in {6} { \foreach \j in {0,...,2} { \draw[very thick,teal] ({2*cos(\i*22.5)},{1*sin(\i*22.5)-0.2})--({2*cos(\i*22.5)},{1*sin(\i*22.5)});
\fill[teal!50!white] ({2*cos(\i*22.5)},{1*sin(\i*22.5)+0.15+0.3*\j}) circle (0.15); \draw[teal] ({2*cos(\i*22.5)},{1*sin(\i*22.5)+0.15+0.3*\j}) circle (0.15); } }
\foreach \i in {12} { \foreach \j in {0,...,8} { \draw[very thick,red] ({2*cos(\i*22.5)},{1*sin(\i*22.5)-0.2})--({2*cos(\i*22.5)},{1*sin(\i*22.5)});
\fill[red!50!white] ({2*cos(\i*22.5)},{1*sin(\i*22.5)+0.15+0.3*\j}) circle (0.15); \draw[red] ({2*cos(\i*22.5)},{1*sin(\i*22.5)+0.15+0.3*\j}) circle (0.15); } }
\end{scope}

\begin{scope}[shift={(4,-4)},scale=0.7]
\draw[] (0,0) ellipse (2 and 1);
\foreach \i in {0,...,15} { \draw[thick] ({2*cos(\i*22.5)},{1*sin(\i*22.5)-0.1})--({2*cos(\i*22.5)},{1*sin(\i*22.5)}); }

\foreach \i in {1} { \foreach \j in {0,...,3} { \draw[very thick,blue] ({2*cos(\i*22.5)},{1*sin(\i*22.5)-0.2})--({2*cos(\i*22.5)},{1*sin(\i*22.5)});
\fill[blue!50!white] ({2*cos(\i*22.5)},{1*sin(\i*22.5)+0.15+0.3*\j}) circle (0.15); \draw[blue] ({2*cos(\i*22.5)},{1*sin(\i*22.5)+0.15+0.3*\j}) circle (0.15); } }
\foreach \i in {6} { \foreach \j in {0,...,2} { \draw[very thick,teal] ({2*cos(\i*22.5)},{1*sin(\i*22.5)-0.2})--({2*cos(\i*22.5)},{1*sin(\i*22.5)});
\fill[teal!50!white] ({2*cos(\i*22.5)},{1*sin(\i*22.5)+0.15+0.3*\j}) circle (0.15); \draw[teal] ({2*cos(\i*22.5)},{1*sin(\i*22.5)+0.15+0.3*\j}) circle (0.15); } }
\foreach \i in {13} { \foreach \j in {0,...,8} { \draw[very thick,red] ({2*cos(\i*22.5)},{1*sin(\i*22.5)-0.2})--({2*cos(\i*22.5)},{1*sin(\i*22.5)});
\fill[red!50!white] ({2*cos(\i*22.5)},{1*sin(\i*22.5)+0.15+0.3*\j}) circle (0.15); \draw[red] ({2*cos(\i*22.5)},{1*sin(\i*22.5)+0.15+0.3*\j}) circle (0.15); } }
\end{scope}
\end{tikzpicture}\caption{\label{fig2.3}Trace process $\widetilde\sigma_{N}(t)$ on $\mathcal E_N$.
Starting from a configuration in $\mathcal K_N^\ell$
at the center, it either visits another configuration in $\mathcal K_N^\ell \cup \mathcal J_N^\ell$ (top six configurations) or visits
$\mathcal{E}_N^{\ell-1}$ (bottom three configurations). Each trace
jump rate to a configuration in $\mathcal J_N^\ell$ is
asymptotically equal to $N^2$ (cf. \eqref{eq:trace-1}), whereas
the total trace jump rate to $\mathcal E_N^{\ell-1}$ is asymptotically
at least $N^2$ (cf. \eqref{eq:trace-3}).}
\end{figure}

Consider the trace process $\color{blue} \widetilde\sigma_{N}(t)$ of $\sigma_{N}(t)$
on the subset $\mathcal E_N = \mathcal A_N \setminus \mathcal B_N$ (cf. Figure \ref{fig2.3}).\footnote{Mind that this process $\widetilde\sigma_N(t)$ is different from the original trace process $\eta_N^{\mathcal E_N}(t)$, since it has been obtained by first annihilating the jump rates, as done in \eqref{eq:rN'-def}, and then tracing on $\mathcal E_N$.}
Denote by ${\color{blue} \widetilde{r}_{N}} (\cdot,\cdot)$
the jump rates of $\widetilde\sigma_N(t)$.
By \cite[Corollary 6.2]{BL10}, for any $\zeta,\zeta' \in \mathcal E_N$,
\begin{equation}\label{eq:trace-rate}
\widetilde{r}_{N}(\zeta,\zeta') = \sum_{\eta\in \mathcal A_N} r_{N}(\zeta,\eta)\,{\bf P}_{\eta}^{N}\left[H_{ \{ \zeta' \} }=H_{\mathcal E_N}\right].
\end{equation}
To reach a configuration $\zeta' \in \mathcal J_N^\ell$
from $\zeta \in \mathcal K_N^\ell$, the
configuration $\eta$ reached after the first jump must be obtained from $\zeta$ by a particle's
jump $x \to y$ which is isolated from any other occupied sites. In this case, by the coupling presented at \eqref{eq:coupling-1} and Lemma \ref{lem:B1}
\[
{\bf P}_\eta^N \left[ H_{\{ \zeta' \}} = H_{\mathcal E_N} \right] \le  \frac1{\zeta_x} \, e^{d_N (1+ \log N)}.
\]
Substituting this to \eqref{eq:trace-rate},
\begin{equation}
\widetilde{r}_{N}(\zeta,\zeta') \le \zeta_x \theta_N d_N \times \frac1{\zeta_x} \, e^{d_N (1+\log N)} = N^2 \, e^{d_N(1+\log N)} .\label{eq:trace-1}
\end{equation}
There are at most $2\ell-2$ such distinct configurations $\zeta'\in\mathcal J_N^\ell$.
This gives
\begin{equation}
\widetilde{r}_{N} ( \zeta,\mathcal J_N^\ell ) :=\sum_{\zeta'\in \mathcal J_N^\ell}\widetilde{r}_{N}(\zeta,\zeta') \le (2\ell-2) N^2 \, e^{d_N (1+\log N)} .\label{eq:trace-2}
\end{equation}
On the other hand, suppose that $\eta$ is obtained from $\zeta \in \mathcal K_N^\ell$
with a particle's jump $x \to y$ or $z \to y$ such that $x,y,z$ are three consecutive elements and $\zeta_x, \zeta_z \ge 1$.
The total jump rate is $(\zeta_x + \zeta_z) \, \theta_N d_N$. Then, using the same
coupling as given in \eqref{eq:coupling-2}, via Lemma \ref{lem:B4}, we deduce that
\[
{\bf P}_{\eta}^{N}\left[H_{ \mathcal E_N^{\ell-1} } = H_{ \mathcal E_N }\right] \ge \frac{1}{(\zeta_x + \zeta_z)\,e^{2d_{N}(1+\log N)}}.
\]
Thus, again by \eqref{eq:trace-rate},
\begin{equation}
\widetilde{r}_{N} ( \zeta,\mathcal E_N^{\ell-1} ) := \sum_{\zeta' \in \mathcal E_N^{\ell-1}} \widetilde{r}_{N}(\zeta,\zeta')\ge\frac{N^2}{e^{2d_{N}(1+\log N)}}.\label{eq:trace-3}
\end{equation}
Combining \eqref{eq:trace-2} and \eqref{eq:trace-3}, along with
the strong Markov property, we obtain for any $\xi \in \mathcal K_N^\ell$ that
\begin{equation}
\begin{aligned}
{\bf P}_{\xi}^{N}\,\Big [H_{ \mathcal E_N^{\ell-1}} & <  H_{ \mathcal J_N^\ell } \Big]\,
\ge \inf_{\zeta \in \mathcal K_N^\ell } \frac{\widetilde{r}_{N} ( \zeta, \mathcal E_N^{\ell-1} ) } {\widetilde{r}_{N} ( \zeta, \mathcal J_N^\ell \cup \mathcal E_N^{\ell-1} ) } \\
& \ge\frac{e^{-2d_{N}(1+\log N)}}{e^{-2d_{N}(1+\log N)}+(2\ell-2) \, e^{d_N (1+\log N)} } \ge \frac{1}{2\ell},
\end{aligned}
\label{eq:go-forward}
\end{equation}
where the last inequality holds for all sufficiently large $N$ by
\eqref{eq:dN-cond}. Thus, we may take e.g. $p=(2k)^{-1}>0$ such
that, by \eqref{eq:stays-inside} and \eqref{eq:go-forward},
\[
\liminf_{N\to\infty} \inf_{\xi \in \mathcal K_N^\ell } \mathbb{P}_{\xi}^{N}\left[H_{\mathcal E_N^{\ell-1} }<H_{\mathcal J_N^\ell }\right]\ge\frac{1}{2\ell}\ge p.
\]
This concludes the proof.
\end{proof}

We are able to control the hitting time of $ \mathcal J_N^\ell \cup \mathcal E_N^{\ell-1} $
when restricted to the event that the process does not escape $\mathcal A_N$.

\begin{lem}\label{lem:good-hitting-time}
For all fixed $\lambda >0$,
\[
\sup_{ \xi \in \mathcal K_N^\ell } {\rm E}^{\mathbb{P}_{\xi}^{N}} \left[ \left( 1 - e^{-\lambda H_{\mathcal J_N^\ell \cup \mathcal E_N^{\ell-1}}} \right) {\bf 1} \left\{ H_{\mathcal J_N^\ell \cup \mathcal E_N^{\ell-1}} < \mathfrak T_N \right\} \right] = o (1).
\]
\end{lem}

\begin{proof}
According to the notation introduced in the proof of Lemma \ref{lem:tube-typ-2},
\begin{equation}\label{eq:sum-1}
{\rm E}^{\mathbb{P}_{\xi}^{N}} \, \Big[ \left( 1 - e^{-\lambda H_{\mathcal J_N^\ell \cup \mathcal E_N^{\ell-1}}} \right) {\bf 1} \left\{ H_{\mathcal J_N^\ell \cup \mathcal E_N^{\ell-1}} < \mathfrak T_N \right\} \Big] \, 
\le {\rm E}^{{\bf P}_{\xi}^{N}} \left[ 1 - e^{-\lambda \mathfrak t_{\mathfrak n'}} \right] .
\end{equation}
Recall \eqref{eq:G-bound}. Then,
\begin{equation}\label{eq:sum-2}
{\rm E}^{{\bf P}_\xi^N} \left[ 1 - e^{-\lambda \mathfrak t_{\mathfrak n'}} \right] \le {\bf P}_\xi^N ( \mathfrak n' > m_{2,N} )
+ {\rm E}^{{\bf P}_\xi^N} \left[ 1 - e^{-\lambda \mathfrak t_{m_{2,N}}} \right] \le o(1) + \lambda \, {\rm E}^{{\bf P}_\xi^N} [ \mathfrak t_{m_{2,N}} ] ,
\end{equation}
where the error is uniform over $\xi \in \mathcal K_N^\ell$. Next, we may write
\begin{equation}\label{eq:expect-sum}
{\rm E}^{{\bf P}_\xi^N} [ \mathfrak t_{m_{2,N}} ] =
{\rm E}^{{\bf P}_{\xi}^{N}} \left[ \sum_{j=1}^{m_{2,N}}(\mathfrak s_{j-1} - \mathfrak t_{j-1})+\sum_{j=1}^{m_{2,N}}( \mathfrak t_j - \mathfrak s_{j-1}) \right] .
\end{equation}
For any $j \in \llbracket 1,m_{2,N} \rrbracket$, as done in \eqref{eq:RW-bound-1}
and \eqref{eq:RW-bound-2}, by the strong Markov property,
\[
{\rm E}^{{\bf P}_{\xi}^{N}} \left[ \mathfrak s_{j-1} - \mathfrak t_{j-1} \right] \le
\sup_{\zeta \in \mathcal E_N} {\rm E}^{{\bf P}_{\zeta}^{N}} \,[ \mathfrak s_0 ] = \frac{1}{2 N^3}
\]
and
\[
{\rm E}^{{\bf P}_{\xi}^{N}} \left[ \mathfrak t_j - \mathfrak s_{j-1} \right] \le \sup_{ \eta \in \mathcal B_N : \,r_{N} ( \eta, \mathcal E_N ) > 0}
{\rm E}^{{\bf P}_{\eta}^{N}} \left[ \mathfrak t_1 \right]  \le 2 \theta_N^{-1} \,e^{d_{N}(1+\log N)}\,(1+\log N).
\]
Substituting these two bounds to \eqref{eq:expect-sum}, we obtain\footnote{Here, $f_N=O (g_N)$ if $|f_N|\le Cg_N$ where $C$ is a global constant independent of $N$.}
\begin{equation}\label{eq:sum-3}
{\rm E}^{{\bf P}_\xi^N} [ \mathfrak t_{m_{2,N}} ] \le m_{2,N} \left( \frac1{2N^3} + 2 \theta_N^{-1} \,e^{d_{N}(1+\log N)}\,(1+\log N) \right) = O \left( \frac{\log N}{N^2} \right) .
\end{equation}
Here, \eqref{eq:dN-cond} was used. Combining \eqref{eq:sum-1}, \eqref{eq:sum-2}, and \eqref{eq:sum-3} completes the proof.
\end{proof}

Finally, according to the analysis conducted in the previous lemmas, we calculate the trace jump rate function $R_N (\cdot,\cdot)$ of $\eta_N^{\mathcal E_N} (t)$, defined in \eqref{eq:trace}.
Recall Definition \ref{def:typ-jump}.
First, we consider the rates from $\mathcal J_N^\ell$.

\begin{lem}\label{lem:JN-trace-rate}
For each $\xi_{\bm n}^{\bm x} \in \mathcal J_N^\ell$, $\ell \in \llbracket 1,k \rrbracket$, we have
\[
R_N (\xi_{\bm n}^{\bm x} , \xi_{\bm n}^{\bm x \pm \bm e_i}) = N^2 + O(d_N N^3 \log N) \qquad \text{for each} \quad i \in \llbracket 1,\ell \rrbracket.
\]
\end{lem}

\begin{proof}
Recall from \cite[Corollary 6.2]{BL10} that, for $\zeta,\zeta' \in \mathcal E_N$,
\begin{equation}\label{eq:trace-rate-for}
R_N (\zeta,\zeta') = \sum_{\eta\in \mathcal A_N} r_{N}(\zeta,\eta)\,\mathbb P_{\eta}^{N}\left[H_{ \{ \zeta' \} }=H_{\mathcal E_N}\right].
\end{equation}
Recall \eqref{eq:sigma-stopping} and \eqref{eq:nu-stopping}. Using the same idea as in the proof of Lemma \ref{lem:tube-typ}, for each $\eta \in \mathcal A_N$ such that $r_N(\zeta,\eta) > 0$ we have by \eqref{eq:RW-bound-1} and \eqref{eq:RW-bound-2} that
\[
\mathbb P_\eta^N \left[ H_{\mathcal E_N} > \mathfrak T_N \right] \le
{\rm E}^{{\bf P}_\eta^N} \left[ 1 - e^{-2N^3 \mathfrak t_1} \right] \le \frac{4d_N N^3}{r_N(\zeta,\eta)} \, e^{d_N (1+\log N)} \, (1+\log N).
\]
Substituting this to the penultimate identity, by \eqref{eq:dN-cond} we obtain
\begin{equation}\label{eq:trace-cal-JN}
R_N (\zeta,\zeta') = \sum_{\eta\in \mathcal A_N} r_{N}(\zeta,\eta)\,\mathbb P_{\eta}^{N}\left[H_{ \{ \zeta' \} }=H_{\mathcal E_N} < \mathfrak T_N \right] + O(d_N N^3 \log N),
\end{equation}
where it was used that the number of $\eta \in \mathcal A_N$ such that $r_N(\zeta,\eta)>0$ is at most $2k$.

Now, \eqref{eq:trace-cal-JN} gives the desired result. Indeed, from Definition \ref{def:typ-jump} we have
\[
\mathcal N ( \xi_{\bm n}^{\bm x} ) = \bigcup_{i=1}^\ell \{ \xi_{\bm n}^{\bm x \pm \bm e_i} \} .
\]
For each of the $2\ell$ configurations in $\mathcal N(\xi_{\bm n}^{\bm x})$, say $\xi \in \mathcal N(\xi_{\bm n}^{\bm x})$, there exists exactly one $\eta \in \mathcal A_N$ such that the probability in \eqref{eq:trace-cal-JN} is nonzero. From that $\eta$, the same computation as the one performed in \eqref{eq:trace-1} gives that
\[
R_N (\xi_{\bm n}^{\bm x} , \xi) = N^2 \, e^{d_N (1+\log N)} + O (d_N N^3 \log N) = N^2 + O(d_N N^3 \log N),
\]
via \eqref{eq:dN-cond}.
This concludes the proof.
\end{proof}

Next, consider the collection $\mathcal K_N^\ell$.

\begin{lem}\label{lem:KN-trace-rate}
For each $\xi_{\bm n}^{\bm x} \in \mathcal K_N^\ell$,
%\[
%\sum_{\xi \in \mathcal E_N \setminus \{ \xi_{\bm n}^{\bm x} \}} R_N ( \xi_{\bm n}^{\bm x} , \xi ) \le 2N^3,
%\]
\[
R_N (\xi_{\bm n}^{\bm x}, \mathcal J_N^\ell) \le (2\ell-2) N^2 \, e^{d_N(1+\log N)} + O(d_N N^3 \log N)
\]
and
\[
R_N ( \xi_{\bm n}^{\bm x} , \mathcal E_N^{\ell-1}) \ge \frac{N^2}{e^{2d_N(1+\log N)}} + O(d_N N^3 \log N) .
\]
\end{lem}

\begin{proof}
%The first inequality follows directly from \eqref{eq:trace-rate-for} and the fact that the holding rate of the original chain at $\xi_{\bm n}^{\bm x}$ is $2N^3$.
The first inequality follows as done in \eqref{eq:trace-2} and Lemma \ref{lem:JN-trace-rate}, whereas
the second inequality follows as done in \eqref{eq:trace-3} and Lemma \ref{lem:JN-trace-rate}. We omit the details.
\end{proof}

\begin{lem}
\label{lem:escape-p-trace}
For the same constant $p \in (0,1)$ in Lemma \ref{lem:escape-p},
\[
\liminf_{N\to\infty} \inf_{\xi \in \mathcal K_N^\ell} \overline{\mathbb Q}_{\xi}^{N}\left[H_{\mathcal E_N^{\ell-1} }<H_{\mathcal J_N^\ell }\right] \ge p.
\]
\end{lem}

\begin{proof}
This follows directly from the two inequalities in Lemma \ref{lem:KN-trace-rate}.
\end{proof}

\section{\label{sec3}First Resolvent Condition and Proof of Theorem \ref{thm1}}

In this section, we state and prove a specific resolvent condition \cite{KL26,LMS25} which is the key estimate to prove Theorem \ref{thm1}.
Let ${\color{blue}\Delta_N} := \Omega_N \setminus \mathcal E_N$.

\begin{thm}
\label{thm:res-inc}Given $\lambda>0$, denote by $F_{N}$ the unique solution to $(\lambda - \mathcal{L}_{N})F_{N}= {\bf 1}_{\Delta_N}$.
Then,
\[
\lim_{N\to\infty} \sup_{\xi \in \mathcal E_N} | F_N (\xi) | = 0.
\]
\end{thm}

The solution $F_N$ of the resolvent equation in Theorem \ref{thm:res-inc} has a stochastic representation:
\[
F_{N}(\eta) = {\rm E}^{\mathbb{Q}_{\eta}^{N}}\left[\int_{0}^{\infty}e^{-\lambda t}\, {\bf 1}_{\Delta_N} (\eta_{N}(t))\,{\rm d}t\right] .
\]
In particular, $F_N$ has a uniform $L^\infty$ bound:
\begin{equation}
|F_{N}(\eta)| = \int_{0}^{\infty}e^{-\lambda t}\, {\rm d}t \le \frac1{\lambda} \qquad \text{for any} \quad \eta \in \Omega_N . \label{eq:1-5}
\end{equation}

First, we prove Theorem \ref{thm1} assuming that Theorem \ref{thm:res-inc} holds:

\begin{proof}[Proof of Theorem \ref{thm1} under Theorem \ref{thm:res-inc}]
Fix $ \xi \in \mathcal E_N$. By definition of $F_N$,
\[
{\rm E}^{\mathbb{P}_{\xi}^{N}}\left[\int_{0}^{T}{\bf 1}\{\eta_{N}(t)\in\Delta_{N}\}\,{\rm d}t\right]\le e^{\lambda T}\,{\rm E}^{\mathbb{P}_{\xi}^{N}}\left[\int_{0}^{\infty}e^{-\lambda t}\,{\bf 1}\{\eta_{N}(t)\in\Delta_{N}\}\,{\rm d}t\right] = e^{\lambda T} \, F_N(\xi),
\]
since $F_N$ solves $(\lambda-\mathcal{L}_{N})F_{N}={\bf 1}_{\Delta_N}$.
By Theorem \ref{thm:res-inc}, $\lim_{N \to \infty} \sup_{\xi \in \mathcal E_N} | F_N (\xi) | = 0$, which completes the proof.
\end{proof}

The rest of the section is devoted to the proof of Theorem \ref{thm:res-inc}. The
overall strategy is as follows. We first prove that for each $\ell\in\llbracket1,k\rrbracket$,
that
\begin{equation}
\lim_{N\to\infty} \sup_{\xi_{\bm n}^{\bm{x}} \in \mathcal{E}_{N}^{\ell}} | F_N (\xi_{\bm n}^{\bm x}) | = 0.\label{eq:res-ell}
\end{equation}
We proceed by an induction on $\ell$.
%For each $\bm n \in {\bf N}_N^\ell$, define $F_{\bm n}^{\ell} : \mathbb{I}_{L}^{\ell} \to\mathbb{R}$
%as
%\begin{equation}
%F_{\bm n}^{\ell}(\bm{x}):=F_{N}(\xi_{\bm{n}}^{\bm{x}}) .\label{eq:Fnell-fNell-def}
%\end{equation}
Recall from \eqref{eq:JNKN} the definition of the sets $\mathcal J_N^\ell$ and $\mathcal K_N^\ell$. Define the operator $\color{blue} \mathcal L_{N}^{\ell}$ which acts on $C(\mathcal{E}_{N}^{\ell})$, the space of continuous real functions on $\mathcal E_N^\ell$,
as
\begin{equation}
\mathcal{L}_{N}^{\ell}\phi(\xi_{\bm n}^{\bm x}):=\begin{cases}
N^{2}\sum_{i=1}^{\ell}(\phi( \xi_{\bm n}^{\bm x + \bm e_i}) + \phi( \xi_{\bm n}^{\bm x - \bm e_i}) - 2\phi ( \xi_{\bm n}^{\bm x } )) & \text{if}\enspace \xi_{\bm n}^{\bm x} \in \mathcal J_{N}^{\ell},\\
0 & \text{if}\enspace \xi_{\bm n}^{\bm x} \in \mathcal K_{N}^{\ell},
\end{cases}\label{eq:LNell-def}
\end{equation}
where, recall, $\bm e_{i}$ denotes the unit vector in the $i$-th
direction. Namely, $\mathcal{L}_{N}^{\ell}$ induces a Markov chain
$(X_{t}^\ell)_{t\ge0}$ on $\mathcal E_N^{\ell}$
which gets absorbed at $\mathcal K_{N}^{\ell}$.

The key idea is to prove the following two estimates:
\begin{lem}\label{lem:res-ind-hyp}
For each $\ell \in \llbracket 1,k \rrbracket$,
\begin{equation}
\lim_{N\to\infty} \left\Vert \left(\lambda-\mathcal{L}_{N}^{\ell}\right) F_N \right\Vert _{L^{\infty}\left(\mathcal J_N^{\ell}\right)}=0,\label{eq:key-1}
\end{equation}
\begin{equation}
\limsup_{N\to\infty} \left\Vert F_N \right\Vert_{L^{\infty} \left(\mathcal K_N^{\ell}\right)}\le (1-p) \limsup_{N\to\infty}
\left\Vert F_N \right\Vert_{L^{\infty} \left(\mathcal E_N^{\ell} \right)},\label{eq:key-2}
\end{equation}
where $p\in(0,1)$ is the constant introduced in Lemma \ref{lem:escape-p}.
Here, $F_N$ should be understood as its restriction to $\mathcal E_N^\ell$.
\end{lem}

\begin{proof}[Proof of Theorem \ref{thm:res-inc} under Lemma \ref{lem:res-ind-hyp}]
By \cite[Theorem 7.15]{BdH15}, for all $\xi_{\bm n}^{\bm x} \in \mathcal E_N^{\ell} $,
\begin{equation}
F_N(\xi_{\bm n}^{\bm x}) = {\rm E}_{\xi_{\bm n}^{\bm x}}^{N,\ell}\left[\int_{0}^{H}e^{-\lambda t}\left(\lambda-\mathcal L_N^\ell \right) F_N (X_{t}^\ell)\,{\rm d}t+e^{-\lambda H} \, F_N(X_{H}^\ell)\right],\label{eq:key-appl}
\end{equation}
where ${\rm E}_{\xi_{\bm n}^{\bm x}}^{N,\ell}$ denotes the expectation
of the law of $(X_{t}^\ell)_{t\ge0}$ starting from $\xi_{\bm n}^{\bm x}$ and $H = H_{\mathcal K_N^\ell}$
denotes the first hitting time of $\mathcal K_N^{\ell}$. Substituting
\eqref{eq:key-1} and \eqref{eq:key-2} to the right-hand side of \eqref{eq:key-appl} yields that
\[
\limsup_{N\to\infty} \left\Vert F_N \right\Vert_{L^{\infty} \left( \mathcal E_N^{\ell} \right)}\le(1-p)\limsup_{N\to\infty}
\left\Vert F_N \right\Vert_{L^{\infty} \left(\mathcal E_N^{\ell} \right)},
\]
which proves \eqref{eq:res-ell} since $p\in(0,1)$. Collecting \eqref{eq:res-ell} for all $\ell \in \llbracket 1,k \rrbracket$ completes the proof of Theorem \ref{thm:res-inc}.
\end{proof}
In the remainder of this section, we prove Lemma \ref{lem:res-ind-hyp}.
%The proof of \eqref{eq:key-1} is relatively straightforward; it is
%presented in Section \ref{sec3.1}. Proving \eqref{eq:key-2} is much
%more demanding.
%In Section \ref{sec3.2}, we take advantage of the local analysis from Section \ref{sec2} and prove that the value of $F_N (\xi_{\bm n}^{\bm x})$, for $\xi_{\bm n}^{\bm x} \in \mathcal K_N^\ell$, can be inferred from the information of $F_N$ on $\mathcal J_N^\ell$ and $\mathcal E_N^{\ell-1}$.
%Then, in Section \ref{sec3.3}, we prove \eqref{eq:key-2}.

\subsection{\label{sec3.1}Interior Estimate: Proof of \eqref{eq:key-1}}

In this subsection, we prove the estimate \eqref{eq:key-1}. We start with the
$\ell=1$ case. Note that $\mathcal J_N^1 = \mathcal E_N^1 = \{ \xi_N^x : x \in \mathbb T_L \}$ and $\mathcal K_N^1 = \emptyset$.
\begin{lem}
\label{lem:key-1-1}Let $F_N$ be the solution of the resolvent equation introduced in Theorem \ref{thm:res-inc}. Then,
\[
\lim_{N\to\infty}\left\Vert \left(\lambda - \mathcal L_{N}^{1}\right) F_N \right\Vert_{L^{\infty} \left( \mathcal E_N^1 \right)}=0.
\]
\end{lem}

Let
\[
{\color{blue} w_{n}}:=\frac{\Gamma(d_{N}+n)}{\Gamma(d_{N})n!}\qquad\text{for}\quad n\in\mathbb{N}_0,
\]
where $\Gamma(\cdot)$ denotes the gamma function. For $n\in\llbracket1,N\rrbracket$,
it is easy to see that
\begin{equation}
\frac{d_{N}}{n}\le w_{n}\le\frac{d_{N}}{n+d_{N}} \, e^{d_{N}(1+\log N)},\label{eq:wn-est}
\end{equation}
which implies that (cf. \eqref{eq:dN-cond})
\[
\lim_{N \to \infty} \max_{1 \le n \le N} \left| \frac{nw_n}{d_N} - 1 \right| = 0.
\]

Define a measure $\color{blue} \mu_{N}$ on $\Omega_{N}$ as
\begin{equation}
\mu_{N}(\eta):=\prod_{y\in\mathbb{T}_{L}}w_{\eta_{y}}.\label{eq:muN-def}
\end{equation}
The measure $\mu_{N}$ is invariant under $\{\eta_{N}(t)\}_{t\ge0}$, as it satisfies
the detailed balance equations
\begin{equation}
\begin{aligned}\mu_{N}(\eta)r_{N}\,\big(\eta,\eta- & \delta^{x}+\delta^{x+1}\big)\, = \theta_N \left(\prod_{y\in\mathbb{T}_{L}\setminus\{x,x+1\}}w_{\eta_{y}}\right)w_{\eta_{x}}w_{\eta_{x+1}}\eta_{x}(d_{N}+\eta_{x+1})\\
 & = \theta_N \left(\prod_{y\in\mathbb{T}_{L}\setminus\{x,x+1\}}w_{\eta_{y}}\right)w_{\eta_{x}-1}w_{\eta_{x+1}+1}(\eta_{x+1}+1)(d_{N}+\eta_{x}-1)\\
 & =\mu_{N}\left(\eta-\delta^{x}+\delta^{x+1}\right)r_{N}\left(\eta-\delta^{x}+\delta^{x+1},\eta\right).
\end{aligned}
\label{eq:det-bal}
\end{equation}
Note that 
\begin{equation}
\mu_{N}(\xi_{N}^{x})=w_{N}\qquad\text{for all}\quad x\in\mathbb{T}_{L}.\label{eq:muN-wN}
\end{equation}

Recall from Definition \ref{def:Anx-def-1} the set $\mathcal A_N^x$:
\begin{equation}
\mathcal{A}_{N}^{x} = \left\{ \eta\in\Omega_{N}:\eta_{x}+\eta_{x+1}=N\right\} \cup\left\{ \eta\in\Omega_{N}:\eta_{x-1}+\eta_{x}=N\right\} .\label{eq:ANx}
\end{equation}
Note that $\xi_N^{x-1} , \xi_N^x , \xi_N^{x+1} \in \mathcal{A}_{N}^{x}$.

The key idea to prove Lemma \ref{lem:key-1-1} is as follows. For each $\xi_N^x \in \mathcal E_N^1$, we construct a suitable test function $h_N^x : \Omega_N \to \mathbb R$ concentrated on the tube of typical trajectories $\mathcal A_N^x$. We multiply the test function $h_N^x$ to the resolvent equation $(\lambda - \mathcal L_N)F_N = {\bf 1}_{\Delta_N}$ and integrate both sides with respect to $\mu_N$ to obtain that
\[
\lambda \int_{\Omega_N} h_N^x F_N \, {\rm d} \mu_N - \int_{\Omega_N} h_N^x (\mathcal L_N F_N) \, {\rm d} \mu_N = \int_{\Omega_N} h_N^x {\bf 1}_{\Delta_N} \, {\rm d} \mu_N
\]
The term on the right-hand side above will become negligible with respect to $\mu_N(\xi_N^x)$ because of the indicator function ${\bf 1}_{\Delta_N}$.
The first term on the left-hand side will be asymptotically equal to $\lambda F_N(\xi_N^x) \mu_N(\xi_N^x)$ by the same reason.
Finally, suppose that $h_N^x$ approximates the equilibrium potential $\mathfrak h_N^x (\eta) := \mathbb P_\eta^N [ H_{\mathcal E_N^x} < H_{\mathcal E_N \setminus \mathcal E_N^x} ]$. Then since $\mathcal L_N \mathfrak h_N^x = 0$ outside $\mathcal E_N$,
\[
\int_{\Omega_N} h_N^x (\mathcal L_N F_N) \, {\rm d} \mu_N = \int_{\Omega_N} (\mathcal L_N h_N^x) F_N \, {\rm d} \mu_N \simeq \int_{\mathcal E_N} (\mathcal L_N h_N^x) F_N \, {\rm d} \mu_N.
\]
This will approximate the generator part $\mathcal L_N^1 F_N (\xi_N^x) \mu_N(\xi_N^x)$.

\begin{proof}[Proof of Lemma \ref{lem:key-1-1}]
Fix $x\in\mathbb{T}_{L}$. From $(\lambda - \mathcal{L}_{N})F_{N} = {\bf 1}_{\Delta_N}$,
we write
\begin{equation}
\lambda\sum_{\eta\in\Omega_{N}}\mu_{N}(\eta)h_{N}^{x}(\eta)F_{N}(\eta) - \sum_{\eta\in\Omega_{N}}\mu_{N}(\eta)h_{N}^{x}(\eta)\mathcal{L}_{N}F_{N}(\eta)=\sum_{\eta\in\Omega_{N}}\mu_{N}(\eta)h_{N}^{x}(\eta) {\bf 1}_{\Delta_N} (\eta),\label{eq:id}
\end{equation}
where $h_{N}^{x}:\Omega_{N}\to\mathbb{R}$ is defined as (cf. \eqref{eq:ANx})
\[
h_{N}^{x}(\eta):=\begin{cases}
\frac{\eta_{x}}{N} & \text{if}\enspace\eta\in\mathcal{A}_{N}^{x},\\
0 & \text{otherwise}.
\end{cases}
\]

Write $\mu=\mu_{N}$ and $h=h_{N}^{x}$. 
Consider first the right-hand side of \eqref{eq:id}. Since $h=0$ outside $\mathcal A_N^x \setminus \{ \xi_N^{x-1} , \xi_N^{x+1} \} $ and ${\bf 1}_{\Delta_N} = 0$ outside $\Delta_N$,
\begin{equation}
\sum_{\eta\in\Omega_{N}} \mu(\eta) h(\eta) {\bf 1}_{\Delta_N}(\eta) =
\sum_{\eta\in\mathcal{A}_{N}^{x} \setminus \{ \xi_N^{x-1} , \xi_N^x , \xi_N^{x+1} \} } \mu(\eta)\frac{\eta_{x}}{N}.\label{eq:1-2}
\end{equation}
By \eqref{eq:muN-def} and \eqref{eq:wn-est},
\begin{align*}
\mu\left(\mathcal{A}_{N}^{x}\setminus \{ \xi_N^{x-1} , \xi_N^x , \xi_N^{x+1} \} \right) & =2\sum_{i=1}^{N-1}w_{i}w_{N-i}\\
\le2d_{N}^{2} \, e^{2d_{N}(1+\log N)} & \sum_{i=1}^{N-1}\frac{1}{i(N-i)}\le2d_{N}^{2}e^{2d_{N}(1+\log N)}\frac{2(1+\log N)}{N}.
\end{align*}
Thus by \eqref{eq:dN-cond},
\begin{equation}
\mu\left(\mathcal{A}_{N}^{x}\setminus \{ \xi_N^{x-1} , \xi_N^x , \xi_N^{x+1} \} \right) = O \left(\frac{d_{N}^{2}\log N}{N}\right).\label{eq:1-3}
\end{equation}
Combining \eqref{eq:1-2} and \eqref{eq:1-3}, along with $0 \le \eta_x/N \le 1$, we obtain that
\begin{equation}\label{eq:1-3.5}
\sum_{\eta\in\Omega_{N}} \mu(\eta) h(\eta) {\bf 1}_{\Delta_N}(\eta) = O \left(\frac{d_{N}^{2}\log N}{N}\right) .
\end{equation}

Next, we consider the first term in \eqref{eq:id}. The same argument yields that
\begin{equation}
\lambda \sum_{\eta\in\Omega_{N}} \mu(\eta) h(\eta) F_N (\eta) =
\lambda \sum_{\eta\in\mathcal{A}_{N}^{x} \setminus \{ \xi_N^{x-1} , \xi_N^{x+1} \} } \mu(\eta)\frac{\eta_{x}}{N} F_N(\eta).\label{eq:1-3.6}
\end{equation}
Substituting \eqref{eq:1-3} to \eqref{eq:1-3.6} and applying \eqref{eq:muN-wN} give that
\begin{equation}
\lambda\sum_{\eta\in\Omega_{N}}\mu(\eta)h(\eta)F_{N}(\eta)=w_{N}(\lambda F_{N}(\xi_{N}^{x})+ O (d_{N}\log N)).\label{eq:1-4}
\end{equation}
Here, we used \eqref{eq:1-5} and the fact that $w_{N}\ge d_{N}/N$ from \eqref{eq:wn-est}.

Finally, let us calculate the second term in \eqref{eq:id}, which
is
\begin{equation}
\sum_{\eta\in\Omega_{N}}\mu(\eta)h(\eta)\mathcal{L}_{N}F_{N}(\eta)= \sum_{\eta\in\Omega_{N}}\mu(\eta)\mathcal{L}_{N}h(\eta)F_{N}(\eta),\label{eq:1-6}
\end{equation}
where the equality holds since $\mathcal{L}_{N}$ is self-adjoint
with respect to $\mu$ (cf. \eqref{eq:det-bal}). First, let us evaluate
the summation in the right-hand side of \eqref{eq:1-6} on $\{ \xi_N^{x-1} , \xi_N^x , \xi_N^{x+1} \}$.
The coefficient of $F_{N}(\xi_{N}^{x})$ equals
\begin{equation}
\begin{aligned}\mu(\xi_{N}^{x})\mathcal{L}_{N} & h(\xi_{N}^{x}) = w_{N}\sum_{y\in\{x+1,x-1\}}r_{N}\left(\xi_{N}^{x},\xi_N^x - \delta^x + \delta^y \right)\left(h\left(\xi_N^x - \delta^x + \delta^y \right)-h(\xi_{N}^{x})\right)\\
 & =w_{N} \theta_N \left(Nd_{N}\left(\frac{N-1}{N}-1\right)+Nd_{N}\left(\frac{N-1}{N}-1\right)\right)=-2w_{N} \theta_N d_{N}.
\end{aligned}
\label{eq:1-7}
\end{equation}
The coefficient of $F_{N}(\xi_{N}^{x+1})$ equals
\begin{equation}
\mu(\xi_{N}^{x+1})\mathcal{L}_{N}h(\xi_{N}^{x+1})=w_{N} \theta_N Nd_{N}\left(\frac{1}{N}-0\right)=w_{N} \theta_N d_{N},\label{eq:1-8}
\end{equation}
and similarly the coefficient of $F_{N}(\xi_{N}^{x-1})$ equals $w_{N} \theta_N d_{N}$.

Next, we handle the elements in $\mathcal{A}_{N}^{x}\setminus \mathcal E_N $.
Fix one element $i\,\delta^x + (N-i)\,\delta^y$ where $y\in\{x+1,x-1\}$
and $i\in\llbracket1,N-1\rrbracket$. Without loss of generality,
suppose that $y=x+1$. Then,
\begin{equation}
\begin{aligned}\mathcal{L}_{N}h & \left( i\,\delta^x + (N-i)\,\delta^y \right) = \theta_N \, \bigg[ i(d_{N}+N-i) \left(\frac{i-1}{N}-\frac{i}{N}\right) \\
 & +(N-i)(d_{N}+i)\left(\frac{i+1}{N}-\frac{i}{N}\right) +O (id_{N})+O ((N-i)d_{N}) \bigg] \,,
\end{aligned}
\label{eq:1-9}
\end{equation}
where the first two terms on the right-hand side correspond to the particle
jumps of types $x\leftrightarrow x+1$, and the last two terms correspond
to the jumps of types $x\to x-1$ and $x+1\to x+2$, respectively. Due to cancelation,
this equals
\[
\theta_N \left[ -\frac{i}{N}d_{N}+\frac{N-i}{N}d_{N}+O (Nd_{N}) \right] = O (N \theta_N d_{N}).
\]
Hence, the whole contribution of $\mathcal{A}_{N}^{x}\setminus \mathcal{E}_{N}$
to the right-hand side of \eqref{eq:1-6} equals
\begin{equation}
\mu\left(\mathcal{A}_{N}^{x} \setminus \mathcal{E}_{N} \right)O (N \theta_N d_{N})\|F_{N}\|_{\infty}= {\lambda}^{-1}\theta_{N}O (d_{N}^{3}\log N),\label{eq:1-10}
\end{equation}
where \eqref{eq:1-3} and \eqref{eq:1-5} are used in the equality.
Finally, for the remaining set $\Omega_{N}\setminus\mathcal{A}_{N}^{x}$,
owing to the detailed balance, its contribution has exactly the same
scale as the one coming from the last two terms on the right-hand side of \eqref{eq:1-9}. It is thus equal to
\begin{equation}
{\lambda}^{-1} \theta_{N}O (d_{N}^{3}\log N).\label{eq:1-11}
\end{equation}
Combining \eqref{eq:1-6}, \eqref{eq:1-7}, \eqref{eq:1-8}, \eqref{eq:1-10},
and \eqref{eq:1-11}, along with \eqref{eq:wn-est}, we obtain
\begin{equation}
\begin{aligned} & \sum_{\eta\in\Omega_{N}}\mu(\eta)h(\eta)\mathcal{L}_{N}F_{N}(\eta)\\
 & =\theta_{N}w_{N}d_{N}\left(F_{N}(\xi_{N}^{x+1})+F_{N}(\xi_{N}^{x-1})-2F_{N}(\xi_{N}^{x})+{\lambda}^{-1}O (d_{N}N\log N)\right).
\end{aligned}
\label{eq:1-12}
\end{equation}

Collecting \eqref{eq:id}, \eqref{eq:1-3.5}, \eqref{eq:1-4}, and \eqref{eq:1-12} yields
\[
\lambda F_N (\xi_N^x) - N^2 \left(F_{N}(\xi_{N}^{x+1})+F_{N}(\xi_{N}^{x-1})-2F_{N}(\xi_{N}^{x}) \right) = \lambda^{-1} O (d_N N^3 \log N) .
\]
By \eqref{eq:LNell-def}, we may rewrite this as
\begin{equation}
\left(\lambda-\mathcal L_{N}^{1}\right)F_{N} (\xi_N^x) = {\lambda}^{-1} O (d_{N}N^{3}\log N),\label{eq:1-13}
\end{equation}
which completes the proof of Lemma \ref{lem:key-1-1} in view of \eqref{eq:dN-cond}.
\end{proof}

Next, we handle the remaining cases of $\ell\ge2$ for \eqref{eq:key-1}.

\begin{lem}
\label{lem:key-1-ell}For each $\ell\in\llbracket2,k\rrbracket$,
\[
\lim_{N\to\infty} \left\Vert \left(\lambda-\mathcal L_{N}^{\ell}\right) F_N \right\Vert_{L^{\infty} \left( \mathcal J_N^{\ell} \right)}=0.
\]
\end{lem}

\begin{proof}
The idea is intuitively the same as in the proof of Lemma \ref{lem:key-1-1},
since each jump from $x_{i}$, $i\in\llbracket1,\ell\rrbracket$,
which involves $x_{i}\pm1$, is isolated from any other $x_{j}$,
$j\ne i$. Thus, all computations remain local and valid.

More precisely, fix $ \bm{x} \in \mathbb J_{L}^{\ell}$ and $\bm n \in {\bf N}_N^\ell$. The
test function $h=h_{\bm{n}}^{\bm{x}}:\Omega_{N}\to\mathbb{R}$ is
defined as
\[
h(\eta):=\begin{cases}
\frac{\eta_{x_{i}}}{n_{i}} & \text{if}\enspace\eta\in\mathcal{A}_{\bm n}^{\bm{x},i,+} \cup \mathcal{A}_{\bm n}^{\bm{x},i,-}\enspace\text{for some}\enspace i\in\llbracket1,\ell\rrbracket,\\
0 & \text{otherwise}.
\end{cases}
\]
As done in \eqref{eq:1-3.5}, \eqref{eq:1-4}, and \eqref{eq:1-12},
the last, first, and second terms in \eqref{eq:id} are respectively equal to
\begin{equation}
\sum_{\eta\in\Omega_{N}}\mu(\eta)h(\eta) {\bf 1}_{\Delta_N} (\eta) = \mu (\xi_{\bm n}^{\bm x}) \,  O \left( d_N \log N \right),\label{eq:2.2-1}
\end{equation}
\begin{equation}
\lambda\sum_{\eta\in\Omega_{N}}\mu(\eta)h(\eta)F_{N}(\eta)=\mu(\xi_{\bm{n}}^{\bm{x}})(\lambda F_{N}(\xi_{\bm{n}}^{\bm{x}})+ O (d_{N}\log N)),\label{eq:2.2-2}
\end{equation}
and
\begin{equation}
\theta_{N}\mu(\xi_{\bm{n}}^{\bm{x}})d_{N}\left(\sum_{i=1}^{\ell}\left(F_{N}\left(\xi_{\bm{n}}^{\bm{x}+\bm e_{i}}\right)+F_{N}\left(\xi_{\bm{n}}^{\bm{x}-\bm e_{i}}\right)-2F_{N}(\xi_{\bm{n}}^{\bm{x}})\right)+ {\lambda}^{-1} O (d_{N}N\log N)\right).\label{eq:2.2-3}
\end{equation}
Combining \eqref{eq:2.2-1}, \eqref{eq:2.2-2}, \eqref{eq:2.2-3}, \eqref{eq:LNell-def},
and \eqref{eq:dN-cond} yields
\begin{equation}
\lambda F_N (\xi_{\bm n}^{\bm x}) - \mathcal L_{N}^{\ell} F_N (\xi_{\bm n}^{\bm x}) = o(1),\label{eq:2.2-4}
\end{equation}
where the error is uniform over all $\bm{x} \in \mathbb J_{L}^{\ell}$ and $\bm n \in {\bf N}_N^\ell$. This is exactly what has been claimed.
\end{proof}

\subsection{\label{sec3.3}Proof of Lemma \ref{lem:res-ind-hyp}}

In this subsection, we prove Lemma \ref{lem:res-ind-hyp}. The first display is already proved in Section \ref{sec3.1}, so we focus on the second display. Consider $\xi_{\bm n}^{\bm x} \in \mathcal K_N^\ell$, i.e., $\bm x \in \mathbb K_L^\ell$ and $\bm n \in {\bf N}_N^\ell$. We have
\[
F_{N}(\xi_{\bm{n}}^{\bm{x}}) =
{\rm E}^{\mathbb{P}_{\xi_{\bm{n}}^{\bm{x}}}^{N}} \left[ \left( \int_{0}^{ H_{ \mathcal J_N^\ell \cup \mathcal E_N^{\ell-1} } }+\int_{H_{ \mathcal J_N^\ell \cup \mathcal E_N^{\ell-1} } }^{\infty} \right) e^{-\lambda t}\, {\bf 1}_{\Delta_N} (\eta_N(t)) \,{\rm d}t \right].
\]
By the strong Markov property at time $H := H_{ \mathcal J_N^\ell \cup \mathcal E_N^{\ell-1} }$,
\begin{equation}\label{eq:SMP-0}
\begin{aligned}
{\rm E}^{\mathbb{P}_{\xi_{\bm{n}}^{\bm{x}}}^{N}} \left[ \int_H^{\infty} e^{-\lambda t}\, {\bf 1}_{\Delta_N} (\eta_N(t)) \,{\rm d}t \right] & =
{\rm E}^{\mathbb{P}_{\xi_{\bm{n}}^{\bm{x}}}^{N}} \left[ e^{-\lambda H} \, {\rm E}^{\mathbb P_{\eta_N(H)}^N} \int_0^{\infty} e^{-\lambda t}\, {\bf 1}_{\Delta_N} (\eta_N(t)) \,{\rm d}t \right] \\
& = {\rm E}^{\mathbb P_{\xi_{\bm n}^{\bm x}}^N} \left[ e^{-\lambda H} \, F_N (\eta_N (H)) \right] .
\end{aligned}
\end{equation}
Substituting this to the penultimate displayed identity,
\begin{equation}\label{eq:SMP}
\begin{aligned}
\left|F_{N}(\xi_{\bm{n}}^{\bm{x}})-{\rm E}^{\mathbb{P}_{\xi_{\bm{n}}^{\bm{x}}}^{N}}\left[ F_N (\eta_N (H)) \right] \right| & =
{\rm E}^{\mathbb{P}_{\xi_{\bm{n}}^{\bm{x}}}^{N}} \left[ \int_{0}^{H} e^{-\lambda t}\, {\bf 1}_{\Delta_N} (\eta_N(t)) \,{\rm d}t
- \left( 1 - e^{-\lambda H} \right) F_N (\eta_N (H)) \right] \\
& \le \frac2{\lambda}\,{\rm E}^{\mathbb{P}_{\xi_{\bm{n}}^{\bm{x}}}^{N}}\left[1-e^{-\lambda H_{ \mathcal J_N^\ell \cup \mathcal E_N^{\ell-1} } } \right] .
\end{aligned}
\end{equation}
In the inequality, we used that ${\bf 1}_{\Delta_N} \le 1$ and $\|F_N\|_\infty \le 1/\lambda$.
First, we estimate the expectation in the right-hand side of \eqref{eq:SMP}.
\begin{lem}
\label{lem:xinx-hit}It holds that
\[
\lim_{N\to\infty} \sup_{\bm n \in {\bf N}_N^\ell} \sup_{\bm{x} \in \mathbb K_{L}^{\ell}}
{\rm E}^{\mathbb{P}_{\xi_{\bm{n}}^{\bm{x}}}^{N}}\left[1-e^{-\lambda H_{ \mathcal J_N^\ell \cup \mathcal E_N^{\ell-1} } } \right] = 0 .
\]
\end{lem}

\begin{proof}
By Lemma \ref{lem:tube-typ-2},
\begin{align*}
\sup_{\bm n \in {\bf N}_N^\ell} \sup_{\bm{x} \in \mathbb K_{L}^{\ell}} {\rm E}^{\mathbb{P}_{\xi_{\bm{n}}^{\bm{x}}}^{N}}\,\Big[\Big(1 & - e^{-\lambda H_{ \mathcal J_N^\ell \cup \mathcal E_N^{\ell-1} } } \Big)\,{\bf 1}\left\{ H_{ \mathcal J_N^\ell \cup \mathcal E_N^{\ell-1} } > \mathfrak T_N \right\} \Big]\,\\
 & \le \sup_{\bm n \in {\bf N}_N^\ell} \sup_{\bm{x} \in \mathbb K_{L}^{\ell}} \mathbb{P}_{\xi_{\bm{n}}^{\bm{x}}}^{N}\left( H_{ \mathcal J_N^\ell \cup \mathcal E_N^{\ell-1} } > \mathfrak T_N \right) \xrightarrow{N\to\infty}0.
\end{align*}
It remains to prove that
\begin{equation}
\lim_{N\to\infty} \sup_{\bm n \in {\bf N}_N^\ell} \sup_{\bm{x} \in \mathbb K_{L}^{\ell}} {\rm E}^{\mathbb{P}_{\xi_{\bm{n}}^{\bm{x}}}^{N}}
\left[ \left(1 - e^{-\lambda H_{ \mathcal J_N^\ell \cup \mathcal E_N^{\ell-1} } } \right){\bf 1}\left\{ H_{ \mathcal J_N^\ell \cup \mathcal E_N^{\ell-1} } < \mathfrak T_N \right\} \right]=0.\label{eq:BNell-pf}
\end{equation}
This follows from Lemma \ref{lem:good-hitting-time}.
\end{proof}
Now, we are ready to prove Lemma \ref{lem:res-ind-hyp}.

\begin{proof}[Proof of Lemma \ref{lem:res-ind-hyp}]
The case $\ell = 1$ follows from Lemma \ref{lem:key-1-1}.
Fix $\ell \ge 2$ and assume that Lemma \ref{lem:res-ind-hyp} holds for $\ell-1$. Assertion \eqref{eq:key-1} holds by Lemma \ref{lem:key-1-ell}.

We turn to assertion \eqref{eq:key-2}. By \eqref{eq:SMP} and Lemma \ref{lem:xinx-hit},
\begin{equation}
\limsup_{N\to\infty} \sup_{\bm n \in {\bf N}_N^\ell} \sup_{ \bm{x}  \in \mathbb K_{L}^{\ell}} \left| F_N(\xi_{\bm n}^{\bm x})\right|\le\limsup_{N\to\infty} \sup_{\bm n \in {\bf N}_N^\ell} \sup_{ \bm{x}  \in \mathbb K_{L}^{\ell}} {\rm E}^{\mathbb{P}_{\xi_{\bm{n}}^{\bm{x}}}^{N}}\left[ \left| F_{N}\left(\eta_{N}\left( H_{ \mathcal J_N^\ell \cup \mathcal E_N^{\ell-1} } \right)\right) \right| \right] .\label{eq:key-2-pf-1}
\end{equation}
Divide the expectation on the right-hand side into two parts: $\{ H_{\mathcal E_N^{\ell-1}} < H_{\mathcal J_N^\ell} \}$ and $\{ H_{\mathcal E_N^{\ell-1}} > H_{\mathcal J_N^\ell} \}$. Start with the first part,
\begin{equation}\label{eq:Ynx-part-bound}
{\rm E}^{\mathbb{P}_{\xi_{\bm{n}}^{\bm{x}}}^{N}}\left[ {\bf 1} \left\{ H_{\mathcal E_N^{\ell-1}} < H_{\mathcal J_N^\ell} \right\} \left| F_{N}\left(\eta_{N}\left( H_{ \mathcal E_N^{\ell-1} } \right)\right) \right| \right]  \le
{\rm E}^{\mathbb{P}_{\xi_{\bm{n}}^{\bm{x}}}^{N}}\left[ \left| F_{N}\left(\eta_{N}\left( H_{ \mathcal E_N^{\ell-1} } \right)\right) \right| \right] .
\end{equation}
In the proof of Theorem \ref{thm:res-inc} we showed that conditions \eqref{eq:key-1} and \eqref{eq:key-2} yield \eqref{eq:res-ell}.
Hence,
the induction hypothesis for $\ell-1$ implies that \eqref{eq:res-ell} holds for $\ell-1$. Thus, the right-hand side of \eqref{eq:Ynx-part-bound} vanishes uniformly as $N \to \infty$.

The second part on the right-hand side of \eqref{eq:key-2-pf-1} can be written as
\[
{\rm E}^{\mathbb{P}_{\xi_{\bm{n}}^{\bm{x}}}^{N}}\left[ {\bf 1} \left\{ H_{\mathcal E_N^{\ell-1}} > H_{\mathcal J_N^\ell} \right\} \left| F_{N}\left(\eta_{N}\left( H_{ \mathcal J_N^\ell } \right)\right) \right| \right] .
\]
By Lemma \ref{lem:escape-p}, this is bounded above by
\begin{equation}\label{eq:Xnx-part-bound}
(1-p) \limsup_{N \to \infty} \sup_{\xi \in \mathcal E_N^\ell} | F_N(\xi) | .
\end{equation}
Combining \eqref{eq:key-2-pf-1}, \eqref{eq:Ynx-part-bound}, and \eqref{eq:Xnx-part-bound} concludes the proof of \eqref{eq:key-2} in the case of $\ell$. This completes the induction step.
\end{proof}

\section{\label{sec4}Proof of Theorem \ref{thm2}}
In this section, we prove Theorem \ref{thm2}.
We need a lemma regarding the time change between the original process and the trace process.

\begin{lem}\label{lem:time-change}
Recall the definition of the time change $S_N(t)$ from \eqref{eq:TNt}. Then,
\[
\lim_{N \to \infty} \sup_{\xi \in \mathcal E_N} \mathbb P_\xi^N ( S_N(t) \ge t+a) = 0 \qquad \text{for any} \quad t,a>0.
\]
\end{lem}
\begin{proof}
Notice that $S_N(t) \ge t+a$ if and only if $T_N(t+a) \le t$. Thus by \eqref{eq:TNt} and the Markov inequality,
\begin{align*}
\mathbb P_\xi^N ( S_N(t) & \ge t+a) = \mathbb P_\xi^N \left( \int_0^{t+a} {\bf 1} \{ \eta_N(s) \in \mathcal E_N \} \, {\rm d}s \le t \right) \\
& = \mathbb P_\xi^N \left( \int_0^{t+a} {\bf 1} \{ \eta_N(s) \in \Delta_N \} \, {\rm d}s \ge a \right)
\le a^{-1} \, {\rm E}^{\mathbb P_\xi^N} \left[ \int_0^{t+a} {\bf 1} \{ \eta_N(s) \in \Delta_N \} \, {\rm d}s \right] .
\end{align*}
This completes the proof in view of Theorem \ref{thm1}.
\end{proof}

\begin{proof}[Proof of Theorem \ref{thm2}]
%Recall that $\mathscr A_{N,T} \subset D([0,\infty);\mathcal E_N)$ denotes the event that only the typical jumps (in the sense of Definition \ref{def:typ-jump}) occur until time $T>0$.
Recall from \eqref{eq:TN-def} the definition of $\mathfrak T_N$.
Suppose that $\mathfrak T_N > 2T$ and $S_N(T) < 2T$. Then, every jump of
the original process happens inside $\mathcal A_{\bm n}^{\bm x}$ for some $\xi_{\bm n}^{\bm x} \in \mathcal E_N$ until time $2T$. In this case, if we trace the process on $\mathcal E_N$, by \eqref{eq:Nnx-Anx}, only the jumps to neighbors occur along the trajectory until time $T$ since $S_N(T) < 2T$. This observation implies that, for each $\xi \in \mathcal E_N$,
\[
\mathbb P_\xi^N ( \mathfrak T_N > 2T , \enspace S_N(T) < 2T ) \le \mathbb Q_\xi^N ( \tau_N > T),
\]
thus it suffices to prove that
\[
\lim_{N \to \infty} \inf_{\xi \in \mathcal E_N} \mathbb P_{\xi}^N ( \mathfrak T_N > 2T , \enspace S_N(T) < 2T ) = 1.
\]
By Lemmas \ref{lem:time-change} and \ref{lem:tube-typ},
\[
\lim_{N \to \infty} \inf_{\xi \in \mathcal E_N} \mathbb P_\xi^N ( S_N(T) < 2T) = 
\lim_{N \to \infty} \inf_{\xi \in \mathcal E_N} \mathbb P_\xi^N ( \mathfrak T_N > 2T) = 1.
\]
These two relations conclude the proof.
\end{proof}

\section{\label{sec5}Second Resolvent Condition for Labeled Processes}

In this section, we present a bound on the solution of the resolvent equation for labeled processes, which is the main estimate in the proof of Theorem \ref{thm4}.
This is summarized in the following theorem.
\begin{thm}[Flatness of Resolvent Solutions]\label{thm:res}
Recall from \eqref{eq:Lrho}, \eqref{eq:ENhat}, Definition \ref{def:lab-trace}, and \eqref{eq:PhiN} the definition of generator $\mathfrak L_\rho$, set $\widehat{\mathcal E}_N$, generator $\widehat{\mathcal L}_N$, and projection $\Phi_N$, respectively.
For each $\lambda>0$ and ${\bf g}\in C({\mathbb T}^{k})$,
denote by ${\bf f}\in\mathfrak{D}^{(k)}$ the unique solution to $(\lambda-\mathfrak L_{\rho}){\bf f}={\bf g}$.
Define a function $G_{N}: \widehat{\mathcal E}_N \to \mathbb{R}$ as
\begin{equation}\label{eq:GN-def}
G_N := {\bf g} \circ \Phi_N.
\end{equation}
%\begin{equation}
%G_{N}(\widehat\xi_{\bm n}^{\bm x}):=
%\begin{cases}
%{\bf g} \left( \left\{ \frac{\bm x}L \right\} \right) & \text{if}\enspace\eta=\xi_{\bm{n}}^{\bm{x}}\in\mathcal{E}_{N}\enspace\text{for some}\enspace\bm{x}\in\mathbb{I}_{L}^{\ell}\enspace\text{and}\enspace\bm{n}\in{\bf N}_{N}^{\ell},\\
%0 & \text{otherwise}.
%\end{cases}
%\label{eq:GN-def}
%\end{equation}
Denote by $F_{N}$ the unique solution to $(\lambda-\widehat{\mathcal{L}}_{N})F_{N}=G_{N}$.
Then,
\[
\lim_{N\to\infty} \sup_{ \widehat\xi_{\bm n}^{\bm x} \in \widehat{\mathcal E}_N} \left| F_{N}(\widehat\xi_{\bm{n}}^{\bm{x}})-{\bf f}\left( \frac{\bm{x}}{L} \right) \right| = 0.
\]
\end{thm}

The strategy to prove Theorem \ref{thm:res} is exactly the same as the proof of Theorem \ref{thm1} presented in Section \ref{sec3}. For each $\ell \in \llbracket 1,k \rrbracket$, define
\[
{\color{blue} \widehat{\mathcal E}_N^\ell } := \left\{ \widehat\xi_{\bm n}^{\bm x} \in \widehat{\mathcal E}_N : \xi_{\bm n}^{\bm x} \in \mathcal E_N^\ell \right\}
\]
and
\[
{\color{blue} \widehat{\mathcal J}_N^\ell } := \left\{ \widehat\xi_{\bm n}^{\bm x} \in \widehat{\mathcal E}_N : \xi_{\bm n}^{\bm x} \in \mathcal J_N^\ell \right\}, \qquad
{\color{blue} \widehat{\mathcal K}_N^\ell } := \left\{ \widehat\xi_{\bm n}^{\bm x} \in \widehat{\mathcal E}_N : \xi_{\bm n}^{\bm x} \in \mathcal K_N^\ell \right\}.
\]
As before, our objective is to prove for each $\ell \in \llbracket 1,k \rrbracket$ that
\[
\lim_{N \to \infty} \sup_{\widehat\xi \in \widehat{\mathcal E}_N^\ell} \left| F_N (\widehat\xi) - ({\bf f} \circ \Phi_N) (\widehat\xi) \right| = 0 .
\]

Define a Markov chain $(Y_t^\ell)_{t \ge 0}$ on $\widehat{\mathcal E}_N^\ell$ as follows. Fix $\widehat\xi_{\bm n}^{\bm x} \in \widehat{\mathcal J}_N^\ell$.

\begin{itemize}
\item If $x_{i-1} \ne x_i = \cdots = x_{i+j-1} \ne x_{i+j}$, then assign jump rate $N^2$ to $\widehat\xi_{\bm n}^{\bm x} \to \widehat\xi_{\bm n}^{\bm x \pm ( \bm e_i + \cdots + \bm e_{i+j-1} ) }$.
\item (Only if $\ell = 1$) If $x_1 = \cdots = x_k$, then assign jump rate $N^2$ to $\widehat\xi_{\bm n}^{\bm x} \to \widehat\xi_{\bm n}^{\bm x \pm (\bm e_1 + \cdots + \bm e_k)}$.
\end{itemize}
Denote by $\color{blue} \widehat{\mathcal L}_N^\ell$ the corresponding infinitesimal generator acting on $C(\widehat{\mathcal E}_N^\ell)$, which induces a Markov chain which gets absorbed at $\widehat{\mathcal K}_N^\ell$.
It is designed in such a way that if we project the trajectory via $\Psi_N : \widehat{\mathcal E}_N \to \mathcal E_N$, it then follows the law of the Markov chain $X_t^\ell$ defined in \eqref{eq:LNell-def}.

Define ${\color{blue} \varepsilon_N} := F_N  - ({\bf f} \circ \Phi_N)$.
The following lemma is an analogue of Lemma \ref{lem:res-ind-hyp}.

\begin{lem}\label{lem:res-ind-hyp-label}
For each $\ell \in \llbracket 1,k \rrbracket$,
\begin{equation}
\lim_{N\to\infty} \left\Vert \left(\lambda- \widehat{\mathcal L}_{N}^{\ell}\right) \varepsilon_N \right\Vert _{L^{\infty}\left(\mathcal J_N^{\ell}\right)}=0,\label{eq:key-hat-1}
\end{equation}
\begin{equation}
\limsup_{N\to\infty} \left\Vert \varepsilon_N \right\Vert_{L^{\infty} \left(\mathcal K_N^{\ell}\right)}\le (1-p) \limsup_{N\to\infty}
\left\Vert \varepsilon_N \right\Vert_{L^{\infty} \left(\mathcal E_N^{\ell} \right)}, \label{eq:key-hat-2}
\end{equation}
where $p \in (0,1)$ is from Lemma \ref{lem:escape-p}.
\end{lem}

\begin{proof}[Proof of Theorem \ref{thm:res} under Lemma \ref{lem:res-ind-hyp-label}]
The proof is exactly the same with that of Theorem \ref{thm:res-inc} presented at the beginning of Section \ref{sec3}; the only difference is to substitute \eqref{eq:key-1} and \eqref{eq:key-2} with \eqref{eq:key-hat-1} and \eqref{eq:key-hat-2}.
\end{proof}

Now, we prove Lemma \ref{lem:res-ind-hyp-label}.

\begin{lem}\label{lem1}
For $\ell \in \llbracket 1,k \rrbracket$ and $\lambda > 0$,
\[
\lim_{N \to \infty} \left\| (\lambda - \widehat{\mathcal L}_N^\ell ) \, \varepsilon_N \right\|_{L^\infty \left( \widehat{\mathcal J}_N^\ell \right)} = 0.
\]
\end{lem}

\begin{proof}
As done in the proof of Lemmas \ref{lem:key-1-1} and \ref{lem:key-1-ell}, we start from $(\lambda - \widehat{\mathcal L}_N) F_N = G_N$. Fix $\widehat\xi_{\bm n}^{\bm x} \in \widehat{\mathcal J}_N^\ell$ and, without loss of generality, suppose that
\[
x_1 = \cdots = x_{j_1} \ne x_{j_1+1} = \cdots = x_{j_1+j_2} \ne \cdots \ne x_{j_1+\cdots+j_{\ell-1}+1} = \cdots = x_{j_1+\cdots+j_\ell} \ne x_1,
\]
where $j_1 ,\dots, j_\ell \ge 1$ and $j_1 + \cdots + j_\ell = k$. Then,
\[
\lambda F_N (\widehat\xi_{\bm n}^{\bm x}) - \widehat{\mathcal L}_N F_N (\widehat\xi_{\bm n}^{\bm x}) = G_N ( \widehat\xi_{\bm n}^{\bm x}) = {\bf g} \left( \frac {\bm x}L \right) = (\lambda - \mathfrak L_\rho) {\bf f} \left( \frac {\bm x}L \right) .
\]
Rearranging, this becomes
\begin{equation}\label{eq:id-label}
\lambda \, \varepsilon_N (\widehat\xi_{\bm n}^{\bm x}) = \widehat{\mathcal L}_N F_N (\widehat\xi_{\bm n}^{\bm x}) - \mathfrak L_\rho {\bf f} \left( \frac {\bm x} L \right) .
\end{equation}

Consider the first term on the right-hand side of \eqref{eq:id-label}. By Definition \ref{def:lab-trace},
\begin{align*}
\widehat{\mathcal L}_N & F_N (\widehat\xi_{\bm n}^{\bm x}) = \sum_{\widehat\xi \in \widehat{\mathcal E}_N} \widehat R_N ( \widehat\xi_{\bm n}^{\bm x} , \widehat\xi ) ( F_N (\widehat\xi) - F_N (\widehat\xi_{\bm n}^{\bm x}) ) \\
& = \sum_{i=1}^\ell \left[
R_N (\xi_{\bm n}^{\bm x} , \xi_{\bm n}^{\bm x + \bm v_i}) ( F_N (\widehat\xi_{\bm n}^{\bm x + \bm v_i}) - F_N(\widehat\xi_{\bm n}^{\bm x}) ) +
R_N (\xi_{\bm n}^{\bm x} , \xi_{\bm n}^{\bm x - \bm v_i}) ( F_N (\widehat\xi_{\bm n}^{\bm x - \bm v_i}) - F_N(\widehat\xi_{\bm n}^{\bm x}) ) \right] ,
\end{align*}
where $\bm v_i := \bm e_{j_1+\cdots+j_{i-1}+1} + \cdots + \bm e_{j_1+\cdots+j_i}$.
Then by Lemma \ref{lem:JN-trace-rate} and the definition of $\widehat{\mathcal L}_N^\ell$,
\begin{align*}
\widehat{\mathcal L}_N F_N (\widehat\xi_{\bm n}^{\bm x})
= (N^2 + O(d_N N^3 \log N) ) & \sum_{i=1}^\ell \left[ F_N (\widehat\xi_{\bm n}^{\bm x + \bm v_i}) + F_N (\widehat\xi_{\bm n}^{\bm x - \bm v_i}) - 2 F_N(\widehat\xi_{\bm n}^{\bm x}) \right]  \\
& = \widehat{\mathcal L}_N^\ell F_N (\widehat\xi_{\bm n}^{\bm x}) + o(1),
\end{align*}
where in the last part we used \eqref{eq:dN-cond}. Thus, it suffices to prove that
\[
\widehat{\mathcal L}_N^\ell ({\bf f} \circ \Phi_N) ( \widehat\xi_{\bm n}^{\bm x}) =
\mathfrak L_\rho {\bf f} \left( \frac {\bm x}L \right) + o(1),
\]
with a uniform error over all $\widehat\xi_{\bm n}^{\bm x} \in \widehat{\mathcal J}_N^\ell$. The left-hand side of this identity is equal to
\begin{equation}\label{eq:sum-repr}
\begin{aligned}
\sum_{i=1}^\ell N^2 & \left[ {\bf f} \left( \frac{\bm x + \bm e_{j_1+\cdots+j_{i-1}+1} + \cdots + \bm e_{j_1+\cdots+j_i}}L \right) - {\bf f} \left( \frac{\bm x}L \right) \right] \\
& + \sum_{i=1}^\ell N^2 \left[ {\bf f} \left( \frac{\bm x - \bm e_{j_1+\cdots+j_{i-1}+1} - \cdots - \bm e_{j_1+\cdots+j_i}}L \right) - {\bf f} \left( \frac{\bm x}L \right) \right] .
\end{aligned}
\end{equation}
Define $\bm\psi := \psi_{1,2}^{j_1-1} \circ \psi_{2,3}^{j_2-1} \circ \cdots \circ \psi_{\ell,\ell+1}^{j_\ell-1}$ (defined in \eqref{eq:psi-ij}), such that
\[
\bm\psi {\bf f} (u_{j_1}, u_{j_1+j_2}, \dots, u_k) = {\bf f} ( \overbrace{u_{j_1}, \dots , u_{j_1}}^{j_1}, \overbrace{u_{j_1+j_2}, \dots , u_{j_1+j_2}}^{j_2}, \dots, \overbrace{u_k, \dots, u_k}^{j_\ell}).
\]
By writing $\bm y := (x_{j_1},x_{j_1+j_2},\dots,x_k)$, the summations in \eqref{eq:sum-repr} are equal to
\[
\sum_{i=1}^\ell \, N^2 \left[ \bm\psi {\bf f} \left( \frac{\bm y + \bm e_i}{L} \right) + \bm\psi {\bf f} \left( \frac{\bm y - \bm e_i}{L} \right) - 2 \bm\psi {\bf f} \left( \frac{\bm y}{L} \right) \right].
\]
Since $\bm\psi {\bf f} \in \mathfrak D_0^{(\ell)}$ and $\frac{\bm y}L \in \mathbb T_\circ^\ell$, Taylor's theorem implies that the above term equals (cf. \eqref{eq:NL-rho})
\begin{equation}\label{eq:Taylor-3rd-used}
\frac{N^2}{L^2} \left( \Delta_\ell ( \bm\psi {\bf f}) \left( \frac{\bm y}L \right) + O \left( \frac{\| ( \bm\psi {\bf f})^{(3)} \|_\infty}{L} \right) \right) = \rho^2 \Delta_\ell ( \bm\psi {\bf f}) \left( \frac{\bm y}L \right) + o(1).
\end{equation}
Here, $\Delta_\ell$ denotes the $\ell$-dimensional Laplacian operator and $\| ( \bm\psi {\bf f} )^{(3)} \|_\infty$ stands for the supremum norm of all third derivatives of $\bm\psi {\bf f}$.
By Lemma \ref{lem:lap-lap-equal}, the final term equals $\mathfrak L_\rho {\bf f} ( \frac {\bm x}L ) + o(1)$. This finishes the proof.
\end{proof}

As in Section \ref{sec3.3}, for $\widehat\xi_{\bm n}^{\bm x} \in \widehat{\mathcal K}_N^\ell$ we may rewrite $F_N(\widehat\xi_{\bm n}^{\bm x})$ as
\[
F_{N}(\widehat\xi_{\bm{n}}^{\bm{x}}) =
{\rm E}^{ \widehat{\mathbb Q}_{\widehat\xi_{\bm{n}}^{\bm{x}}}^{N}} \left[ \left( \int_{0}^{ H_{ \widehat{\mathcal J}_N^\ell \cup \widehat{\mathcal E}_N^{\ell-1} } }+\int_{H_{ \widehat{\mathcal J}_N^\ell \cup \widehat{\mathcal E}_N^{\ell-1} } }^{\infty} \right) e^{-\lambda t}\, G_N( \widehat\eta_N(t)) \,{\rm d}t \right].
\]
By the same idea as in \eqref{eq:SMP-0} and \eqref{eq:SMP}, the strong Markov property implies that
\begin{equation}\label{eq:SMP-label}
\left|F_{N}(\widehat\xi_{\bm{n}}^{\bm{x}})-{\rm E}^{\widehat{\mathbb Q}_{\widehat\xi_{\bm{n}}^{\bm{x}}}^{N}}\left[ F_N ( \widehat\eta_N (H)) \right] \right|
\le \frac{\| {\bf g} \|_\infty}{\lambda}\,{\rm E}^{\widehat{\mathbb Q}_{\widehat\xi_{\bm{n}}^{\bm{x}}}^{N}}\left[1-e^{-\lambda H} \right],
\end{equation}
where $H := H_{ \widehat{\mathcal J}_N^\ell \cup \widehat{\mathcal E}_N^{\ell-1} }$. In the inequality, we bounded $G_N$ by $\| {\bf g} \|_\infty$ due to \eqref{eq:GN-def}.
\begin{lem}
\label{lem:xinx-hit-label}It holds that
\[
\lim_{N\to\infty} \sup_{ \widehat\xi_{\bm n}^{\bm x} \in \widehat{\mathcal K}_N^\ell }
{\rm E}^{\widehat{\mathbb Q}_{\widehat\xi_{\bm{n}}^{\bm{x}}}^{N}}\left[1-e^{-\lambda H_{ \widehat{\mathcal J}_N^\ell \cup \widehat{\mathcal E}_N^{\ell-1} } } \right] = 0 .
\]
\end{lem}

\begin{proof}
By Lemma \ref{lem:KN-trace-rate} and \eqref{eq:dN-cond}, the jump rates from $\widehat\xi$ to $\widehat{\mathcal E}_N^{\ell-1}$ is bounded below by $N^2 (1+o(1))$ uniformly in $\widehat\xi \in \widehat{\mathcal K}_N^\ell$. Therefore, by the strong Markov property, starting from $\widehat\xi_{\bm n}^{\bm x}$, the hitting time $H_{\widehat{\mathcal J}_N^\ell \cup \widehat{\mathcal E}_N^{\ell-1}}$ is bounded above by an exponential random variable with rate $N^2(1+o(1))$. Thus,
\[
{\rm E}^{\widehat{\mathbb Q}_{\widehat\xi_{\bm{n}}^{\bm{x}}}^{N}}\left[1-e^{-\lambda H_{ \widehat{\mathcal J}_N^\ell \cup \widehat{\mathcal E}_N^{\ell-1} } } \right]
\le \lambda \sup_{\widehat\xi \in \widehat{\mathcal K}_N^\ell} {\rm E}^{\widehat{\mathbb Q}_{\widehat\xi}^{N}}\left[ H_{ \widehat{\mathcal J}_N^\ell \cup \widehat{\mathcal E}_N^{\ell-1} } \right]
\le \frac \lambda{N^2(1+o(1))} = o(1).
\]

\end{proof}

Now we prove that, starting from $\widehat\xi \in \widehat{\mathcal K}_N^\ell$, the positions
of the $\ell$ condensates do not move macroscopically until it enters
$\widehat{\mathcal J}_N^\ell \cup \widehat{\mathcal E}_N^{\ell-1}$.

\begin{lem}\label{lem:not-moving}
We have
\begin{equation}
\lim_{N\to\infty} \sup_{ \widehat\xi_{\bm n}^{\bm x} \in \widehat{\mathcal K}_N^\ell} \widehat{\mathbb Q}_{\widehat\xi_{\bm n}^{\bm x}}^{N}\left[ d \left( \frac{\bm{x}}{L} , \Phi_{N} \left( \widehat\eta_N \left( H_{ \widehat{\mathcal J}_N^\ell \cup \widehat{\mathcal E}_N^{\ell-1} } \right) \right) \right)>\frac{1}{\sqrt{N}} \right] = 0.\label{eq:not-move}
\end{equation}
\end{lem}

\begin{proof}
Along the trajectory from $\widehat{\mathcal K}_N^\ell$
to $\widehat{\mathcal J}_N^\ell \cup \widehat{\mathcal E}_N^{\ell-1}$,
to observe a movement of a condensate, a jump of type A in Definition \ref{def:lab-trace} must occur.
By Lemma \ref{lem:JN-trace-rate}, this has rate at most $N^2(1+o(1))$ where the error is uniform.
Therefore, the number of such jumps before hitting $\widehat{\mathcal J}_N^\ell \cup \widehat{\mathcal E}_N^{\ell-1}$ is smaller than $\frac L{\sqrt N} \simeq \rho^{-1} \sqrt N$ with high probability, and each jump has rate at least $N^2(1+o(1))$ with a uniform error. This proves the desired result.

\end{proof}

\begin{rem}
The same argument yields that $\Phi_{N} ( \widehat\eta_N (t))$
remains close to $\Phi_{N} ( \widehat\eta_N (0))$ until time
$H_{ \widehat{\mathcal J}_N^\ell \cup \widehat{\mathcal E}_N^{\ell-1}
}$:
\begin{equation*}
\lim_{N\to\infty}
\sup_{ \widehat\xi_{\bm n}^{\bm x} \in \widehat{\mathcal K}_N^\ell}
\widehat{\mathbb Q}_{\widehat\xi_{\bm n}^{\bm x}}^{N}
\left[ \, \sup_{t \in \left[ 0, H_{ \widehat{\mathcal J}_N^\ell \cup \widehat{\mathcal E}_N^{\ell-1} } \right] }
d \left( \frac{\bm{x}}{L} , \Phi_{N} \left( \widehat\eta_N (t)\right)
\right)
>\frac{1}{\sqrt{N}} \right] = 0.
\end{equation*}
In addition, Lemma \ref{lem:not-moving} remains valid if the
scale $\frac{1}{\sqrt{N}}$ is replaced by any other scale
$a_{N}\gg\frac{1}{N}$.
\end{rem}

\begin{proof}[Proof of Lemma \ref{lem:res-ind-hyp-label}]
By Lemma \ref{lem1} for $\ell=1$, the initial case holds since $\widehat{\mathcal E}_N^1 = \widehat{\mathcal J}_N^1$.

Fix $\ell \ge 2$ and assume that the lemma holds for $k \in \llbracket 1,\ell-1 \rrbracket$.
By Lemma \ref{lem1}, equation \eqref{eq:key-hat-1} holds. In addition, by \eqref{eq:SMP-label} and Lemma \ref{lem:xinx-hit-label},
\[
\limsup_{N\to\infty} \sup_{ \widehat\xi_{\bm n}^{\bm x} \in \widehat{\mathcal K}_N^\ell} \left| \varepsilon_N ( \widehat\xi_{\bm n}^{\bm x})\right| \le
\limsup_{N\to\infty} \sup_{ \widehat\xi_{\bm n}^{\bm x} \in \widehat{\mathcal K}_N^\ell} {\rm E}^{ \widehat{\mathbb Q}_{ \widehat\xi_{\bm{n}}^{\bm{x}}}^{N}}\left[ \left| F_{N}\left( \widehat\eta_N \left( H_{ \widehat{\mathcal J}_N^\ell \cup \widehat{\mathcal E}_N^{\ell-1} } \right) \right) - {\bf f} \left( \frac {\bm x}L \right) \right| \right] .
\]
By Lemma \ref{lem:not-moving}, and the fact that $\bf f$ is uniformly Lipschitz, the right-hand side equals
\[
\limsup_{N\to\infty} \sup_{ \widehat\xi_{\bm n}^{\bm x} \in \widehat{\mathcal K}_N^\ell}
{\rm E}^{ \widehat{\mathbb Q}_{ \widehat\xi_{\bm{n}}^{\bm{x}}}^{N}}\left[ \left| F_{N}\left( \widehat\eta_N \left( H_{ \widehat{\mathcal J}_N^\ell \cup \widehat{\mathcal E}_N^{\ell-1} } \right) \right) - ( {\bf f} \circ \Phi_{N} ) \left( \widehat\eta_N \left( H_{ \widehat{\mathcal J}_N^\ell \cup \widehat{\mathcal E}_N^{\ell-1} } \right) \right) \right| \right] .
\]
Divide the above expectation subject to $\{ H_{\widehat{\mathcal E}_N^{\ell-1}} < H_{\widehat{\mathcal J}_N^\ell} \}$ and $\{ H_{\widehat{\mathcal E}_N^{\ell-1}} > H_{\widehat{\mathcal J}_N^\ell} \}$.
The first part becomes bounded by
\[
{\rm E}^{ \widehat{\mathbb Q}_{ \widehat\xi_{\bm{n}}^{\bm{x}}}^{N}}\left[ \left| F_{N}\left( \widehat\eta_N \left( H_{ \widehat{\mathcal E}_N^{\ell-1} } \right) \right) - ( {\bf f} \circ \Phi_{N} ) \left( \widehat\eta_N \left( H_{ \widehat{\mathcal E}_N^{\ell-1} } \right) \right) \right| \right],
\]
which is uniformly negligible by the induction hypothesis. By Lemma \ref{lem:escape-p-trace} and \eqref{eq:proj-law},
the second part is bounded above by
\[
(1-p) \limsup_{N \to \infty} \sup_{\widehat\xi \in \widehat{\mathcal E}_N^\ell} \left| \varepsilon_N (\widehat\xi) \right| .
\]
Combining all displays concludes the proof of Lemma \ref{lem:res-ind-hyp-label} in the case of general $\ell \ge 2$.
\end{proof}

\section{\label{sec6}Proof of Theorem \ref{thm4}}

We claim that Theorem \ref{thm4} follows from Theorem \ref{thm:res}.
For $\widehat\xi_{\bm n}^{\bm x} \in \widehat{\mathcal E}_N$ and $\gamma>0$, define
\[
\breve{\mathcal{E}}_{\bm{n}}^{\bm{x}}(\gamma) :=
\left\{ \widehat\xi_{\bm{m}}^{\bm{y}} \in \widehat{\mathcal E}_{N}: d\left( \frac{\bm{x}}{L} , \frac{\bm{y}}{L} \right) \ge \gamma \right\} .
\]

\begin{lem}
\label{lem4.1}For any $\gamma>0$,
\[
\lim_{t\to0}\lim_{N\to\infty} \sup_{ \widehat\xi_{\bm n}^{\bm x} \in \widehat{\mathcal E}_N } \widehat{\mathbb Q}_{\widehat\xi_{\bm{n}}^{\bm{x}}}^{N} \left[H_{\breve{\mathcal{E}}_{\bm{n}}^{\bm{x}}(\gamma)}<t\right]=0.
\]
\end{lem}

\begin{proof}
We may take a function ${\bf f}\in\mathfrak{D}^{(k)}$ such that ${\bf f}(\frac{\bm{x}}{L})=0$
and ${\bf f}( \bm u )=1$ for all $\bm u \in \mathbb T^k$
such that $d( \bm{u} , \frac{\bm x}L ) \ge \gamma$,
which is possible via Lemma \ref{lem:test-fcn} (by subtracting that
function from $1$). At this point the arguments presented in the proof of
\cite[Lemma 4.2]{LMS25} works here as well. We omit tedious repetition
of details.
\end{proof}
\begin{lem}
\label{lem4.2}For any $\widehat\xi \in \widehat{\mathcal E}_N$, $\epsilon>0$, and $T>0$,
\[
\lim_{a\to0}\lim_{N\to\infty}\sup_{\tau\in\mathscr{T}_{T}}\sup_{\delta\in(0,a)} \widehat{\mathbb Q}_{\widehat\xi}^{N}\left[ d \left( \Phi_N (\widehat\eta_N(\tau+\delta)),\Phi_N (\widehat\eta_N(\tau)) \right) \ge\epsilon\right]=0,
\]
where $\mathscr T_T$ is the collection of finite-range stopping times bounded by $T$.
\end{lem}

\begin{proof}
The main ingredient for the proof is Lemma \ref{lem4.1}.
We refer to \cite[Proposition 4.4]{LMS23}.
\end{proof}
\begin{lem}
\label{lem4.3}Under the assumptions in Theorem \ref{thm4}, any
limit point of $\{ \widehat{\mathbb Q}_{\widehat\xi_{\bm{n}}^{\bm{x}}}^{N} \circ \Phi_N^{-1} \}_{N \ge 1}$
solves the martingale problem in Theorem \ref{thm3} and is concentrated on continuous paths.
\end{lem}

\begin{proof}
We follow the ideas in the proof of \cite[Proposition 4.4]{LMS23}.
Lemma \ref{lem4.1} implies that any
limit point ${\bf Q}$ of $\{ \widehat{\mathbb Q}_{\widehat\xi_{\bm{n}}^{\bm{x}}}^{N} \circ \Phi_N^{-1} \}_{N \ge 1}$
satisfies ${\bf Q}\,[ \omega_0 = \{\bm{u}\}]=1$ and
\begin{equation}
{\bf Q}\,[ \omega_{\delta-} \ne \omega_\delta ]=0\qquad\text{for all}\quad\delta>0.\label{eq:discont-small}
\end{equation}
Moreover, using the ideas from \cite[Proposition 4.5]{LMS23}, we
can prove that
\[
{\bf f}( \omega_t)-{\bf f}( \omega_s)+\int_{s}^{t}e^{-\lambda u}\,(\lambda-\mathfrak L_\rho) {\bf f}( \omega_u)\,{\rm d}u
\]
is a ${\bf Q}$-martingale for any ${\bf f}\in\mathfrak{D}^{(k)}$.
Here, Theorem \ref{thm:res}
allows us to replace the macroscopic resolvent solution ${\bf f}$
with the microscopic resolvent solution $F_{N}$, and \eqref{eq:discont-small}
allows us to disregard the path discontinuities in the integral. Therefore,
${\bf Q}$ solves the martingale problem presented in Theorem \ref{thm3} and $\bf Q$ is concentrated on continuous paths, as desired.
\end{proof}

\begin{proof}[Proof of Theorem \ref{thm4}]
By Aldous' tightness criterion, Lemma \ref{lem4.2} implies that the collection $\{ \widehat{\mathbb Q}_{\widehat\xi_{\bm{n}}^{\bm{x}}}^{N} \circ \Phi_N^{-1} \}_{N \ge 1}$ is tight. In addition, Lemma \ref{lem4.3} indicates that any limit point must be the unique solution to the martingale problem in Theorem \ref{thm3}. These two observations complete the proof.
\end{proof}

\begin{acknowledgement*}
SK and CL would like to thank Instituto Superior T\'ecnico (Lisbon)
for their warm hospitality during their stay in February 2026, during
which part of the manuscript was written. SK has been supported by the
Basic Science Research Program through the National Research
Foundation of Korea funded by the Ministry of Science and ICT
(RS-2025-00518980), the Yonsei University Research Fund of 2025
(2025-22-0133), and the POSCO Science Fellowship of POSCO TJ Park
Foundation. CL has been partially supported by FAPERJ CNE
E-26/201.117/2021, and by CNPq Bolsa de Produtividade em Pesquisa PQ
305779/2022-2.
\end{acknowledgement*}

\appendix

\section{\label{secA}Coalescing Brownian Motions}

In this section,  we present a few facts regarding the
$k$-coalescing Brownian motions that was introduced
in Section \ref{sec1.5}.

We start with a proof of Theorem \ref{thm3}, which is based on
Theorems 2.1 and 2.2 in \cite{DLZ04}. Recall that a continuous
$\bb T$-valued process $X(\cdot)$ is called a $\bb T$-Brownian motion
with speed $a>0$ if there exists a Brownian motion $B(\cdot)$ with
speed $a$ such that $X (t) = \Pi (B(t)))$ for all $t\ge 0$. Here
$\Pi\colon \bb R\to \bb T$ is the usual projection defined by
$\cb{\Pi(x) } = x - [x]$, where $[r]$ stands for the integer value of
a real number $r$. The process $B(\cdot)$ is called the lifting of
$X(\cdot)$ to $\bb R$. We may assume, without loss of generality,
that $B(0) \in [0,1)$, and this determines uniquely the lifting.

Fix $k\ge 2$, $a>0$. We say that a continuous random process
$(X_1(t), \dots, X_k(t))$ taking values in $\bb T^k$ is a $k$-system
of coalescing Brownian motions with speed $a$ on the torus if

\begin{enumerate}
\item[(a)] Each coordinate $X_i(t)$ is a $\bb T$-Brownian motion with
speed $a$. Denote by $B_i(\cdot)$ the lifting of $X_i(\cdot)$ to
$\bb R$.

\item[(b)] For each $1\le i<j\le k$, $|B_i(t) - B_j(t)|$ is a Brownian
motion with speed $2a$ absorbed at $\bb Z$. 
\end{enumerate}

Of course as $B_\ell(0) \in [0,1)$, $1\le \ell\le k$, $|B_i(t) - B_j(t)|$ is
actually absorbed at $0$ or $1$.

For each $1\le i<j\le k$, let $\cb{T_{i,j}} \colon \bb T^k \to [0,1]$
be the function defined by $T_{i,j}(x_1, \dots, x_k) = T(x_i,x_j)$,
where $T \colon \bb T^2\to [0,1]$ is given by
\begin{equation*}
\cb{T (x,y) } \, :=\, (y-x) \, \mb 1\{ x\le y\}\, +\, (1+y-x)\,  \mb
1\{ y< x \}\, \,. 
\end{equation*}
The function $T(x,y)$ can be understood as follows. First, map the
pair $(x,y)\in \bb T^2$ to $\bb R^2$ placing the second coordinate
ahead of the first: $T^{(1)}(x,y) = (x,\hat y)$, where $\hat y = y$ if
$x\le y$, and $\hat y = 1+y$ if $y<x$. Then, compute $\hat y - x$.

Denote by $\tau_{i,j}$, $\sigma_{i,j}$ the hitting times defined by
\begin{equation}
\label{a04}
\cb{\tau_{i,j} } \,:=\, \inf\big\{ \, t\ge 0 : |B_i(t) - B_j(t)| \in \bb Z\,
\big\}\,, \quad
\cb{\sigma_{i,j} } \,:=\, \inf\big\{ \, t\ge 0 : T(X_i(t), X_j(t)) \in
\{0,1\} \,
\big\}\,.
\end{equation}
It is easily seen that $\tau_{i,j}  = \sigma_{i,j}$ and that
\begin{equation}
\label{a03}
T(X_i(t), X_j(t)) \, =\,  |B_j(t) - B_i(t)| \, =\,  B_j(t) - B_i(t)
\quad \text{for all}\;\;
0\le t\le \tau_{i,j} \;\;\text{if}\;\; x_i \le  x_j \,.
\end{equation}
In the case where $x_j<x_i$,
$T(X_i(t), X_j(t)) = |1 + B_j(t) - B_i(t)|$ for all
$0\le t\le \tau_{i,j} $.

\begin{lem}
\label{al01}
Fix $k\ge 2$, $a>0$.  Suppose that $(X_1(t), \dots, X_k(t))$ is a
$k$-system of coalescing Brownian motions with speed $a$ on the torus.
Then, $\< B_i \,, B_j\>_t = a\, (t - t \wedge \tau_{i,j})$. Conversely
if $X_i(\cdot)$, $1\le i\le k$, are Brownian motions with speed $a$ on
$\bb T$, and $\< B_i \,, B_j\>_t = a\, (t - t \wedge \tau_{i,j})$,
then $(X_1(t), \dots, X_k(t))$ is a $k$-system of coalescing Brownian
motions with speed $a$ on the torus.
\end{lem}

\begin{proof}
By property (b) of the definition,
$[B_j(t) - B_i(t)]^2 - 2\, a\, (t\wedge \tau_{i,j})$ is a
martingale. As $B_i(\cdot)$ and $B_j(\cdot)$ are Brownian motions with
speed $a$, $B_i(t)\, B_j(t) - a\, (t - t\wedge \tau_{i,j})$ is a
martingale, which proves the first claim of the lemma.

Conversely, as the lifts $B_i(\cdot)$ and $B_j(\cdot)$ are Brownian
motions, $B_j-B_i$ is a continuous martingale. On the other hand, by
hypothesis, $\< B_i - B_j\>_t = 2\, a\, (t \wedge \tau_{i,j})$. By L\'evy's
characterization of the Brownian motion, $B_i - B_j$ is a Brownian
motion of speed $2a$ stopped at $\tau_{i,j}$. This implies that
$|B_i(t) - B_j(t)|$ is a Brownian motion with speed $2a$ absorbed at
$\bb Z$ (actually at $0$ or $1$).
\end{proof}

We turn to the proof of Theorem \ref{thm3} for $k=2$.

\begin{prop}
\label{al02}
Fix $(x,y)\in \bb T^2$.  Suppose that a probability measure
$\mb Q_{(x,y)}$ on $C(\bb R_+, \bb T^2)$ solves the martingale problem
formulated in Theorem \ref{thm3}, and that
$\mb Q_{(x,y)} [\, \omega : \omega_0 = (x,y)\,] = 1$. Then,
$\mb Q_{(x,y)}$ is the measure induced by a $2$-system of coalescing
Brownian motions with speed $2\rho^2$ on the torus starting from
$(x,y)$.
\end{prop}

\begin{proof}
In the statement of Theorem \ref{thm3}, sending $\lambda\to 0$ yields
that
\begin{equation*}
f (\omega (t) ) \,-\, \int_{0}^{t} (\mathfrak L_\rho f )  (\omega (s) )\,{\rm d}s 
\end{equation*}
is a martingale for every $f\in \mf D^{(2)}$.

Fix a function $f$ in $C^3(\bb T)$, and let $F\colon \bb T^2 \to \bb R$ be
given by $F(x,y) = f(x)$. It is easy to check that $F$ belongs to $\mf
D^{(2)}$. Thus,
\begin{equation*}
f (\omega_1(t)) \,-\, f (\omega_1(0))
\,-\,  \int_{0}^{t} \rho^2\, f'' (\omega_1(s) )\,{\rm d}s 
\end{equation*}
is a martingale. Approximating a function in $C^2(\bb T)$ by functions
in $C^3(\bb T)$ yields that the previous expression is a martingale
for all functions $f$ in $C^2(\bb T)$. This implies that
$\omega_1 (\cdot)$ is a Brownian motion with speed $2\rho^2$ on
$\bb T$.  The same assertion holds for $\omega_2 (\cdot)$ as well.

Fix a function $f\in C^3([0,1])$ such that $f(0) = f(1)$,
$f'(0) = f'(1)$, $f''(0) = f''(1) =0$, $f'''(0)=f'''(1)$, and let
$F\colon \bb T^2 \to \bb R$ be given by $F(x,y) = f(T(x,y))$. Here
again, it is easy to check that $F$ belongs to $\mf D^{(2)}$. The
condition $f''(0)=0$ is needed to ensure that
$\partial_x \partial_y F =\partial_y \partial_x F =0$ at the boundary
$x=y$.  The other conditions ensure that $F$ and its partial
derivatives can be extended continuously to the boundary (as $y$ and
$x$ get closer $T(x,y)$ may converge to $0$ or to $1$. This forces $f$
and its first three derivatives to have the same value at $0$ and
$1$. It has nothing to do with the fact that the process evolves on a
torus).  Thus,
\begin{equation}
\label{a01}
f (T(\omega(t)) ) \,-\, f (T(\omega(0)))
\,-\,  \int_{0}^{t} 2\, \rho^2\, f'' (T(\omega(s) )) \,{\rm d}s 
\end{equation}
is a martingale.

Recall that $T(\cdot)$ takes value in $[0,1]$. Fix $0<\delta<1/4$, and
let $Z_t = T(\omega(t))$. Define the sequence of stopping times
$0\le \sigma_1< \tau_1 < \sigma_2 < \cdots$ by
\begin{gather*}
\sigma_1 = \inf \big\{\, t\ge 0 : Z_t \le 1 - 2\delta\,
\big\}\,, \quad
\tau_1 = \inf \big\{\, t>\sigma_1 : Z_t \ge 1 - \delta
\, \big\}\,,
\\
\sigma_{p+1} = \inf \big\{\, t> \tau_p : Z_t \le 1 - 2\delta\,
\big\}\,, \quad
\tau_{p+1} = \inf \big\{\, t>\sigma_{p+1} : Z_t \ge 1 - \delta
\, \big\}\,, \quad p\ge 1\,.
\end{gather*}
Fix a function $f$ in $C^3([0,1])$ satisfying the boundary conditions,
and such that $f(x) = x$ for $x\in [0 , 1-\delta]$.  By \eqref{a01},
in the time intervals $[\sigma_p, \tau_p]$, $Z_t = T(\omega(t))$ is a
martingale.

Similarly, define the sequence of stopping times
$0\le \sigma^*_1< \tau^*_1 < \sigma^*_2 < \cdots$ by
\begin{gather*}
\sigma^*_1 = \inf \big\{\, t\ge 0 : Z_t \ge 2\delta\,
\big\}\,, \quad
\tau^*_1 = \inf \big\{\, t>\sigma_1 : Z_t \le \delta
\, \big\}\,,
\\
\sigma^*_{p+1} = \inf \big\{\, t> \tau^*_p : Z_t \ge 2\delta\,
\big\}\,, \quad
\tau^*_{p+1} = \inf \big\{\, t>\sigma^*_{p+1} : Z_t \le \delta
\, \big\}\,, \quad p\ge 1\,.
\end{gather*}
Consider a test function $g$ in $C^3([0,1])$, satisfying the boundary
conditions, and such that $g(x) = x$ for $x\in [\delta,1]$.  By
\eqref{a01}, in the time intervals $[\sigma^*_p, \tau^*_p]$,
$Z_t = T(\omega(t))$ is a martingale.

\begin{claim*}
The process $Z_t = T(\omega(t))$, $t\ge 0$, is a martingale.
\end{claim*}

\begin{proof}[Proof of Claim]
Fix $(x,y)\in \bb T^2$ and suppose that $y=x$. By the previous
arguments, $Z_t = T(\omega(t))$ is a continuous martingale on the
interval $[0, \tau_1)$. As it is continuous, $\tau_1>0$. As it is
non-negative and starts from $0$, $Z_t=0$ for $t\in [0,
\tau_1)$. Hence, $\tau_1=+\infty$ and $Z_t=0$ for $t\ge 0$.

Fix $(x,y)\in \bb T^2$ such that $x\neq y$. Choose $\delta$ small
enough for $2\delta< T(x,y) < 1-2\delta$.  As the process
$T(\omega (\cdot))$ is continuous, $\sigma_p$, $\tau_p$, $\sigma^*_p$
and $\tau^*_p$ converge a.s. to infinity as $p\to\infty$.

By definition $\sigma_1=\sigma^*_1=0$.  The previous argument asserts
that $T(\omega (\cdot))$ is a martingale in the time interval
$[\sigma_1, \tau_1] =[0, \tau_1]$. If $\tau_1= +\infty$, the claim is
proved. Suppose that $\tau_1 <\infty$.  Since $\sigma^*_1=0$, and
$\sigma^*_p$ and $\tau^*_p$ diverge to $\infty$, by definition of
these stopping times, there exists $p_1$ such that
$\tau_1 \in (\sigma^*_{p_1}, \tau^*_{p_1})$.  As $T(\omega (\cdot))$
is a martingale in this interval and in the interval $[0, \tau_1]$, we
conclude that it is a martingale in the time interval
$[0, \tau^*_{p_1}]$.

We may iterate the argument.  The stopping time $\tau^*_{p_1}$ is
either $+\infty$ or belongs to some interval
$(\sigma_{q_1}, \tau_{q_1})$, and $T(\omega (\cdot))$ is a martingale
in the time interval $[0, \tau_{q_1}]$.  As
$\tau_{q_1} > \tau^*_{p_1} > \tau_1$, $q_1>1$.  Proceeding in this
way, as $q_j\to\infty$ we conclude that $T(\omega (\cdot))$ is a
martingale in $\bb R_+$, proving the claim.
\end{proof}

Since the martingale $Z_t$, $t \ge0$ is
bounded by $0$ and $1$, it is absorbed at $0$ and $1$.
By It\^o's formula,
\begin{equation}
\label{a02}
T(\omega(t))^2 \,-\, T(\omega(0))^2 \, -\, \< T(\omega)\>_t 
\end{equation}
is a martingale.

Fix $0<\delta<1/2$, and let $f_\delta$ be a $C^3$-function satisfying the
boundary conditions such that $f_\delta (x) = x^2$ for $\delta \le
x\le 1-\delta$. By \eqref{a01}, if $\tau_\delta = \inf\{t > 0 :
T(\omega(t)) \not\in [\delta, 1-\delta]\, \}$,
\begin{equation*}
T(\omega(t \wedge \tau_\delta))^2 \,-\, T(\omega(0))^2
\, -\, 4\, \rho^2 \, (t \wedge \tau_\delta )
\end{equation*}
is a martingale. As $\delta\to 0$,
$\tau_\delta\to \tau_0 =\inf \{ t : T(\omega (t)) \in \{0,1\}\}$, so
that
$T(\omega(t \wedge \tau_0))^2 \,-\, T(\omega(0))^2 \, -\, 4\, \rho^2
\, (t \wedge \tau_0)$ is a martingale because $T(\cdot)$ is
bounded. Since $T(\omega (\cdot)) $ is absorbed at $0$ and $1$,
$T(\omega(t)) = T(\omega(t \wedge \tau_0))$ so that
\begin{equation*}
T(\omega(t))^2 \,-\, T(\omega(0))^2
\, -\, 4\, \rho^2 \, (t \wedge \tau_0 )
\end{equation*} 
is a martingale. Combining this with \eqref{a02} yields that $\< T(\omega)\>_t  = 4\, \rho^2
\, (t \wedge \tau_0)$. 

Let $B_j(\cdot)$, $j=1$, $2$, be the lifting to $\bb R$ of
$\omega_j(\cdot)$, as described at the beginning of this section.
Recall the definition of the stopping times $\tau_{1,2}$,
$\sigma_{1,2}$ introduced at \eqref{a04}, as well as the
identity $\tau_{1,2} = \sigma_{1,2}$.  In this proof, $\sigma_{1,2}$
has been represented by $\tau_0$ so that $\tau_0 = \tau_{1,2} = \sigma_{1,2}$.

Assume, without loss of generality, that
$\omega_1(0) \le \omega_2(0)$.  Since $T(\omega(\cdot))$ is absorbed
at $\tau_0$, by \eqref{a03}, $T(\omega(t)) =|B_2(t) - B_1(t)|$ for all
$t\ge 0$. Hence $\< B_2 - B_1 \>_t = 4\, \rho^2 \, (t \wedge \tau_0)$.
As $B_1$, $B_2$ are Brownian motions,
$\< B_1 \,, B_2 \>_t = \< T(\omega)\>_t = 4\, \rho^2 \, [\, t - (t
\wedge \tau_0)\,]$.  To complete the proof of the proposition, it
remains recalling  that $B_i(\cdot)$ are Brownian motions of speed
$2\rho^2$ and applying Lemma \ref{al01}.
\end{proof}

\begin{proof}[Proof of Theorem \ref{thm3}]
Fix $k\ge 3$. Let $f$ be a function in $\mf D^{(2)}$. It is easily
seen that the map $(x_1, \dots, x_k) \mapsto f(x_i,x_j)$ belongs to
$\mf D^{(k)}$ for any $1\le i<j\le k$. Hence, if ${\bf Q}_{\bs x}$ is a
solution of the martingale problem formulated in Theorem \ref{thm3},
the pair $(\omega_i (t), \omega_j(t))$ is a solution of the same
martingale problem with $k=2$. By the previous proposition,
$(\omega_i (t), \omega_j(t))$ is a $2$-system of coalescing Brownian
motions with speed $2\rho^2$ on the torus. Thus, by Lemma \ref{al01},
$\< B_i \,, B_j\>_t = 2\,\rho^2\, ( t - t \wedge \tau_{i,j})$, where
$\tau_{i,j}$ has been introduced in \eqref{a04}. Therefore, by Lemma
\ref{al01}, $(\omega_1 (t), \dots, \omega_k(t))$ is a $k$-system of
coalescing Brownian motions with speed $2\rho^2$ on the torus.
\end{proof}

\begin{lem}\label{lem:lap-lap-equal}
Suppose that $\bm u \in \mathbb T^k$ has $\ell$ different elements, and satisfies
\[
u_1 = \cdots = u_{j_1}, \quad u_{j_1+1} = \cdots = u_{j_1+j_2}, \quad \cdots ,\quad u_{j_1+\cdots+j_{\ell-1}+1} = \cdots = u_{j_1 + \cdots + j_\ell},
\]
where $j_1 , \dots , j_\ell \ge 1$ and $j_1 + \cdots + j_\ell = k$. 
Define $\bm\psi := \psi_{1,2}^{j_1-1} \circ \psi_{2,3}^{j_2-1} \circ \cdots \circ \psi_{\ell,\ell+1}^{j_\ell-1}$ (cf. \eqref{eq:psi-ij}). 
Let $\bm v := (u_{j_1},u_{j_1+j_2},\dots,u_k)$. Then for any ${\bf f} \in \mathfrak D^{(k)}$,
\[
\Delta {\bf f} ( \bm u) = \Delta_\ell ( \bm\psi {\bf f} ) (\bm v).
\]
\end{lem}

\begin{proof}
Notice that, for each $i \in \llbracket 1,\ell \rrbracket$,
\[
\partial_i (\bm\psi {\bf f}) (\bm v)
= \lim_{\epsilon\to0} \frac{ \bm\psi {\bf f} (\bm v + \epsilon \bm e_i) - \bm\psi {\bf f} (\bm v)}{\epsilon}
= \lim_{\epsilon\to0} \frac{{\bf f} \left(\bm{u}+\epsilon\sum_{n\in I_i} \bm e_{n}\right) - {\bf f} (\bm u) } {\epsilon} = \sum_{n\in I_i}\partial_{n} {\bf f} (\bm{u}),
\]
where $I_i := \llbracket j_1+\cdots+ j_{i-1}+1 , j_1 +\cdots + j_i \rrbracket$. The last
equality is due to the fact that all partial derivatives are continuous
up to the boundary $\partial\mathbb{T}_{\circ}^{k}$. Going one step
further,
\[
\partial_i^2 (\bm\psi {\bf f}) (\bm v) = \sum_{n,m\in I_i} \partial_m \partial_n {\bf f} (\bm u)
= \sum_{n\in I_i} \partial_n^2 \, {\bf f} (\bm u),
\]
where the second equality holds by \eqref{eq:didj}. Adding this up for all $i \in \llbracket1,\ell\rrbracket$ yields
\[
\Delta_{\ell} (\bm\psi {\bf f}) (\bm v) = \sum_{i=1}^{\ell} \sum_{n\in I_i} \partial_n^2 \, {\bf f} (\bm u)
= \Delta {\bf f} (\bm u),
\]
which concludes the proof of the lemma.
\end{proof}

\begin{lem}
\label{lem:test-fcn}For any $\bm u \in \mathbb T^k$
and $\gamma>0$, there exists ${\bf f}\in\mathfrak{D}^{(k)}$ such
that
\[
{\bf f}( \bm u )=1,\qquad \text{and} \qquad {\bf f} ( \bm{u'} ) = 0\qquad\text{for all}\quad \bm{u'} \in \mathbb T^k \qquad \text{with} \quad
d \, ( \bm u , \bm{u'} ) \ge \gamma.
\]
\end{lem}

\begin{proof}
Suppose that $\bm u$ has $\ell \in \llbracket 1,k \rrbracket$ distinct elements and write $\{\bm{u}\}=\{u_{1},\dots,u_{\ell}\}$. Define
\[
\mathscr{U}:=\left\{ \bm{v}\in\mathbb{T}^{k}:\{\bm{v}\}=\{u_{1},\dots,u_{\ell}\}\right\} .
\]
Take a small number $\epsilon\in(0,\gamma)$ such that $\epsilon\le\frac{1}{3} d \, (u_{i},u_{i'})$
for all $i\ne i'$. Then, the closed boxes
\[
B_{\epsilon}[\bm{v}]=\left\{ \bm{w}\in\mathbb{T}^{k}:d \, (\bm{v},\bm{w})\le\epsilon\right\} ,\qquad\bm{v}\in\bm{\mathscr U}
\]
are mutually disjoint. 

Fix a smooth bump function $g_{\epsilon}:\mathbb{R}\to[0,1]$ such
that $g_{\epsilon}(0)=1$ and $g_{\epsilon}(x)=0$ if $|x|\ge\epsilon$.
Then, define ${\bf f}:\mathbb{T}^{k}\to\mathbb{R}$ as follows. First,
${\bf f}\equiv0$ in $\mathbb{T}^{k}\setminus\bigcup_{\bm{v}\in\mathscr{U}}B_{\epsilon}[\bm{v}]$.
Next, on each $B_{\epsilon}[\bm{v}]$, suppose that $\bm{v}=(v_{1},\dots,v_{k})$
where $v_{i}=u_{a(i)}$ for $i\in\llbracket1,k\rrbracket$. Write
$n_{j}:=|\{i\in\llbracket1,k\rrbracket:a(i)=j\}|\ge1$. Then, define
\begin{equation}
{\bf f} (\bm{w}):=\frac{1}{\ell}\sum_{j=1}^{\ell}\frac{1}{n_{j}}\sum_{ i\in\llbracket1,k\rrbracket : \, a(i)=j}
g_{\epsilon}(w_{i}-v_{i})\qquad\text{for}\quad\bm{w}\in B_{\epsilon}[\bm{v}].\label{eq:h-def}
\end{equation}
Notice that ${\bf f}$ is smooth since $g_{\epsilon}$ is smooth. In
addition, $\partial_{i} \partial_{j} {\bf f}=0$ for any $i\ne j$.
Since $\bm u \in B_\epsilon[\bm u]$ we may calculate as
\[
{\bf f} (\bm u) = \frac 1\ell \sum_{j=1}^\ell \frac1{n_j} \sum_{ i\in\llbracket1,k\rrbracket : \, a(i)=j} g_\epsilon (0) = 1,
\]
and clearly if $d ( \bm u , \bm{u'} ) \ge \gamma > \epsilon$ then $\bm{u'} \notin B_\epsilon[\bm v]$ for any $\bm v \in \mathscr U$, thus ${\bf f}(\bm{u'}) = 0$.
These facts verify that ${\bf f}\in \mathfrak D^{(k)}$ as desired.
\end{proof}

\section{\label{secB}Estimates for Random Walks}

Here, we record some simple estimates that are necessary in Section \ref{sec2}.

\subsection{\label{secB.1}First Case}

Fix an integer $M\ge2$ and consider a continuous-time random walk
$(x_{t})_{t\ge0}$ on $\llbracket0,M\rrbracket=\{0,1,\dots,M\}$.
The jump rates are given as
\[
r(i,i+1)=a_{i}:= \theta (M-i)(i+d),\quad r(i,i-1)=b_{i}:= \theta i(M-i+d)\quad\text{for} \enspace i\in\llbracket1,M-1\rrbracket,
\]
where $\theta,d>0$ are constants. This random walk gets absorbed in $\{0,M\}$ almost surely. Let us denote by $H = H_{\{0,M\}}$ this
absorption time. Denote by $P_{i}$ the law of this random walk starting
from $i \in \llbracket 0,M \rrbracket$.
\begin{lem}
\label{lem:B1}We have
\[
\frac{1}{M \, e^{d(1+\log M)}} \le P_{1}[x_H = M] \le \frac{e^{d(1+\log M)}}{M} , \qquad {\rm E}^{P_1}[H]\le e^{d(1+\log M)} \, \frac{2(1+\log M)}{\theta M}.
\]
\end{lem}

\begin{proof}
It is elementary to check that, for each $i \in \llbracket 0,M \rrbracket$,
\begin{equation}
P_{i}[x_H = M]=\frac{\sum_{j=1}^{i}\prod_{m=1}^{j-1}\frac{b_{m}}{a_{m}}}{\sum_{j=1}^{M}\prod_{m=1}^{j-1}\frac{b_{m}}{a_{m}}}\label{eq:prob}
\end{equation}
and
\begin{equation}
{\rm E}^{P_i} [H] = \sum_{j=1}^{i}\sum_{n=j}^{M-1}\frac{1}{b_{n}}\prod_{m=j}^{n-1}\frac{a_{m}}{b_{m}}-\frac{\left(\sum_{j=1}^{i}\prod_{m=1}^{j-1}\frac{b_{m}}{a_{m}}\right)\left(\sum_{j=1}^{M}\sum_{n=j}^{M-1}\frac{1}{b_{n}}\prod_{m=j}^{n-1}\frac{a_{m}}{b_{m}}\right)}{\sum_{j=1}^{M}\prod_{m=1}^{j-1}\frac{b_{m}}{a_{m}}}.\label{eq:expect}
\end{equation}
Note that
\begin{equation}
\frac{b_{m}}{a_{m}}=\frac{m(M-m+d)}{(M-m)(m+d)}=1+\frac{(2m-M)d}{(M-m)(m+d)}\le1+\frac{d}{M-m}\le e^{\frac{d}{M-m}},\label{eq:ambm-1}
\end{equation}
and similarly,
\begin{equation}
\frac{a_{m}}{b_{m}}=\frac{(M-m)(m+d)}{m(M-m+d)} = 1 + \frac{(M-2m)d}{m(M-m+d)} \le 1 + \frac dm \le e^{\frac{d}{m}}.\label{eq:ambm-2}
\end{equation}
Thus, the first object in the lemma is estimated via \eqref{eq:ambm-1} as
\[
P_1 [x_H = M] \ge \frac{1}{\sum_{j=1}^{M}\prod_{m=1}^{j-1}e^{\frac{d}{M-m}}}
\ge \frac{1} {Me^{\sum_{m=1}^{M-1}\frac{d}{M-m}}} \ge \frac 1{M \, e^{d(1+\log M)}} ,
\]
and
\[
P_1 [x_H = M] \le \frac{1}{\sum_{j=1}^{M}\prod_{m=1}^{j-1}e^{-\frac{d}m}}
\le \frac{1} {Me^{- \sum_{m=1}^{M-1} \frac dm}} \le \frac {e^{d(1+\log M)}}{M} .
\]
By neglecting all minus terms, the second object is estimated via \eqref{eq:ambm-2} as
\[
{\rm E}^{P_1} [H] \le \sum_{n=1}^{M-1}\frac{1}{b_{n}}\prod_{m=1}^{n-1}\frac{a_{m}}{b_{m}}
\le \sum_{n=1}^{M-1}\frac{1}{ \theta n(M-n)}\prod_{m=1}^{M-1}e^{\frac{d}{m}}
\le e^{d(1+\log M)} \, \frac{2(1+\log M)}{\theta M}.
\]
This completes the proof.
\end{proof}
%\begin{rem}
%A more careful estimate gives an improved bound:
%\[
%\sup_{i\in\llbracket1,M-1\rrbracket}E_{i}[\tau]\le Ce^{3d(1+\log M)}.
%\]
%However, the results as presented are sufficient.
%\end{rem}

%We may further pinpoint the expectation given boundary hitting conditions.
%\textbf{TBD: maybe we do not need this anymore.}
%\begin{lem}
%\label{lem:B3}We have
%\[
%E_{1}[\tau\,|\,x_{\tau}=0]\le2e^{3d(1+\log M)}\,\frac{1+\log M}{M-1}\quad\text{and}\quad E_{1}[\tau\,|\,x_{\tau}=M]\le2e^{3d(1+\log M)}\,(1+\log M).
%\]
%\end{lem}
%
%\begin{proof}
%Using \eqref{eq:prob}, \eqref{eq:ambm-1}, and \eqref{eq:ambm-2},
%we may also prove that
%\[
%P_{1}[x_{\tau}=0]=\frac{\sum_{j=2}^{M}\prod_{m=1}^{j-1}\frac{b_{m}}{a_{m}}}{\sum_{j=1}^{M}\prod_{m=1}^{j-1}\frac{b_{m}}{a_{m}}}\ge\frac{M-1}{Me^{2d(1+\log M)}}.
%\]
%Thus, by the previous lemma,
%\[
%E_{1}[\tau\,|\,x_{\tau}=0]=\frac{E_{1}[\tau\,{\bf 1}\{x_{\tau}=0\}]}{P_{1}[x_{\tau}=0]}\le e^{3d(1+\log M)}\,\frac{2(1+\log M)}{M-1}
%\]
%and
%\[
%E_{1}[\tau\,|\,x_{\tau}=M]=\frac{E_{1}[\tau\,{\bf 1}\{x_{\tau}=M\}]}{P_{1}[x_{\tau}=M]}\le e^{3d(1+\log M)}\,2(1+\log M),
%\]
%as wanted.
%\end{proof}

\subsection{\label{secB.2}Second Case}

Next, we consider another random walk $(y_{t})_{t\ge0}$ on $\llbracket0,M\rrbracket$
which is defined via
\[
r'(i,i+1)=a_{i}':= \theta (M-i)(i+d),\quad r'(i,i-1)=b_{i}':= \theta i(M-i+2d)\quad\text{for}\enspace i\in\llbracket1,M-1\rrbracket.
\]
Denote by $P_{i}'$ the law starting from $i \in \llbracket 0,M \rrbracket$, and again by $H$ the hitting time of $\{0,M\}$.
\begin{lem}
\label{lem:B4}We have
\[
\frac{1}{M \, e^{2d(1+\log M)}} \le P_{1}' [x_H = M] \le \frac{e^{d(1+\log M)}}{M},
\qquad {\rm E}^{P_1'} [H] \le e^{d(1+\log M)} \, \frac{2(1+\log M)}{\theta M}.
\]
\end{lem}

\begin{proof}
This follows from
\[
\frac{b_{m}'}{a_{m}'}=\frac{m(M-m+2d)}{(M-m)(m+d)}=1+\frac{(3m-M)d}{(M-m)(m+d)}\le1+\frac{2d}{M-m}\le e^{\frac{2d}{M-m}},
\]
and
\[
\frac{a_{m}'}{b_{m}'}=\frac{(M-m)(m+d)}{m(M-m+2d)} = 1 + \frac{(M-2m)d}{m(M-m+2d)} \le 1 + \frac dm \le e^{\frac{d}{m}},
\]
as done in the proof of Lemma \ref{lem:B1}.
\end{proof}

%\textbf{TBD: maybe we do not need this anymore.}
%
%\begin{lem}
%\label{lem:B5}We have
%\[
%E_{1}'[\tau\,|\,x_{\tau}=0]\le2e^{4d(1+\log M)}\,\frac{1+\log M}{M-1}\quad\text{and}\quad E_{1}'[\tau\,|\,x_{\tau}=M]\le2e^{3d(1+\log M)}\,(1+\log M).
%\]
%\end{lem}
%
%\begin{proof}
%As done in the previous subsection, we can prove that
%\[
%P_{1}'[x_{\tau}=M]\ge\frac{1}{Me^{2d(1+\log M)}},\quad P_{1}'[x_{\tau}=0]\ge\frac{M-1}{Me^{3d(1+\log M)}},
%\]
%and
%\[
%E_{1}'[\tau]\le e^{d(1+\log M)}\frac{2(1+\log M)}{M}.
%\]
%These bounds prove the desired results.
%\end{proof}

\end{document}